\documentclass[12pt]{article}

\usepackage{xcolor}

\usepackage{amsmath}
\usepackage{amssymb}
\usepackage{amsfonts}
\usepackage{graphicx}

\usepackage{amsthm}

\newtheorem{prop}{Proposition}[section]
\newtheorem{theo}[prop]{Theorem}

\newtheorem{lemm}[prop]{Lemma}

\theoremstyle{definition}
\newtheorem{remk}[prop]{Remark}

\newcommand{\eqco}{\setcounter{equation}{0}}
\newcommand{\allco}{\eqco}
\newcommand{\lbl}{\label}

\newcommand{\Po}{{\cal P}}
\newcommand{\cN}{{\cal N}}
\newcommand{\M}{{\cal M}}
\newcommand{\tPo}{\tilde{{\cal P}}}
\newcommand{\tx}{{\tilde{x}}}
\newcommand{\tiu}{{\tilde{u}}}
\newcommand{\tiy}{{\tilde{y}}}

\newcommand{\tmu}{\tilde{{\mu}}}
\newcommand{\tv}{\tilde{{v}}}
\newcommand{\tw}{\tilde{{w}}}

\newcommand{\tgamma}{\tilde{{\gamma}}}
\newcommand{\ggamma}{\alpha}

\newcommand{\Q}{{\end{document}}}

\newcommand{\cM}{{\cal M}}
\newcommand{\m}{{\bf M}}
\newcommand{\tg}{\tilde{g}}

\newcommand{\sbeta}{\theta}

\newcommand{\nalpha}{u}

\newcommand{\N}{\mathbb{N}}

\newcommand{\bH}{\mathbb{H}}

\newcommand{\R}{\mathbb{R}}
\newcommand{\Z}{\mathbb{Z}}

\newcommand{\X}{\mathcal{X}}
\newcommand{\I}{\mathcal{I}}

\renewcommand{\Pr}{\mathbb{P}}
\renewcommand{\emptyset}{\varnothing}

\newcommand{\eqd}{\stackrel{{\cal D}}{=}}

\newcommand{\vvol}{\omega}  
\newcommand{\hH}{\hat{H}}

\newcommand{\lglg}{\log \log}

\newcommand{\LL}{{\cal L}}
\newcommand{\cH}{{\cal H}}
\newcommand{\card}{\#}

\newcommand{\diam}{{\rm diam}}

\newcommand{\Y}{{\cal Y}}
\newcommand{\cB}{{\cal B}}
\newcommand{\cA}{{\cal A}}

\newcommand{\cF}{{\cal F}}

\newcommand{\tR}{{\tilde{R}}}
\newcommand{\tD}{{\tilde{D}}}
\newcommand{\tJ}{{\tilde{J}}}

\newcommand{\eps}{\varepsilon}
\newcommand{\edm}{\end{displaymath}}
\def\benu{\begin{enumerate}}
\def\eenu{\end{enumerate}}
\def\beqn{\begin{equation}}
\def\eeqn{\end{equation}}
\def\be{\begin{equation}}
\def\ee{\end{equation}}
\def\bea{\begin{eqnarray}}
\def\eea{\end{eqnarray}}
\newcommand{\bean}{\begin{eqnarray*}}
\newcommand{\eean}{\end{eqnarray*}}
\newcommand{\bear}{\begin{eqnarray}}
\newcommand{\eear}{\end{eqnarray}}

\renewcommand{\epsilon}{\varepsilon}

\newcommand{\bbS}{\mathbb{S}}

\newcommand{\dist}{{\rm dist}}

\def\qed{\hfill\hbox{${\vcenter{\vbox{
    \hrule height 0.4pt\hbox{\vrule width 0.4pt height 6pt
    \kern5pt\vrule width 0.4pt}\hrule height 0.4pt}}}$}}

\begin{document}
\title{\bf Random coverage of a manifold with  boundary}

\author{
	Mathew D. Penrose$^{1,3}$ and Xiaochuan Yang$^{2,3}$ \\
{\normalsize{\em University of Bath and Banque Internationale \`a Luxembourg S.A.} }}

 \footnotetext{ $~^1$ Department of
Mathematical Sciences, University of Bath, Bath BA2 7AY, United
Kingdom: {\texttt m.d.penrose@bath.ac.uk} }

 \footnotetext{ $~^2$ {\texttt xiochuan.j.yang@gmail.com} }

 \footnotetext{ $~^3$ Supported by EPSRC grant EP/T028653/1 }






\maketitle

\begin{abstract}
Let $A$ be a compact $d$-dimensional $C^2$ Riemannian manifold with boundary, embedded in $\R^m$ where $m \geq d \geq 2$, and let $B$ be a nice subset of $A$ (possibly $B=A$).  Let $X_1,X_2, \ldots $ be independent random uniform points in $A$.  Define the {\em coverage threshold} $R_n$ to be the smallest $r$ such that $B$ is covered by the geodetic balls of radius $r$ centred on $X_1,\ldots,X_n$.  We obtain the limiting distribution of $R_n$ and also a strong law of large numbers for $R_n$ in the large-$n$ limit.  For example, if $A$ has Riemannian  volume 1 and its boundary has surface measure $|\partial A|$, and $B=A$, then if $d=3$ then $\Pr[n\pi R_n^3 - \log n - 2 \log (\log n) \leq x]$ converges to $\exp(-2^{-4}\pi^{5/3} |\partial A| e^{-2 x/3})$ and $(n \pi R_n^3)/(\log n) \to 1$ almost surely, while if $d=2$ then $\Pr[n \pi R_n^2 - \log n - \log (\log n) \leq x]$ converges to $\exp(- e^{-x}- |\partial A|\pi^{-1/2} e^{-x/2})$.

	We generalize to allow for multiple coverage.  For the strong laws of large numbers, we can relax the requirement that the underlying density on $A$ be uniform. For the limiting distribution, we have a similar result for Poisson samples.  Our results still hold if we use Euclidean rather than geodetic balls.


\end{abstract}

\section{Introduction}
\label{SecIntro}

The {\em random coverage} problem asks whether
a  set of $n$ balls in a metric space $A$, with centres
placed independently at random according to
some probability distribution 
$\mu$ on $A$, fully covers a specified set
$ B \subset A$ (possibly $B=A$).
Restricting attention to balls of equal radius, one may ask
what is the minimum radius required to achieve coverage.
This is a random variable (determined by the locations of the centres),
often called the {\em coverage threshold}, and here denoted $R_n$. 
The exact distribution of $R_n$ is often intractable, motivating
the study of large-$n$ asymptotics.

An important special case of this problem arises when
$\mu$ is the uniform distribution on 
a compact subset $A$ of Euclidean space $\R^d$, $d \geq 2$, 
and
the closure of $B$ is contained in the interior of $A$.
The large-$n$ limiting distribution
for $R_n$ (suitably scaled and centred)
was derived for this case by Janson \cite{Janson}, and also Hall 
\cite{HallZW}. For $d=1$, the exact distribution was
determined much earlier in \cite{Stevens}.

Another case arises where $\mu$ is uniform
on compact $A \subset \R^d$ as before, but
now $B=A$ so boundary effects come in to play.
For this case,
the limiting  distribution for $R_n$ was determined 
 in \cite{ECover}, for the case where $A$ has a smooth boundary.
 It turns out that
 if $d \geq 3$ then
 $R_n$ scales differently in this case than in the
 previous case, and if $d=2$ then the limiting distribution
 is different from in the previous case.

 A third case arises where $A$ is a compact $d$-dimensional Riemannian manifold
 (without boundary)
 and $\mu$ is uniform on $A$, with $B=A$.
 For this case  the limiting distribution for $R_n$ was also determined
 by Janson in \cite{Janson}.

 It would be desirable to
 generalize all of the cases above, by 
 taking
 $A$ to be a compact $d$-dimensional Riemannian manifold, possibly
 with boundary, and $B$ a closed subset of $A$
 (possibly $B=A$)
 and with $\mu$  uniform on $A$.  

 In this article we derive the limiting 
 distribution of $R_n$ in this general setting,  provided
 (i) $A$ is a submanifold-with-boundary
 of a further Riemannian  manifold $\cM$  (without boundary);
 (ii) $\cM$ is isometrically embedded 
 in the Euclidean space $\R^m$ for some $m \geq d$, and (iii)
 either $B \cap \partial A$ has non-zero surface measure or
 $B$ is bounded away from $\partial A$.

 As well as weak convergence results for $R_n$, we also
 derive a strong law of large numbers (SLLN), exhibiting
 a sequence of constants $c_n$ such that $R_n/c_n \to 1$
 as $n \to \infty$, almost surely. For the SLLN we
 consider more general probability density functions than
 just the uniform on $A$.

 Since we take $\cM \subset \R^m$, there  
 is a choice of two natural measures of distance  on $\cM$
 (used to define the balls); we can either
 use the geodetic distance induced by the Riemannian metric,
 or the Euclidean distance in the ambient space $\R^m$.
 We generally use the Riemannian distance, but will also comment
 on the Euclidean distance.
	 Our results hold whichever of these two
 choices of distance we use.
 We write $R_n^*$ to identify the
 coverage threshold using balls  with the Euclidean rather than
 geodetic distance.

 While the Euclidean coverage problem is rather classical (see
  \cite{Aldous,Flatto,HallZW,HallBk,Janson,Stevens}),
 more recent impetus for investigating
 these problems on manifolds comes from the modern 
 discipline of topological data analysis (TDA).
 One goal of TDA is to learn about the topological
 structure an underlying space from a random sample
 of points in that space. Often the space is a lower-dimensional
 submanifold embedded in Euclidean space $\R^m$. One common
way to create a topological structure from the sample of points
is to place Euclidean balls of radius $r$ around the points, and use these  
 to create a simplicial complex such as the \v{C}ech or 
Vietoris-Rips complex.
 See \cite{BK} for a survey.

 Niyogi et al. \cite[Proposition 3.1]{NSW} have shown
 that in the case $A = B = \cM$, if
 $r$ is below a certain constant $\tau(\cM)$  but greater than $R_n^*$,
 the homology of the union of balls
 equals that of $\cM$; Wang and Wang \cite{Wang2}
 extend this to the case where $A=B$ is a manifold with boundary.
 These results have created a further incentive to understand the probabilistic 
 behaviour of $R_n$ and $R_n^*$. Motivated by applications to
 TDA, Bobrowski and Weinberger \cite{Bob22,BW} 
 have derived (among other things) 
 results on the coverage threshold
 for manifolds without boundary, that are similar
 to those of Janson \cite{Janson}, while Kergorlay et al. 
 \cite{KTV} derive a weaker result
 for manifolds with boundary (we discuss this further later on). 

 We would expect some of the methods developed here to be relvant to investigateing
 other geometrical quantities of inerest in TDA (and elsewhere)
 for large samples in manifolds. Consider for example
 the {\em connectivity threshold} (i.e., the smallest $r$ such that
 the union of balls of radius $r/2$ centred on sample points is connected) and
 the {\em largest nearest-neighbour link} (i.e. the smallest $r$ such
 that each sample point is within distance $r$ of at least one other sample point).
 The large-sample asymptotic distribution for these quantities is established
 in \cite{PYCon} in the case where $A \subset \R^d$ is compact with smooth boundary,
 and it could be of interest to extend these results to more general manifolds.

 \section{Statement of results}

 \subsection{Mathematical framework}
 \label{s:MathFrame}

To be able to state the results of this paper,
we need to elaborate on the mathematical setup which was sketched
in the introduction.
We shall review definitions of concepts relating to manifolds
in Section \ref{secAltManif}.

 Let $d,m \in \N$ with $d \leq m$. Suppose 
 $\cM$ is a Riemannian manifold (without boundary)
 of dimension $d$, 
 and moreover that $\cM$ is isometrically embedded in the Euclidean
 space $\R^m$; more precisely, we shall assume 
 $\cM$ is a $C^2$- submanifold of $\R^m$.
 Suppose $A$ is a compact submanifold-with-boundary
 of $\cM$.
 Let $B \subset A$ (often we shall take $B=A$).
 Assume $B$ is closed (and hence compact).


Let $f$ be a probability density function on $\cM$, with respect
to the Riemannian volume measure; we assume throughout that
$f$ is supported by $A$. Let $\mu $ be the measure on
$A$ with density $f$.
Suppose on some probability space $(\Omega,\cF,\Pr)$ that
$X_1,X_2,\ldots$ are independent, identically distributed
random elements of $A$ with common distribution $\mu$.
For $x \in \cM$ and $r>0$ let $B(x,r):= \{y \in \cM:\dist(x,y) \leq r\}$,
where
$\dist(x,y)$ is the geodetic distance.
For $n \in \N$,  let $\X_n:= \{X_1,\ldots,X_n\}$.
Given also $k \in \N$, we
 define the {\em $k$-coverage threshold} $R_{n,k}$  by
\bea
R_{n,k} : =
 \inf \left\{ r >0: 
 \card(\X_n   \cap B(x,r))
 \geq k ~~~~ \forall x \in B \right\},
~~~ n,k  \in \N,
\label{Rnkdef}
\eea 
where for any point set  $\X \subset \R^m$
we let $\card(\X) $  denote the number of points of
$\X$,
 and we use the convention $\inf\{\} := +\infty$.
In particular $R_n : = R_{n,1}$ is the coverage threshold.  Observe that
$R_n  = \inf \{ r >0: B \subset \cup_{i=1}^n B(X_i,r) \}$.
In the case $B=A$, $R_n$ is the Hausdorff distance between the
sample $\X_n$ and the region $A$.

Thus, $R_{n,k}$ is a random variable, whose value is determined
by $(X_1,\ldots,X_n)$.  We investigate its  probabilistic behaviour 
as $n$ becomes large.  In particular, we would like to determine
the limiting behaviour of $\Pr[R_{n,k} \leq r_n]$
for any fixed $k$ and any sequence of numbers $(r_n)$.  Equivalently,
we  are concerned with the limiting distribution of $R_{n,k}$, suitably
scaled and centred, as $n \to \infty$ with $k$ fixed.

We also consider similar  results for a high-intensity {\em Poisson}
 sample in $A$, which may be more relevant in some applications.
Let  $(Z_t,t\geq 0)$ be a unit rate Poisson counting
process, independent of $(X_1,X_2,\ldots)$
and on the same probability space 
$(\Omega,\cF,\Pr)$
(so $Z_t$ is Poisson distributed with mean $t$ for each $t >0$).
The point process $\Po_t:= \{X_1,\ldots,X_{Z_t}\}$  is 
a Poisson point process in $\R^d $ with intensity measure $t \mu$
 (see e.g. \cite{LP}).
For $t \in (0,\infty), k \in \N$ we define a secondary $k$-coverage threshold
\bea
R'_{t,k} : = R_{Z_t,k} :=
 \inf \left\{ r >0: 
 \card( \Po_t \cap B(x,r))
 \geq k 
~~~~ \forall x \in B \right\},
~~~ t >0,
\label{Rdashdef}
\eea
with $R'_t:= R'_{t,1}$.

We mention further notation used throughout. For
$D \subset \cM$, let $\overline{D}$ 
denote the closure of $D$
and let $D^o := \cM \setminus (\overline{\cM \setminus D})$,
the interior of $D$ with respect to $\cM$.
Let $\partial D := \overline{D} \setminus D^o$,
the boundary of $D$
and let $\partial_{\partial A} (D \cap \partial A) := \overline{D \cap \partial A}
\cap \overline{\partial A \setminus D}$, the boundary
of $D \cap \partial A$ relative to $\partial A$.
Let $v(D)$ denote the Riemannian volume of $D$, when defined, and
if $D \subset \partial A$,
let $\tv(D)$ denote the Riemannian surface measure of
$D$ (see Section \ref{secAltManif}
for further details).
Given $t >1$, we write $\lglg t$ for $\log (\log t)$.
Given two  sets $\X,\Y \subset \R^d$, we
write $\X \oplus \Y$ for the set $\{x+y: x \in \X, y \in \Y\}$.
%
Given $x,y \in \R^d$, we denote by $[x,y]$ the line segment from
$x$ to $y$, that  is, the convex hull of the set $\{x,y\}$.
%
%
Given $d \in \N$, define the constant
\bea
c_d := \frac{1}{d!} \left( \frac{\sqrt{\pi} \; \Gamma(1 + d/2) }{ \Gamma((d+1)/2) }
\right)^{d-1}.
\label{cdef}
\eea
Note that $c_1=c_2=1$, and  $c_3 =3 \pi^2/32$.
Set $\vvol_d := \pi^{d/2}/\Gamma(1+ d/2)$,
the volume of the unit ball in $\R^d$, with $\vvol_0 := 1$.
Given also $k \in \N$, for $d \geq 2$ set
\bea
c_{d,k} := \left( \frac{c_{d-1}}{(k-1)!} \right)
\vvol_d^{2-d-1/d} \vvol_{d-1}^{2d-3}
\vvol_{d-2}^{1-d} 
(1- 1/d)^{d+k-3 + 1/d} 2^{-1+1/d}.
\label{cdkdef}
\eea
Note that $c_{2,k}
= 2^{1-k} \pi^{-1/2}/(k-1)! $ and $c_{3,k}= 2^{k-5} 3^{1-k} \pi^{5/3}/(k-1)!$.
In all of our limiting statements,
$n$ takes values in $\N$ while $t$ takes values in $(0,\infty)$.

Given $n \in \N$, we write $[n]$ for $\{1,\ldots,n\}$.

	Sometimes, given $a \in (0,1)$ and
	$b >0$, we shall write $(1 \pm a)b$ for the
	interval $[b-ab, b+ ab]$.

\subsection{Main results}

From now on, we always assume $d \geq 2$.
Our first result is a weak limit for the coverage threshold
of a manifold $A$ with boundary or subset thereof, when the
centres are uniformly distributed on $A$.

\begin{theo}
\label{th:weak}
	Set  $f_0 := 1/v(A)$ and assume $f \equiv f_0$ on $A$.
	Assume
	$v(\partial B)=0$ and $\tv(\partial_{\partial A} ( B \cap \partial A )) =0$.
	Let $k \in \N$, $\zeta \in \R$.  Then
\bea
	\lim_{n \to \infty} \Pr \left[ \left( \frac{n \vvol_d f_0 R_{n,k}^d}{2}
	\right)- 
	\left(\frac{d-1}{d}\right) \log (n f_0) -
\left( d+k-3+1/d \right)
	\lglg n \leq \zeta \right]
\nonumber \\
=
	\lim_{t \to \infty} \Pr \left[ \left( \frac{t \vvol_d 
	f_0 (R'_{t,k})^d}{2}
	\right)- 
	\left(\frac{d-1}{d}\right) \log (t f_0) -
\left( d+k-3+1/d \right) 
	\lglg t \leq \zeta \right]
\nonumber \\
=
	\begin{cases}
		\exp \left( - v(B) e^{-2 \zeta}
		- c_{2,1}  \tv( B \cap \partial A) e^{- \zeta}  
		\right) 
		&{\rm if} ~ d=2, k =1 \\
		\exp \left( - c_{d,k}  \tv( B \cap \partial A) e^{- \zeta}  
		\right) 
		&{\rm otherwise.}
	\end{cases}
		~~~~~~~~~~~~~~~
\label{eqmain3}
\eea
\end{theo}

\begin{remk}
	Theorem \ref{th:weak}
	was previously proved  in
	\cite[Theorem 2.1]{ECover}  for the special case
where $A$ is a smoothly bounded subset of $\R^d$,
i.e. when $\cM = \R^d$, 
with the Euclidean inner product as Riemannian metric,
and where moreover $B=A$.
Even in this setting, our result generalises that one by allowing
	for more general $B$.
\end{remk}

\begin{remk}
When $d=2, k=1 $ and $B=A$,
the exponent in (\ref{eqmain3}) has two terms.
This is because
the location in $A$ furthest from the sample $\X_n$ might lie
	either in the interior of $A$, or on the boundary. 
If $d > 2$ or $k>1$ then for large $n$, the location in $A$
furthest from $\X_n$ is very likely to lie near the boundary of $A$. 
\end{remk}

\begin{remk}
	\label{rmkB}
	When $d > 2$ or $k > 1$,
	and $\tv(B \cap \partial A) =0$, the limit in
	(\ref{eqmain3}) equals 1, so in this case taking
	$\frac12 n \vvol_d f_0 R_{n,k}^d - (1-1/d) \log n - (d+k -3 + 1/d)
	\lglg n + {\rm const.}$
	is the `wrong' transformation of $R_{n,k}$
	to yield a sequence of random
	variables converging in distribution to
	a non-degenerate random variable. In
	the case where $\overline{B} \subset A^o$,
	the `right' transformation is identified in 
	Proposition  \ref{Hallthm} below, but in intermediate
	cases where $\tv(B \cap \partial A) = 0$
	but $B$ touches $\partial A$, finding the right sequence
	of transformations  remains an open problem.
\end{remk}
\begin{remk}
	One possible direction of future research would be to give a more
	quantitative version of Theorem \ref{th:weak}, specifically a
	bound on the rate of convergence (for either choice
	of pre-limit) in \eqref{eqmain3}. In \cite{PY25}
	we make a start in this direction by providing
	an $O((\log \log t)/\log t)$ rate of convergence for
	convergence of the second pre-limit in \eqref{0114a}
	below,
	in the special case where $\cM = \R^d$.
\end{remk}

Our next result provides
a  
SLLN for $R_{n,k}$. For this,
  we relax the condition that the density function
  $f$ be constant on $A$.
%
Moreover,  we  allow $k$ to vary with $n$.
	Assume we are given 
	a constant $\sbeta \in [0,\infty]$
	and a sequence $k: \N \to \N$ with
	\bea
	\lim_{n \to \infty} \left( k(n)/\log n \right) = \sbeta;
	~~~~~ \lim_{n \to \infty} \left( k(n)/n \right) = 0.
	\label{kcond}
	\eea
We make use of the following notation throughout:
\bea
f_0 := \inf_{x \in B} f(x); ~~~~~~~ f_1:= \inf_{x \in B \cap \partial A}
f(x);
\label{f0def}
\\
 H(t) := \begin{cases}  1-t + t \log t, ~~~ & {\rm if} ~ t >0 \\
	 1 ,  & {\rm if} ~ t =0.
 \end{cases}
	 \label{Hdef}
\eea
If $B \cap \partial A = \emptyset$ we interpret $f_1$ as $+\infty$
and $1/f_1$ as 0.
	Observe that $-H(\cdot)$ is unimodal with a maximum
value of $0$ at $t=1$. Given $a \in [0,\infty)$, we define
the
function $\hH_a: [0,\infty) \to [a,\infty)$ by
\bea
y = \hH_a(x)  \Longleftrightarrow y H(a/y) =x,~ y \geq a,
\label{hatHdef}
\eea
with $\hH_0(0):=0$.
Note that 
$\hH_a(x)$  is  increasing in  $x$,
and that 
$\hH_0(x)=x$.

	Throughout this paper, the phrase `almost surely' or 
	`a.s.' means `except on a set of $\Pr$-measure zero'.
%
Let $B_A^o : = B \setminus (\overline{A \setminus B})$, 
the interior of $B$ relative to $A$, and
 let $f|_A$ denote the restriction of $f$ to $A$.

\begin{theo}
\label{thm2}
Suppose 
 that $f_0 >0$, and that $f|_A$ is continuous on $A$.
	Assume also that $B= \overline{B^o}$, and
	$B\cap \partial A = \overline{B_A^o \cap \partial A}$
	with $\int_B f dv >0$,
	and that \eqref{kcond} holds.
Then as $n \to \infty$, almost surely
\bea
(n \vvol_d R_{n,k(n)}^d)/k(n) \to \max \left( 1/f_0, 2 /f_1 \right)
	~~~~~{\rm if}~ \sbeta =\infty;
\label{th2eq1}
\\
(n \vvol_d R_{n,k(n)}^d)/\log n \to 
	\max \left( \hH_\sbeta(1)/f_0, 2 \hH_{\sbeta}(1-1/d)/f_1  
\right),
	~~~~~{\rm if}~ \sbeta < \infty.
\label{th2eq}
\eea
In particular, if $k \in \N$ is a constant, then
as $n \to \infty$, almost surely 
\bea
	(n \vvol_d R_{n,k}^d)/\log n \to 
	\max \left( 1/f_0,  (2-2/d)/f_1 \right).
\label{th2eq0}
\eea
\end{theo}

In the special case where
$\mu$ is the uniform distribution on
$A$ (i.e. $f(x) \equiv f_0$  on $A$),
the right hand
side of (\ref{th2eq0}) comes to $(2-2/d)/f_0$.

\begin{remk}
	We can compare the above result with \cite[Corollary 4.2]{KTV}.
	Translated into our notation, the latter result says 
	that when $k(n) \equiv 1$ (so $\sbeta =0$)
	and $f \equiv 1$, then
	$\limsup_{t \to \infty} \left( t \vvol_d (R'_t)^d / \log t
	\right) \leq 2^d $, in probability.

	Also related are results of Chai \cite{Chai}, which  show
	limiting behaviour  for the expected number of balls of radius $r$
	needed to cover $A$, and allow for other singularities of
	$\partial A$ rather than just considering smooth boundaries.
\end{remk}

\begin{remk}
	\label{rmkeuc}
	Since we are considering a submanifold $\cM$ of
	$\R^m$, a natural alternative,
	instead of the geodetic distance in $\cM$, is
	to consider the Euclidean distance in $\R^m$.
For $x \in \R^m, r \geq 0$ let $B_{\R^m}(x,r)$ denote
the closed Euclidean ball of radius $r $ in $\R^m$ centred on $x$.
	Then define $R^*_{n,k}$
	like $R_{n,k} $ at \eqref{Rnkdef} but using $B_{\R^m}(x,r)$
	instead of the geodetic ball $B(x,r)$.

	Theorems \ref{th:weak} 
	and \ref{thm2} still hold with $R^*_{n,k}$ instead of $R_{n,k}$
	and $R^*_{Z_t,k}$ instead of $R'_{t,k}$. 
	We shall justify this assertion in
	Remark \ref{rk:geovseuc}.
\end{remk}

\begin{remk}
	\label{rmknonsub}
	We conjecture that Theorems \ref{th:weak}
	and \ref{thm2} still hold if we assume only
that  $A$ is a compact Riemannian manifold-with-boundary,
 of dimension $d \geq 2$ and Riemannian volume $v(A) \in (0,\infty)$,
	dropping the assumption that $A$ is a submanifold of 
	a further manifold $\cM$ that is isometrically embedded
	in $\R^m$. 
\end{remk}
\begin{remk}
	When $B=A$, with some extra work one can
	relax the condition `$f|_A$ is continuous on $A$'
	to `$f|_A$ is continuous at all points of $\partial A$'
	in Theorem \ref{thm2}, provided we define $f_0$ as
	an essential infimum.
\end{remk}

\subsection{Coverage in the interior region}

Given $r >0$, we define the
`$r$-interior' of $A$ by
$
A^{(r)} := \{x \in A: B(x,r) \subset A^o\},
$
i.e.
 the set of points in $A$ at geodetic distance
(strictly) more than $r$ from $\cM \setminus A$. 

Given closed $B \subset A$
as before,
let $\tilde{R}_{n,k} $ be the smallest $r$ such that $ B \cap A^{(r)} $
is covered $k$ times by the balls of radius $r$ centred on the points of $\X_n$, i.e.  set
\bea
\tilde{R}_{n,k} : = \inf \left\{ r >0: 
\card(\X_n  \cap B(x,r))
\geq k ~~~ \forall x \in  B \cap A^{(r)} \right\},~~~ n,k \in \N,
\label{eqmaxspac}
\eea
with $\tR_n: = \tR_{n,1}$.  Also, for $t >0$ define
\bea
\tilde{R}_{Z_t,k} : = \inf \left\{ r >0: 
\card(\Po_t  \cap B(x,r)) 
\geq k ~~~ \forall x \in  B \cap A^{(r)}
\right\},~~~ n,k \in \N.
\label{eqmaxspacPo}
\eea

In Proposition \ref{Hallthm},
we give the limiting distribution  for $\tR_{n,k}$.
This is simpler than Theorem \ref{th:weak}
because boundary effects are avoided. We shall be 
using  this result to prove Theorem \ref{th:weak}.

\begin{prop}
\label{Hallthm}
	Suppose 
	$f(x)= f_0 = 1/v(A)$ for all $x \in A$, and
	$v(\partial B) =0$.  Let $k \in \N$ and $\beta \in \R$.
Then
\bea
\lim_{n \to \infty} \Pr[ n \vvol_d f_0  \tR_{n,k}^d   - \log (n f_0)  -
(d+k-2) \lglg n \leq \beta]
\nonumber \\
=
\lim_{t \to \infty} \Pr[ t \vvol_d f_0 (\tR_{Z_t,k})^d   - \log (t f_0)  -
(d+k -2) \lglg t \leq \beta]
\nonumber \\
	= \exp(-(c_d / (k-1)!) v(B)e^{-\beta} ).
\lbl{1228a}
\eea
        If moreover $B$ is compact with $B \subset A^o$, then
\bea
\lim_{n \to \infty} \Pr[ n \vvol_d f_0  R_{n,k}^d   - \log (n f_0)  -
(d+k-2) \lglg n \leq \beta]
\nonumber \\
=
\lim_{t \to \infty} \Pr[ t \vvol_d f_0 (R'_{t,k})^d   - \log (t f_0)  -
(d+k -2) \lglg t \leq \beta]
\nonumber \\
	= \exp(-(c_d / (k-1)!) v(B)e^{-\beta} ).
\lbl{0114a}
\eea
\end{prop}

Proposition \ref{Hallthm}
is concerned with coverage of that part of $B$ which is not
close to $\partial A$, and   
 of interest in its own right.
The special cases where
$B=A = \cM$
(i.e. $A$ has no boundary)
 \cite[Theorem 1.2]{Janson} 
and where $\cM = \R^d$ \cite{HallZW, Janson, PY25}
were already known.


The following SLLN for
$\tR_{n,k}$ will be needed 
for proving Theorem \ref{thm2}.

\begin{prop}
\label{thm1}
	 Suppose that 
	$\mu(B) >0$,
	and $B = \overline{B^o}$,
	and $f|_A$ is continuous on $A$. 
	Assume (\ref{kcond}) holds.
	Then, almost surely,
	\begin{align}
		\lim_{n \to \infty} ( n \vvol_d \tR_{n,k(n)}^d/k(n)) & = f_0^{-1}
		&{\rm if}~ \beta = \infty;
\label{0617a}
\\
		\lim_{n \to \infty} (n \vvol_d  \tR_{n,k(n)}^d /\log n) & =  
	\hH_\beta(1)/f_0,
		&{\rm if}~ \beta < \infty.
\label{0315a}
	\end{align}
In particular, if $k \in \N$ is a constant then
\bea
 \Pr[ \lim_{n \to \infty} ( n \vvol_d f_0 \tR_{n,k}^d /\log n) = 1 ] =1.
\label{Strong1}
\eea
If also 
	$B \subset A^o$, all of the above almost sure limiting
 statements hold for $R_{n,k(n)}$ as well as for $\tR_{n,k(n)}$.
\end{prop}

\subsection{Overview of proofs}

The proofs all use the following fact about manifolds, presented in detail 
in Section \ref{s:ApproxEuc}.
Given $x_0 \in A$, we can find a neighbourhood $V$ of $x_0$ in $\cM$ and
for each $y \in V$, an injective mapping $\phi_{y}$ from $V$ to 
a region of  Euclidean $d$-space
such that the distortion of $\cM$ by $\phi_y$
in a small neighbourhood of
$y$ is itself small, and such that moreover  $\phi_y \circ 
\phi_{x_0}^{-1}$
can be extended to a linear map on $\R^d$ that is
fairly close (uniformly over $y \in V$) to the identity map. 
 If $x_0 \in \partial A$, we can arrange for $\phi_{y}(A \cap V)$ to be
 the intersection of $\phi_y(V)$ with the upper half-space of $\R^d$.

As we shall describe below, we can compute the limiting
probability of covering $A \cap V$ with balls of radius $r_t$ centred
on $\Po_t$, where $r_t$ is chosen so this limiting probability is
non-trivial;
at the end we can put together different regions $V$ to determine
the limiting probability of covering all of $A$ (or more generally,
a specified subset $B$ of $ A$).

To compute this  limiting probability, we fix
a small positive constant $\eta$ and given $t$, we
divide $\phi_{x_0}(A \cap V)$
into cubes $H_{t,i}$ of side $t^{-\eta}$
(along with a small boundary region denoted $D_t$ for now).
The sets $\phi_{x_0}^{-1}(H_{t,i})$ (along with $\phi_{x_0}^{-1}(D_t)$)
partition the set $A \cap V$, and
therefore the total volume of the regions $\phi_{x_j}(\phi_{x_0}^{-1}(H_{t,i}))$
approximates to that of $A \cap V$.

For each $i$ we select a point $x_i \in \phi_{x_0}^{-1}(H_{t,i})$, and
 consider also $\phi_{x_i}(\phi_{x_0}^{-1}(H_{t,i}))$, a region
of $\R^d$ which is not
a cube (it is a $d$-dimensional parallelogram) so we
subdivide it into
smaller cubes denoted $Q_{t,i,\ell}$ of side $t^{-\alpha}$, where $\alpha  \in
(\eta, 1/d)$ is a constant (along with a boundary region that is small because
$t^{-\alpha} \ll t^{-\eta}$).
Then
the total volume of all the cubes $(Q_{t,i,\ell})_{i,\ell}$
approximates the Riemannian volume of $A \cap V$.

By the theory of Poisson point processes, the restriction of $\phi_{x_i}(
\Po_t \cap  \phi_{x_i}^{-1}(Q_{t,i,\ell}))$ is a further Poisson process in
$Q_{t,i,\ell}$ that we can approximate to with
a homogeneous Poisson process, and the event that $A$ is covered
corresponds to the event that
the union of all the cubes $Q_{t,i,\ell}$ are covered. We can reassemble
these cubes into a macroscopic region of volume about $v(A \cap V)$.
There are boundary effects on 
the annular shells of thickness $r_t$ in each cube, but these effects can be
shown to be negligible in the limit $t \to\infty$ because
$r_t \ll t^{-\alpha}$. Thus we can show our coverage probability is
asymptotically equivalent to a Euclidean coverage problem for the
union of reassembled cubes.

We have to take into account the effect of the boundary $\partial A$,
if $V$ intersects this set.
In the asymptotically equivalent Euclidean coverage problem just described, 
this amounts to the probability of covering a region of the half-space in
$\R^d$ by balls centred on a Poisson process in the half-space, which can be
dealt with by noticing that the induced coverage process on the boundary of
the half-space is a lower-dimensional Boolean model coverage problem which is already well understood.

 For the laws of large numbers, we shall 
 use  asymptotic upper and lower bounds
 on coverage thresholds for a region in a general metric space endowed with a
 probability measure, given in
 \cite[Section 6]{ECover}. These bounds are based on having good upper
 bounds on the covering number of a set by small balls,
 and good lower bounds on a certain packing number, roughly
 the number of disjoint  small balls of measure less than a certain amount that can be found with disjoint centres in that region.  Using 
 can use the preceding fact one can show
 that given $\varepsilon > 0$, we can bound these covering and packing numbers
 the same way as in the Euclidean case to within a factor of $\varepsilon$,
 and this enables us to get the law of large numbers.

\subsection{Manifolds}
\label{secAltManif}
Our results refer to a number of concepts pertaining to manifolds,
and  we review these  now. We mainly follow the approach of 
\cite{PYmfld}.


For $k \in \N$, let $\| \cdot \|$ be the Euclidean norm in $\R^k$.
Let $o$ denote 
the origin of $\R^k$.

 Let $m \in \N$ and $d \in \N$ with $d \leq m$.
We write $\lambda_d$
for Lebesgue measure in $\R^d$, and $\lambda_{d-1}$ 
for $(d-1)$-dimensional Hausdorff measure.
Let $\bH$ denote the half-space $\R^{d-1} \times [0,\infty)$,
and set $\partial \bH:= \R^{d-1} \times \{0\}$.

A nonempty subset $\M$ of $\R^m$,
endowed with the subset topology,
 is called a $d$-dimensional $C^2$
submanifold of $\R^m$ if for each $y \in \M$ there exists an open
subset $U$ of $\R^d$ and a twice continuously differentiable
injection $g $ from
 $U $ to $\R^m$, such that (i) $y \in g(U) \subset \M$,
and (ii) $g$ is an open map from $U$ to $\M$, and
(iii) the linear map $ g'(u)$ has full rank for all $ u \in U$ (see
e.g. Theorem 2.1.2(v) of \cite{BG}).  The pair $(U,g)$ is called a
{\em chart}. Let ${\m}:= \m(d,m)$ denote the class of all
$d$-dimensional $C^2$ submanifolds of $\R^m$ which are also closed
subsets of $\R^m$.

Given $\M \in \m$, using a routine compactness argument we can
choose an index set $\I \subset \N$,
and  a set
$\{(y_i,\delta_i,U_i,g_i),i \in {\cal I} \}$ of ordered quadruples
with $y_i \in \M$, $\delta_i \in  (0, \infty)$, and
$(U_i,g_i)$ a chart for each $i$, such that
(i)  $\M \cap B_{\R^m}(y_i,3 \delta_i)
\subset g_{y_i}(U_i)$ for each
$i$, and (ii) $\M \subset \bigcup_{i \in {\cal I}} B_{\R^m}(y_i,\delta_i)$.

We refer to $((U_i,g_i), i \in {\cal I})$ as an  {\em atlas} for
$\M$. Given such an  atlas, we can find a {\em partition of unity}
$\{\psi_i\}$ subordinate to the atlas, that is, a collection of measurable
functions $(\psi_i,i \in {\cal I})$ from $\M$ to $[0,1]$, such that
$\sum_{i \in {\cal I}} \psi_i(y) =1$ for all $y \in \M$, and such
that for each $i$, $ \psi_i(y) = 0 $ for $y \notin g_i(U_i)$, and
 $\psi_i \circ g_i$ is a measurable function on $U_i$.
The more common definition of a partition of unity has some extra
continuity and differentiability conditions on $\psi_i$ but these are not needed
here. With our more relaxed definition, the existence of a partition
of unity is completely elementary to prove.

Given $i \in {\cal I} $ and $x \in U_i$, let $D_{g_i}(x):= \sqrt{
\det (J_{g_i,x}^TJ_{g_i,x}) }$, with $J_{g_i,x}$ standing for the
Jacobian of ${g_i}$, evaluated at $x$.
For bounded measurable $h: \M \to \R$, the integral $\int_{\M} h(y)
dy$ is defined by
\be
 \lbl{int-form}
 \int_{\M} h(y) dy = \sum_{i \in {\cal I}}
\int_{U_i} h(g_i(x)) \psi_i (g_i(x)) D_{g_i}(x) \lambda_d(dx) 
\ee 
which is
well-defined in the sense that it does not depend on the choice of
atlas or the partition of unity.
Also, if $h= {\bf 1}_D$ for some set $D \subset \M$, then set $
v(D) := \int_{\cM}  h(y) dy$, the 
Riemannian volume of $D$.

Here is further justification of Equation (\ref{int-form}).
For bounded measurable $A\subset U_i$,
and $x \in U_i$,
it is a fact from linear algebra
that
\bea
 \lambda_d(g'_i(x)(A)) = D_{g_i}(x) \lambda_d(A). \lbl{0508d}
\eea
Indeed, the columns of
the Jacobian matrix $J_{g_i,x}$ are the images under $g'(x)$ of
the standard basis vectors of $\R^d$ so (\ref{0508d})
clearly holds when the standard basis vectors map to an orthonormal
system, but then it can be deduced in the general case using
standard properties of determinants. Equation (\ref{0508d}) is the
basis of the formula (\ref{int-form}).

Given any manifold $\M \in \m$ and $A \subset \M$,
recall our notation $A^o := \cM \setminus (\overline{\cM \setminus A})$
and $\partial A := \overline{A} \setminus A^o$.
Also, set
$ \diam ( A)  := \sup\{\|x - y\|: \ x,y \in A  \}$
(thus we define diameter using Euclidean rather than geodetic distance).
We say
$A$ is a $d$-dimensional $C^2$
{\em submanifold-with-boundary} of $\M$ if for all $y \in \partial A$,
there exists a choice  of chart $(U,g)$ for $\M$
such that $o \in U$, and $g(o)=y$, and $g( U \cap \bH)
= g(U) \cap A$,
where we set $\bH := \R^{d-1} \times [0,\infty) $.
This includes the possibility that $\partial A = \emptyset$
(but if $\partial A \neq \emptyset$ then it implies $A$ is closed).

For such $A$, we claim $\partial A$ is a $(d-1)$-dimensional
submanifold of $\R^m$.
Indeed, let $y \in \partial A$ and
take a chart $(U,g)$ with $o \in U$, $g(o) =y$ and
$g(U \cap \bH) = g(U) \cap A$. Set $U' := U \cap \partial \bH$.
This is an open set in $\R^{d-1}$ (where we identify $\R^{d-1}$
with $\partial \bH$). Let $\tg$ denote the restriction of
$g$ to $U'$. Then $\tg$ is continuously differentiable
on $U'$ (since $g$ is continuously differentiable on $U$),
and $y = g(o) \in \tg(U') \subset \partial A$. Also,
if $W \subset U'$ is open, then $W' := W \cup (U \setminus U')$
is open in $U$, so $g(W')$ is  open in $\M$, 
so that $g(W) = g(W') \cap \partial A$ is open in $\partial A$;
hence, $\tg$ is an open map from $\R^{d-1}$ to $\partial A$.
Finally, for all $u \in U'$, since $g'(u)$ has full rank,
the $m \times d$ matrix $(\frac{\partial \tg_i}{\partial x_j}|_u)$  
has rank $d$, and therefore the same  matrix with the last
column removed must have rank $d-1$, i.e. $\tg'(u)$ has full rank.

By the preceding claim, we can define surface integrals over $\partial A$
analogously to the earlier definition given at (\ref{int-form}).
In particular, for measurable $E \subset \partial A$
the
surface measure $\tv(E)$ is given by
\be
 \lbl{surf-form}
 \tv(E ) = \sum_{i \in {\cal I}}
\int_{U_i \cap \partial \bH} 
{\bf 1}_E (g_i(x))
\psi_i (g_i(x)) \tD_{g_i}(x)\lambda_{d-1}(dx), 
\ee 
where $\lambda_{d-1}$ is  
is  $(d-1)$-dimensional Lebesgue measure,
and where, for $x \in U_i \cap \partial \bH$,
we set $\tD_{g_i}(x) = \sqrt{ \det (\tJ_{g_i,x}^T
\tJ_{g_i,x}) } $, with $\tJ_{g_i,x}$ standing for
the $m \times (d-1)$ matrix obtained by removing
the last column from $J_{g_i,x}$.

 Given  $\M \in \m$, a {\em probability density function}
on $\M$ is a non-negative scalar field
 $f$ on $\M$ satisfying $\int_{\M} f(y) dy = 1$.

\begin{remk}
	Theorems \ref{th:weak} and \ref{thm2}
	generalize results from \cite{ECover}
	(Theorems 3.1 and 4.1 there), which
	are concerned with the
	case where $\cM = \R^d$ and
	$A$ is a compact  subset of $\R^d$ with a $C^2$ boundary.
	It can be shown using the proof of 
	 \cite[Lemma 6.7]{ECover} that in this case $A$ is indeed
	 a submanifold-with-boundary of $\R^d$.
\end{remk}

\section{Approximating the manifold by Euclidean space}
	\label{s:ApproxEuc}
\allco

	In this section we write $\lambda_d$ for $d$-dimensional
	Lebesgue measure, either on the Euclidean space $\R^d$,
	or on any $d$-dimensional subspace of a higher-dimensional
	Euclidean space. We let $v_d$ denote $d$-dimensional Hausdorff
	measure, and $\bbS^{d-1} : = \{ e \in \R^d: \|e\| =1\}$,
	the $(d-1)$-sphere.


	Let $I_d$ denote the identity map from
	$\R^d$ to itself,  and let $\| \cdot \|_{d \times d}$
	be the operator norm on 
	linear
	maps from $\R^d$ (with the Euclidean norm) to itself.
	Denote the unit coordinate vectors in $\R^d$ by $e_1,\ldots, e_d$.

	To prove
	Theorems
	\ref{th:weak} and \ref{thm2},
	we need to quantify the idea that a $d$-dimensional
	manifold can be approximated locally by Euclidean $d$-space $\R^d$.
	The next result is key to doing this.
\begin{lemm}
	[Comparing manifold with Euclidean space]
	\label{lemcompare3}
	Let $x_0 \in \M$, $\eps \in (0,1/2)$.
	Then there exists an open set $V \subset \M$,
	with $x_0 \in V$, a constant $K_1 \in (0,\infty)$,
	and for each $\tx \in V$ an injective open map
	$\phi_\tx: V \to \R^d$, with
	$\phi_{\tx}(x_0) = o$ 
	such that:

	(a)
	If $\tx \in V \cap \partial A$ then
	$\phi_\tx(V \cap A)
	= \phi_\tx(V) \cap \bH$. 

	(b) The mapping $\phi_\tx \circ \phi_{x_0}^{-1}: 
	\phi_{x_0}(V) \to \R^d$
	extends to a linear mapping of of full rank (also
	denoted $\phi_\tx \circ \phi_{x_0}^{-1}$) from
	$\R^d$
	to itself, such that
	$\|\phi_\tx \circ \phi_{x_0}^{-1} - I_d\|_{d \times d} < \eps$
	and 
	$\phi_{\tx} \circ\phi_{x_0}^{-1}(\bH)= \bH$.

	(c) For all measurable  $F \subset V $  and $\tx \in F$,
	$$
	\left| \frac{\lambda_d(\phi_\tx(F))}{ v(F) } -1 \right|
	\leq K_1 \diam (F).
	$$

	(d) If $x_0 \in \partial A$  then
	for all  measurable $F \subset V \cap \partial A $  and $\tx \in F$,
	$$
	\left| \frac{\lambda_{d-1} (\phi_\tx(F))}{ \tv(F) } -1 \right| \leq K_1 \diam (F).
	$$

	(e) For all $\tx, y,z \in V$, we have  
	$$
	\left| \frac{\|\phi_\tx(y) - \phi_\tx(z) \|}{ \dist(y,z)  } -1 \right|
	\leq K_1 ( \dist(\tx,y) + \dist(y,z) ).
	$$



	(f) 
	We have
	$|(\|\phi_\tx(y)-\phi_\tx(z)\|/ \dist(y,z)) - 1| \leq \eps$
	for all distinct $y,z \in V$
	and  all $\tx \in V$.
\end{lemm}
\begin{remk}
	\label{fromcompare2}
	Since we take $\eps \leq 1/2$, and
	$(1+x)^{-1} \geq 1-x$ and $(1-x)^{-1} \leq 1+2x$
	for all  $x \in [0,1/2]$, we have from (f)
	a similar bound for the ratio  the other way, namely 
	$|(\dist(y,z)  / \|\phi_\tx(y)-\phi_\tx(z)\|    ) - 1| \leq 2 \eps$;
	likewise for other ratios such as those in (c) and (d) when
	$\diam (F)$ is sufficiently small.
	We shall use this fact from time to time without further comment.
\end{remk}
\begin{proof}
	Let $(U,g)$ be a chart with
	$x_0 \in g(U)=: V$, so $U \subset \R^d$ and $g:U \to \R^m$
	is $C^2$. Without loss of generality, we can and do assume 
	that $o \in U$ and $g(o) =x_0$, and if $x_0 \in \partial A$, that 
	$g(U \cap \bH) = V \cap A$.
	Assume moreover that $U$ is convex. 

	Given $\tx \in V$, set $\tiu = g^{-1}(\tx)$ and
	let $\psi_{\tx} : V \to \R^d$ be given by
	$\psi_\tx := g'(\tiu) \circ g^{-1}$; note that
	the matrix representation of the linear map
	$g'(\tiu)$ is the Jacobian matrix $J_{g,\tiu}$ in
	the notation of Section \ref{secAltManif}.
	Let $\rho_\tx$
	be a linear isometry 
	from $\phi'(\tiu)(\R^d)$ to $\R^d$, to be specified below. 
	Let $\phi_{\tx} := \rho_{\tx} \circ \psi_{\tx}$.
	Then $\phi_{\tx}: V \to \R^d$ is an open map. Also
	$\phi_{x_0} = \rho_{x_0} \circ g'(o) \circ g^{-1}$.
	Note $g'(o):\R^d \to \R^m$ is linear and of full rank.
	Thus $\psi_{x_0}(V)= g'(o)(U)$ is part of the $d$-dimensional
	subspace $g'(o)(\R^d)$ of $\R^m$.

	We now specify the linear isometry
	$\rho_{\tx}$
	explicitly, for each $\tx$ in $V$,
	using a Gram-Schmidt type procedure.
	Set $L_\tiu := g'(\tiu)$
	(where $\tiu = g^{-1}(\tx)  \in U$). Then
	$L_{\tiu}$ is a linear
	map of full rank from $\R^d$ to $\R^m$.

	Set $v_1 =  L_\tiu(e_1)$, set
	$w_1 = \|v_1\|^{-1} v_1$, and
	set $\rho_{\tx}(w_1) = e_1$. 
	
	Set $v_2 = L_\tiu(e_2) - \langle L_\tiu(e_2) , w_1 \rangle w_1$,
	set $w_2 = \|v_2\|^{-1} v_2$ and 
	set $\rho_{\tx}(w_2 ) = e_2$.

	Continuing in this way if $d \geq 3$, for $3 \leq j \leq d$ set
	 $v_j = L_\tiu(e_j) - \sum_{i=1}^{j-1} \langle L_\tiu(e_j), w_i \rangle 
	 w_i$,
	 then set $w_j =  \|v_j\|^{-1} v_j $ and $\rho_{\tx}(w_j) = e_j$.
	 Since $L_\tiu$ has full rank, we have $v_j \neq o$ for all $j$.

	 Extend $\rho_{\tx} $ to all of $L_\tiu(\R^d)$ by linearity.
	 Since $w_1,\ldots,w_d$ are orthonormal and
	mapped to orthonormal vectors in $\R^d$,
	$\rho_\tx$ is an isometry.
	 Then the matrix representation of
	 $\rho_{g(\tiu)} \circ L_\tiu$,
	 with $(i,j)$ entry given by
	 $\langle \rho_{g(\tiu)} \circ L_\tiu (e_i),e_j \rangle$ for
	 each $i,j \in [d]$,
	 has entries which depend continuously on $\tiu$.
	 It is a lower triangular matrix with diagonal entries $\|v_1\|,\ldots,
	 \|v_d\|$.

	 Since its matrix representation
	 is  lower triangular with positive diagonal entries, 
	 $\rho_{g(\tiu)} \circ L_\tiu$ maps $\bH$ to itself, 
	that is, $\rho_\tx \circ L_\tiu (\bH) = \bH$.
	Since $g( U \cap \bH)= V \cap A$, if $y \in V \cap A$ then
	$g^{-1}(y) \in \bH$ and hence 
	$$
	\phi_\tx(y)
	= \rho_{\tx} \circ L_\tiu(g^{-1}(y)) \in \bH.
	$$
	Conversely, if $z \in V \setminus A$ then $g^{-1}(z) \notin \bH$,
	so $\phi_\tx(z) \notin \bH$. This yields part (a).

	By definition we have  $\phi_{x_0} = \rho_{x_0}
	\circ g'(o) \circ g^{-1}$
	and $\phi_{\tx} = \rho_{\tx} \circ g'(\tiu) \circ g^{-1}$.
	Therefore 
	\bea
	\phi_{\tx} \circ \phi_{x_0}^{-1} = 
	\rho_{\tx} \circ g'(\tiu) 
	\circ
	(\rho_{x_0} \circ g'(o))^{-1} .
	\label{psicompo2}
	\eea 
	This is a composition of linear maps 
	(defined on $\rho_{x_0}\circ g'(o)(U)$),
	and therefore extends to
	a linear map defined on the whole of $\R^d$ (and also
	denoted $\phi_\tx \circ \phi^{-1}$).

	Note $g(o)=x_0$.
	 By the continuity in $\tiu$
	 of $\rho_{g(\tiu)} \circ L_\tiu$, the matrix 
	 $(\rho_{g(\tiu)} \circ L_\tiu) \circ (\rho_{g(o)} \circ L_o)^{-1}$
	 is close to the identity matrix $I_d$, so all of
	 the entries of 
	 $(\rho_{g(\tiu)} \circ L_\tiu) \circ (\rho_{x_0} \circ L_o)^{-1} -I_d$
	 can be taken to be close to  zero, by taking $U$ small
	 enough. Hence by (\ref{psicompo2}), by taking $U$ to
	 be small enough we can arrange
	 that $\phi_{\tx} \circ \phi_{x_0}^{-1}$ is close
	 to the identity for all $\tx \in U$,
	 which yields  part (b).

	Given measurable $F \subset V$ and $\tx \in F$,
	we have $v(F)= \int_{g^{-1}(F)} D_g(w) dw$ by (\ref{int-form}).
	Also setting $\tiu = \phi^{-1}(\tx)$,
	we have that 
	$\lambda_{d}(\phi_{\tx}(F))
	=
	\lambda_{d}(\psi_{\tx}(F))
	= \int_{g^{-1}(F)} D_g(\tiu) dw$,
	and 
	$$
	\sup \left\{
		\left| \left( \frac{ D_g(w) }{D_g(\tiu)}
		\right)  -1 \right|:  
		w \in g^{-1}(F)
		\right\} 
	=  
	O(\diam(F)),
	$$
	because $D_g(w)$ is continuously differentiable in $w$ on 
	the convex set $U$, so by the Mean Value theorem
	  $w \mapsto D_g(w)$ is 
	Lipschitz on $U$, and is bounded away from zero. Thus,
	$|(\lambda_d(\phi_{\tx}(F))/v(F))  -1 | < K \diam(F)$.
	This yields part (c).

	Now suppose $F \subset \partial A \cap V $ and $\tx \in F$;
	again set $\tiu = g^{-1}(\tx)$. Then
	by (\ref{surf-form}),
	$\tv(F)= \int_{g^{-1}(F)} \tD_g(w) \lambda_{d-1}(dw)$. 
	Also 
	$\lambda_{d-1}(\phi_{\tx}(F)) =
	\lambda_{d-1}(\psi_{\tx}(F)) =
	\int_{g^{-1}(F)} \tD_g(\tiu) 
	\lambda_{d-1}(dw)$,
	and
	$$
	\sup \left\{
		\left| \left( \frac{ \tD_g(w) }{\tD_g(\tiu)}
		\right)  -1 \right|:  
		w \in g^{-1}(F) 
		\right\} 
	=  
	O(\diam (F)),
	$$
	because $\tD_g(w)$ is continuously differentiable and hence
	Lipschitz on $U \cap \partial \bH$, and is bounded away from zero,
	so that
	$|(\lambda_{d-1}(\phi_{\tx}(F))/\tv(F))  -1 | = O(\diam (F))$.
	This yields part (d).

	Now suppose $\gamma$ is a path in $U$, i.e. a
	continuously differentiable function from $[0,1]$ to $U$.
	Then $g(\gamma)$ is a path
	in $V$. 
	Writing $x(s) := g(\gamma(s))$ 
	and (for $j \in [d]$)  $\gamma_j(s) := \langle \gamma(s) , e_j
	\rangle $ (the $j$-th coordinate
	of $\gamma(s)$)
	for $0 \leq s \leq 1$,
	we have for $1 \leq i \leq m$ that $dx_i(s) = 
	\sum_{j=1}^d (\partial g_i/\partial u_j) d\gamma_j(s)$  
	and hence
	$
	|dx(s)| = \sqrt{ (d\gamma)' J_{g,\gamma(s)}^T J_{g,\gamma(s)}
	(d\gamma) },
	$
	where $J_{g,u} = ( \partial g_i/\partial u_j)_{i \leq m, j \leq d}$
	is the Jacobian matrix of $g$, evaluated at $u \in U$. Writing
	$\langle \cdot, \cdot \rangle_u$ for this inner product,
	i.e. $\langle w, \tilde{w} \rangle_u = w' J_{g,u}^T J_{g,u} \tilde{w} $
	for any $w, \tilde{w} \in \R^d$,
	we thus have that the length of the path $g(\gamma)$ is
	\bea
	|g(\gamma)| = \int_0^1 \sqrt{ \langle d\gamma/ds, d\gamma/ds 
	\rangle_{\gamma(s)}
	} \; ds,
	\label{0608a}
	\eea
	while  
	$\phi_{\tx}(g(\gamma))$, 
	the corresponding path in $g'(\tiu)(U)$, has length 
	\bea
	|\phi_{\tx}(g(\gamma))| =
	\int_0^1 \sqrt{ \langle d\gamma/ds, d\gamma/ds \rangle_{\tiu}
	} \; ds.
	\label{0608b}
	\eea
	For $y, z \in V$, $\dist(y,z)$ is the infimum over all
	such paths with $g(\gamma(0))=y$ and $g(\gamma(1))=z$,
	of $|g(\gamma)|$, while
	$\|\phi_{\tx}(y)- \phi_{\tx}(z)\|
	$
	is the infimum over all such paths, of $|\phi_{\tx}(g(\gamma))|$.

	We assert that if $K$ is taken sufficiently large 
	and $U$ sufficiently small, we have for 
	all $w,\tw \in U$  and any unit vector $e \in \bbS^{d-1}$,
	that
	\bea
	\left|
	\frac{ \langle e ,e \rangle_{w}  }{
		 \langle e ,e \rangle_{\tw}  }
	 -1
	\right|  \leq K \|w-\tw\|.
	\label{0607a}
	\eea
	We work towards proving (\ref{0607a}).
	Given $e \in  \bbS^{d-1}$ we have $\langle e,e \rangle_o >0$
	since $g'(o)$ has full rank. Also the mapping
	$(f,w) \mapsto \langle f,f \rangle_w$ is continuous on
	$\bbS^{d-1} \times U$ (since our manifold is $C^2$ so the relevant
	Jacobian matrices vary continuously with $w$), so
	given $e \in \bbS^{d-1}$, there exists 
	an open neighbourhood $\cN_e$
	in $\bbS^{d-1}$ of $e$,
	and open  $U_e$ with $o \in U_e \subset U$,
	and $\delta_e >0$
	such that $\delta_e \leq \langle f, f\rangle_w \leq \delta_e^{-1}$ for
	all $f \in \cN_e$ and all $w \in U_e$. Hence by a compactness
	argument,
	there exist $\delta >0$ and an  open ball
	$U'$ with $ o \in U' \subset U$,
	such that 
	\bea
	\delta \leq \langle f,f \rangle_w \leq \delta^{-1} , ~~~ \forall f \in 
	\bbS^{d-1}, ~ w \in U'.
	\label{0615a}
	\eea

	Next, for all $u \in U'$,
	$e \in \bbS^{d-1}$ we have $\langle e, e\rangle_{u}
	= e' J_{g,u}^T J_{g,u} e$.
	For  $i \leq m, j \leq d$ let
	$h_{i,j}(u) := \sum_{k=1}^m 
	( \frac{\partial g_k}{\partial u_i})
	( \frac{\partial g_k}{\partial u_j})$, which is
	the $(i,j)$-entry of $J_{g,u}^T J_{g,u}$, and is
	a continuously differentiable function of $u$.
	Then by the Mean Value theorem and the convexity of  $U'$, 
	for any $w,\tw \in U'$ we have
	$$
	h_{i,j}(w ) - h_{i,j}(\tw) = \langle \nabla h_{i,j}(v_{i,j})
	, w-\tw \rangle
	$$
	for some $v_{i,j} \in [w,\tw]$, so
	that for a suitable constant $K_{i,j}$
	we have
	$$
	| h_{i,j}(w) - h_{i,j}(\tw) | \leq  K_{i,j} \|w-\tw\|.
	$$
	Hence there exists  $K >0$ such that if
	$v= (v_1,\ldots,v_d)' \in \R^d$ with $\|v \| \leq 1$ then
	\bean
	|\langle v,v \rangle_w
	- \langle v,v \rangle_{\tw}
	|
	= \left| \sum_{i=1}^d \sum_{j=1}^d v_i  (h_{i,j}(w) - h_{i,j}(\tw) )
	 v_j \right| \leq K  \delta \|w-\tw\|, 
	\eean
	and then using (\ref{0615a}), we obtain 
	the claim (\ref{0607a}).

With $\gamma$ denoting a path in $U'$ as before,
	we have 
	$
	|\gamma| = \int _0^1 \sqrt {
		\langle d \gamma /ds, d\gamma/ds \rangle } ds.
	$
	Comparing this with (\ref{0608a}) and using (\ref{0615a}),
	we may deduce that 
	$$
	\delta | g(\gamma) |\leq |\gamma| \leq \delta^{-1} |g(\gamma)|.
	$$
	Set $V' := g(U')$. 
	Let $y,z \in V'$. Since $\| g^{-1}(y)-g^{-1}(z)\| $ is the infimum of
	$|\gamma|$ over all  paths $\gamma$ from $g^{-1}(y)$ to $g^{-1}(z)$, 
	\bea
	\delta \dist(y,z) \leq \|g^{-1}(y) - g^{-1}(z) \|
	\leq \delta^{-1} \dist ( y,z).
	\label{1009a}
	\eea
	Hence there is a further constant $K'$ such that
	for $w, \tw \in U'$
	the right side of (\ref{0607a}) is bounded by
	$K' \dist(g(w),g(\tw))$; in particular, for any  $e \in \bbS^{d-1}$
	and $w,\tw \in U'$,
	we have
	$$
	\langle e,e \rangle_w \leq (1+ K'\dist(g(w),g(\tw))) 
	\langle e,e \rangle_{\tw},
	$$
	and therefore since $(1+s)^{1/2} \leq 1+s/2$ for
	all $s >0$, 
	\bea
	\sqrt{ \langle e,e \rangle_w } \leq 
	(1+ K' \dist(g(w),g (\tw)))
	\sqrt{ 
	\langle e,e \rangle_{\tw} }.
	\label{0608c}
	\eea

		For any $\tx,y,z \in g(U')$,
	any path $\gamma$ in $U'$  with $g(\gamma(0))=y$ and
	$g(\gamma(1))=z$, and with $\dist(g(\gamma (s)),y) \leq 
	\dist(z,y)$ for all  $s \in [0,1]$,
	for all such $s$ we have 
	$$
	\dist(g(\gamma(s)), \tx) \leq \dist(g(\gamma(s)),y)
	+ \dist (y,\tx) 
	\leq \dist (z,y) + \dist (y,\tx).
	$$
	Therefore by setting 
		$w= \tiu = g^{-1}(\tx) $ and $\tw= \gamma(s) $
	in (\ref{0608c}),
	\bean
	\sqrt{ 
	\langle e,e \rangle_{\tiu} }
	\leq 
	(1+ K' [
		\dist (z,y) + \dist(y,\tx)
		])
	\sqrt{ \langle e,e \rangle_{\gamma(s)} } 
	.
	\eean
	Hence by (\ref{0608a}) and (\ref{0608b}), 
	$$
	|\phi_{\tx}(g(\gamma)) | 
	\leq 
	(1+ K_3 [ \dist (z,y) + \dist(y,\tx) ])
	|g(\gamma)|
	,
	$$
	and taking the infimum over all such paths, 
		using the fact that no geodesic from $y$ to
	$z$ passes through any point distant further than 
	$\dist(z,y)$ from $y$, yields
	that
	\bea
	\| \phi_{\tx}(y) - \phi_{\tx}(z)\|
	\leq
	(1+ K_3 [\dist(\tx,y) + \dist (y,z) ])
	\dist(y,z)
	. 
	\label{0608d}
	\eea
	To derive an inequality the other way, note that
	with $\tiu := g^{-1}(\tx)$, we have
	\begin{align}
	\|\phi_{\tx}(z) - \phi_{\tx}(y) \|  
		& = \|\psi_{\tx}(z) -\psi_{\tx}(y) \| = \|g'(\tiu)( g^{-1}(z)
	-g^{-1}(y))\|
	\nonumber \\
		& = \|J_{g,\tiu} (g^{-1}(z) - g^{-1}(y))\|
	\nonumber \\
		& = \sqrt{\langle g^{-1}(z) -g^{-1}(y), g^{-1}(z) - g^{-1}(y)
	\rangle_{\tiu}}.
	\label{1009b}
	\end{align}
	Suppose now that $\gamma$ is the straight-line path from
	$g^{-1}(y)$ to $g^{-1}(z)$, i.e. $\gamma(s) = 
	g^{-1}(s) + s(g^{-1}(z) - g^{-1}(y))$, for $0 \leq s \leq 1$.
	Then for all $s \in [0,1]$, we have
	$d \gamma/ ds = g^{-1}(z) - g^{-1}(y)$. Hence by (\ref{0608a}),
	\bea
	|g(\gamma)| = \int_0^1 \sqrt{
		\langle g^{-1}(z) - g^{-1}(y),g^{-1}(z) -g^{-1}(y)
	\rangle_{\gamma (s)} } ds,
	\label{1009c}
	\eea
	and moreover, for all $s \in [0,1]$,
	\begin{align*}
	\| \gamma(s) - \tiu \| 
	& \leq \|\gamma(s)-g^{-1}(y) \| + \|g^{-1}(y)- \tiu\|
		\\
	& \leq \|g^{-1}(z) - g^{-1}(y) \| + \|g^{-1}(y) - \tiu\|.
	\end{align*}
	Hence for any $e \in \bbS^{d-1}$, 
	by (\ref{0608c}) and (\ref{1009a}) we have
	\begin{align*}
	\sqrt{ \langle e,e \rangle_{\gamma(s)} }
		& \leq (1+ K' \dist (g(\gamma(s)),\tx)) \sqrt{\langle e,e\rangle_\tiu}
	\\
		& \leq (1+ K' \delta^{-1} \|\gamma(s)-\tiu\| ) \sqrt{\langle e,e\rangle_\tiu}
	\\
		& \leq (1+ K' \delta^{-1} (\|g^{-1}(z)- g^{_1}(y) \| 
	+ \|g^{-1}(y) -\tiu\|) ) \sqrt{\langle e,e\rangle_\tiu}\; .
	\end{align*}
	Using (\ref{1009a}) once more yields that
	\bean
	\sqrt{ \langle e,e \rangle_{\gamma(s)} }
	\leq (1+ K' \delta^{-2} 
	(\dist (z,y) + \dist(y, \tx))) 
	\sqrt{\langle e,e\rangle_\tiu} \; .
	\eean
	Thus, comparing (\ref{1009b}) with (\ref{1009c}) yields
	for this choice of $\gamma$ that
	\bean
	|g(\gamma)| 
	\leq (1+ K' \delta^{-2} (\dist (z,y) + \dist(y, \tx))) 
	\|\phi_{\tx}(z) - \phi_{\tx}(y) \|,
	\eean
	and therefore since $\dist(y,z) \leq |g(\gamma)|$,
	$$
	\dist(y,z) 
	\leq (1+ K' \delta^{-2} (\dist (z,y) + \dist(y, \tx))) 
	\|\phi_{\tx}(z) - \phi_{\tx}(y) \|.
	$$
	Together with (\ref{0608d}), this yields part (e).

	 Finally, since we can take $\diam(V)$ to be as small
	 as we please, we can clearly deduce part (f) from part (e).
\end{proof}

\begin{lemm}
	\label{lemcompare}
	Let $x \in \M$, $\eps >0$. Then there exists an open set $V \subset \M$,
	with $x \in V$, and an open map $\phi: V  \to \R^d$,
	with $\phi(x) = o$, such that for all distinct 
	$y,z \in V$ we have $|(\|\phi(y)-\phi(z)\|/ \dist(y,z)) - 1| \leq \eps$,
	and also for all measurable $ F \subset V$ we have 
	$|(\lambda_d(\phi(F)) / v(F))-1| \leq \eps$, and moreover 
	(i) if $x \in A^o$ then $V \subset A$, while
	(ii) if $x \in \partial A$, then
	$\phi(V \cap A) = \phi(V) \cap \bH$, and also
	for  all measurable $F \subset \partial A \cap V$,
	$|(\lambda_{d-1}(\phi(F))/\tilde v(F))-1| < \eps$.
\end{lemm}
\begin{proof}
	Apply Lemma \ref{lemcompare3}, taking $x_0= x$,  and
	take 
	$\phi = \phi_x$.
\end{proof}

\begin{remk}
	\label{rk:geovseuc}
We now describe  how to adapt Lemma \ref{lemcompare3}
to the case where the distance between $x$ and $y$,
for $x,y \in \cM$, is given by $\|x-y\|$ (the Euclidean
distance between them in $\R^m$) rather than by
$\dist(x,y)$ (the manifold distance between them). 
It is then possible to adapt the proofs of our main results to this
	modified setting, as asserted in Remark \ref{rmkeuc}.

Let $x \in \M$ and let
$(U,g)$ be a chart with
$ V := g(U)$ open (so $U \subset \R^d$ is open) and $x \in V$,
and with $V$ having the properties in Lemma \ref{lemcompare},
taking $\eps = 1/9$. Assume also that $U$ is convex.

Suppose $y,z, \in V$. As stated just after (\ref{0608b}), 
the manifold distance $\dist(y,z)$ is the infimum of $|g(\gamma)|$
over all paths $\gamma$ from $g^{-1}(y)$ to
$g^{-1}(z)$, provided $y$ and $z$ are further from $\partial V$ than
from each other.
Note that for such paths $\gamma$, $g(\gamma)$ is a path
in $\cM$ from $y$ to $z$. 
On the other hand $\|y-z\|$ is the infimum of
$|\tgamma|$ over all paths $\tgamma$ in $\R^m$ from $y$ to $z$, including
paths that stay in $\cM$ as well
as paths that do not. Therefore  $\|y-z\| \leq \dist(y,z)$.

For an inequality the other way, we consider a particular
path in $\cM$ from $y$ to $z$ and show that it is
not much worse than the straight-line path
from $y$ to $z$ in $\R^m$. The particular path 
we consider  is the image under $g$ of
the straight-line path $\gamma_0$ in $\R^d$ from 
$u $ to $v$, where we set
$u:=g^{-1}(y)$ and $v:=g^{-1}(z)$.
%
That is, set $\gamma_0(t) := u + t(v-u)$,
for $0 \leq t \leq 1$.
Then by definition $\dist(y,z) \leq |g(\gamma_0)|$.
Also
$d \gamma_0/d t =
(v-u)$ for all $t \in [0,1]$, so by (\ref{0608a}),
\bea
|g(\gamma_0)| = \int_0^1 \sqrt{ (v-u)' J_{g,\gamma_0(t)}^T J_{g,\gamma_0(t)}
(v-u) } \; dt.
\label{0825a}
\eea

We need to compare this with $\|g(v)-g(u)\|$.
	Let $[u,v]$ denote the line segment from $u$ to $v$ in 
	$\R^d$ (i.e., the convex hull of $\{u,v\}$).
	For $1 \leq i \leq m$, let $g_i(\cdot)$ denote the $i$th component of
	$g(\cdot)$.
	By the Intermediate Value theorem
there exists $w_i \in [u,v]$ such that
$g_i(v) - g_i (u) = \langle \nabla g_i(w_i) , (v-u) \rangle$.
Therefore $g(v)-g(u) = \tJ (v-u)$,
where $\tJ$ is the matrix $(\partial g_i /\partial u_j|_{w_i})_{i \leq m,
j \leq d}$. Hence
$$
\| g(v) - g(u) \|
= \sqrt{ (v-u)' \tJ^T \tJ (v-u) }.
$$
Therefore by (\ref{0825a}), setting $e = \|v-u\|^{-1} (v-u)$,
we have that
\bea
|g(\gamma_0)| -
\| g(v) - g(u)\| = \|v-u\|
\int_0^1 
\left(
\sqrt{e' J_{\gamma_0(t)}^T J_{\gamma_0(t)} e} - \sqrt{e' \tJ^T \tJ e } \right) dt.
\label{0825b}
\eea
By the continuous differentiability in $u$ of $\partial g_i /\partial u_j$,
and the Mean Value theorem,
there exists a constant $K$ such that
for each $i \leq m,j \leq m$ and $t \in [0,1]$ we have
\bean
|(J_{\gamma_0(t)} - \tJ)_{i,j} |  = |(
\partial g_i /\partial u_j|_{\gamma_0(t)} -
\partial g_i /\partial u_j|_{w_i} ) | \leq 
K \|u-v\|.
\eean
Thus
there is a further constant $K'$ such that 
for all $e \in \bbS^{d-1}$ and $t \in [0,1]$ we have
$$
|(\|J_{\gamma_0(t)} e \| - \|\tJ e\| )| \leq
\|(J_{\gamma_0(t)} - \tJ) e\| \leq K' \|u-v\|.
$$
Therefore since 
$\|J_{\gamma_0(t)} e\| - \|\tJ e\|$ is the integrand
in (\ref{0825b}), we obtain that 
\begin{align}
\dist(y,z) \leq |g(\gamma_0)| \leq \|y-z\| + K' \|v-u\|^2.
	\label{0811a}
\end{align}
By a further application of the Mean Value theorem,
there is  further constant $K''$ such that
\begin{align*}
	\|y-z - g'(u)(v-u) \| \leq K''  \|v-u\|^2.
\end{align*}
By the compactness of $\bbS^{d-1}$ there exists $\delta >0$ such
that 
$\|g'(o) e\| \geq 3 \delta $ for all $e \in \bbS^{d-1}$,
and by a further compactness argument, provided $U$ is taken 
small enough we have $\|g'(u) (e) \| \geq 2 \delta $
for all $u \in U, e \in \bbS^{d-1}$, so that (possibly
after taking $U$ even smaller) we have
$
\|y-z\| \geq \delta \|v-u\|
$
and hence by \eqref{0811a},
\bean
1 \leq \frac{\dist(y,z)}{\|y-z\|} 
\leq 1 +  K' \delta^{-2}\|y-z\|.
\leq 1 +  K' \delta^{-2}\dist(y,z).
\eean
This enables us to obtain parts (e) and (f) of
Lemma \ref{lemcompare3} with $\dist(y,z) $ replaced
by $\|y-z\|$.
\end{remk}

\section{Proof of strong laws of large numbers}
\allco


Throughout this section we are assuming we are  given a 
constant $\sbeta \in [0,\infty]$
and a sequence $(k(n))_{n \in \N}$ satisfying (\ref{kcond}). 
Recall that $\mu$ denotes the distribution of $X_1$,
and this has a density $f$ with support $A$, where
$A$ is a compact submanifold-with-boundary of $\cM$, and that $B \subset A$
is fixed and closed, and $R_{n,k}$ is defined at (\ref{Rnkdef}).

We shall repeatedly use the following  lemma. 
%
%
%

\begin{lemm}[Subsequence trick]
\label{lemtrick}
Suppose $(U_{n,k},(n,k) \in \N \times \N)$ is an array of random 
variables on a common probability space such that $U_{n,k} $ is
nonincreasing in $n$ and nondecreasing in $k$, that is, 
$U_{n+1,k} \leq U_{n,k} \leq U_{n,k+1} $ almost surely,
for all $(n,k) \in \N \times \N$. Let $\sbeta \in [0,\infty), \eps >0, c  >0$,
and suppose $(k(n),n \in \N)$
is an $\N$-valued sequence such that 
 $k(n)/\log n \to \sbeta$ as $n \to \infty$. 

	(a) Suppose
$
\Pr[n U_{n, \lfloor (\sbeta + \eps) \log n \rfloor } > \log n] \leq 
c n^{-\eps},
$
	for all but finitely many $n$. 
Then $\Pr[ \limsup_{n \to \infty} n U_{n,k(n)} /\log n \leq 1] =1$.

	(b)
	Suppose  $\eps < \sbeta $ and
$
\Pr[n U_{n, \lceil (\sbeta - \eps) \log n \rceil } \leq \log n] \leq 
c n^{-\eps},
$
	for all but finitely many $n$. 
Then $\Pr[ \liminf_{n \to \infty} n U_{n,k(n)} /\log n \geq 1] =1$.
\end{lemm}
\begin{proof} See \cite[Lemma 6.1]{ECover}. 
\end{proof}

\subsection{{\bf General lower and upper bounds}}
\label{subsecgenbds}


In this subsection 
we present 
asymptotic lower and upper
bounds on $R_{n,k(n)}$, not requiring 
any extra
assumptions on $A$ and $B$. In fact, $A$ here can be any metric space
endowed with a probability  measure $\mu$, and $B$ can
be any subset of $A$. The definition of $R_{n,k}$ at (\ref{Rnkdef})
carries over in an obvious way to this general setting.

%

Later, we shall derive the results
stated in 
Theorem \ref{thm2} and Proposition \ref{thm1}
by  applying the results of this subsection to the different 
regions within $A$ (namely interior and boundary).

Given $r >0, a>0$,   define the `packing number' $ \nu(B,r,a) $
 be the largest number $m$ such that there exists a collection of $m$
 disjoint closed geodetic balls of radius $r$ centred on points of $B$,
each with $\mu$-measure at most $a$.

Recall that $H(\cdot)$ and $\hat{H}(\cdot)$ were defined at \eqref{Hdef}
and \eqref{hatHdef}.

\begin{lemm}[General lower bound]

	\label{gammalem}
	Let $a >0, b \geq 0$. Suppose
	$\nu(B,r,a r^d) = \Omega (r^{-b})$ as $r \downarrow 0$.
	If $\sbeta = \infty$ then
	$\liminf_{n \to \infty} \left(n  R_{n,k(n)}^d/k(n) \right) \geq
	1/a$, almost surely. If $\sbeta < \infty$ then
	$\liminf_{n \to \infty} \left(n  R_{n,k(n)}^d/\log n \right) 
	\geq a^{-1} \hH_\sbeta(b/d)$, almost surely.

\end{lemm}
\begin{proof}
	See \cite[Lemma 6.2]{ECover}.
\end{proof}

Given  $r>0$, and $D \subset A$,
define the `covering number' 
\bea
\kappa(D,r): = \min \{m  \in \N: \exists x_1,\ldots,x_m \in D
~{\rm with} ~ D \subset \cup_{i=1}^m B(x_i,r) 
\}.
\label{covnumdef}
\eea
We need a complementary upper  bound.
For this,  we shall require a condition on the `covering number'
that is roughly dual to the condition on `packing number'
 used in Lemma \ref{gammalem}.
Instead of stating the lemma in terms  
of $R_{n,k}$ directly, we shall
state it in terms of 
the `fully covered' region 
$F_{k,r}(\X)$,
defined for 
$n,k \in \N$, $r>0$ and $\X \subset \cM$
by
\bea
F_{k,r}(\X):= \{x \in \cM : \card( \X \cap B(x,r)) \geq k\}.
\label{F3def}
\eea
 We can characterise the event
 $\{\tR_{n,k} \le r\} $ in terms of the set $F_{k,r}(\X_n)$
 as follows:
\bea
\tR_{n,k } \leq r {\rm~~~ if ~and~only~if~~~}
(B \cap A^{(r)} )
\subset F_{k,r}(\X_n).
\label{Fnequiv}
\eea
Indeed, the `if' part of this statement is clear from (\ref{eqmaxspac}).  
For the `only if' part, note that  if there exists $x \in (B \cap A^{(r)}) 
\setminus F_{k,r}(\X_n)$, then there exists $s >r$ with $x \in B \cap 
A^{(s)} \setminus F_{k,s}(\X_n)$. Then for all $s' <s$ we have
$x \in B \cap A^{(s')} \setminus F_{k,s}(\X_n)$, and therefore
$\tR_{n,k} \geq s > r$.


\begin{lemm}[General upper bound]
\label{lemmeta}
Suppose $r_0, a, b \in (0,\infty)$,  and
 a family of sets  $A_r \subset A,$ defined for $0 < r < r_0$, 
	are
	 such that  for all $r \in (0,r_0)$, $x \in A_r$ and
 $s \in (0,r)$ we have $\mu(B(x,s)) \geq a s^d$, and
moreover $\kappa(A_r,r) = O(r^{-b})$ as $r \downarrow 0$.

	(i) If $\sbeta = \infty$  
and  $r_n = (\nalpha k(n)/n)^{1/d}$, $ n \in \N$,
	for some fixed  $\nalpha > 1/a$,
then with probability one,  
	$A_{r_n} \subset F_{k(n),r_n}(\X_n)$ 
for all large enough $n$.

	(ii) If $\sbeta < \infty$, 
	and $r_n = (\nalpha (\log n)/n)^{1/d}$, $n \in \N$,
	for some fixed $\nalpha >a^{-1} \hH_\sbeta(b/d)$.
 then there exists $\eps >0$ such that
$
\Pr[\{A_{r_n} \subset
	F_{\lfloor (\sbeta + \eps) \log n\rfloor,r_n
	}(\X_n) \}^c ]
 = O(n^{-\eps})
$
as $n \to \infty$.

	(iii)
	If $r_t = (1+o(1))((u\log t)/t)^{1/d}$ as $t \to \infty$,
for some fixed
	$\nalpha> b/(ad)$, then given $k \in \N$,
	we have
$\Pr[\{A_{r_t}  \subset F_{k,r_t}(\Po_t)
	\}^c] \to 0 $
	as $t\to \infty$.
\end{lemm}
\begin{proof}
	See \cite[Lemma 6.3]{ECover} for (i) and (ii).
	The proof for (ii) 
	carries over
	to (iii) 
	by using \cite[Lemma 5.1(d)]{ECover} in
	place of \cite[Lemma 5.1(b)]{ECover}.
\end{proof}

	We shall use parts (i) and (ii) of Lemma \ref{lemmeta}
	to prove an upper bound for the strong LLN; we shall
	use (iii)
	later on to
	handle various boundary effects in the proof of the weak limit 
	result, with properly chosen $A$ and $A_r$.

\subsection{Proof of Proposition \ref{thm1}} 

Recall from Section \ref{s:MathFrame}
that we take closed $B \subset A$ with $A$ a compact submanifold-with-boundary
of the $d$-dimensional manifold $\cM$, 
 and from \eqref{f0def} that
 $f_0: = \inf_{x \in B} f(x)$.
 Let $\mu$ denote the probability measure on $A$ with density $f$.
 Assume, as in the statement of Proposition \ref{thm1}, that
 $B = \overline{B^o}$, $\mu(B) >0$ and $f|_A$ is continuous on $A$.

\begin{lemm}[Packing number away from $\partial A$]
\label{packlem}
Let $\alpha > f_0$.  There exists $\delta>0$ such that
	\begin{align} 
	\liminf_{r \downarrow 0} r^d \nu(B \cap A^{(\delta)}, r, \alpha \vvol_d r^d)
		> 0.
		\label{e:pack}
	\end{align}
\end{lemm}

\begin{proof}
Let $\kappa = 4^{-d} \vvol_d$. 
	Choose $\eps \in (0,\frac19)$ 
	such that $(1-2 \eps)^{d+1}\alpha > f_0$.
	
	By the condition $B= \overline{B^o}$ and the continuity of $f$,
	we can and do choose $x_0 \in B^o \subset A^o$ such that
	$f(x_0) < (1-2 \eps)^{d+1}\alpha $. Choose $\delta >0$ such
	that $x_0 \in A^{(2 \delta)}$.
%

	Let $(V,\phi)$ be as in 
%
	Lemma \ref{lemcompare}, with $x_0 \in V$.
%
%
%
	Let $r_0 \in (0,\delta/2)$ be chosen such that  with
	$B_0:= B_{\R^d}(\phi(x_0),r_0)$, 
	we have $B_0 \subset \phi( B^o \cap A^{(\delta)} \cap V)$
	and moreover $f(x) < (1-2 \eps)^{d+1} \alpha
	$ for all $x \in \phi^{-1}(B_0)$

	For $r>0$ sufficiently small, we can and do take 
	a collection of disjoint closed Euclidean balls
	$B_{r,j}:= B_{\R^d}(x_{r,j},r(1+\eps)), 1 \leq j \leq \sigma(r)$,
	all contained in $B_0$,
	with $\sigma(r) >  \kappa (r_0/r)^d$.
	Indeed, we can take $(1+ o(1))\vvol_d r_0^d (3r)^{-d}$ 
	[as $r \downarrow 0$] disjoint cubes of side $3r$ inside $B_0$ 
	and can take $x_{r,j}$, $1 \leq j \leq \sigma(r)$,
	to be the centres of 
	these cubes.

	For each $j \in [\sigma(r)]$, we have
	$B_{r,j} \subset B_0$
	so $f(x) < (1-2 \eps)^{d+1}\alpha$ for all $x \in \phi^{-1}(B_{r,j})$,
	and also $B_{r,j} \subset \phi(B \cap A^{(\delta)} \cap V)$.
	Let $B'_{r,j} := B(\phi^{-1}(x_{r,j}),r) \subset \cM$.
	Then $B'_{r,j} \subset \phi^{-1}(B_{r,j})$ 
		by Lemma \ref{lemcompare}. 
	Therefore the balls $B'_{r,j}, 1 \leq j \leq \sigma(r)$ are disjoint.
	Also for $1 \leq j \leq \sigma(r)$, 
	$\phi^{-1}(x_{r,j}) \in B \cap A^{(\delta)}$ and
	\begin{align*}
	\mu(B'_{r,j})
		\leq \mu(\phi^{-1}(B_{r,j})) 
		\leq (1- 2\eps)^{d+1} \alpha v(\phi^{-1}(B_{r,j})).
	\end{align*}
	Also by Lemma \ref{lemcompare} again $\lambda_d(B_{r,j}) \geq
	(1- \eps) v(\phi^{-1}(x_{r,j}))$, so
	\begin{align*}
	\mu(B'_{r,j})
		\leq (1-2 \eps)^{d+1} \alpha \vvol_d r^d(1+2\eps)^{d+1} \leq \alpha \vvol_d r^d.
	\end{align*}
	Hence $\sigma'(r) \leq \nu(B \cap A^{(\delta)}, r,
	\alpha  \vvol_d r^d)$ 
	for $r$ sufficiently small, and the result follows.
\end{proof}

\begin{lemm}[Asymptotic a.s. lower bound for $\tR$]
\label{lemliminf}
	If $(k(n))_{n \geq 1}$ satisfies \eqref{kcond}, then
	\begin{align}
\Pr[ \liminf_{n\to\infty}(n \vvol_d \tR_{n,k(n)}^d/k(n) ) \geq 1/f_0]=1
	~~~~~{\rm if} ~\sbeta = \infty;
\label{0328a}
\\
 \Pr[  \liminf_{n\to\infty} 
	( n \vvol_d \tR_{n,k(n)}^d /\log n) \geq \hH_\sbeta(1)/f_0  ] =1
	~~~~~{\rm if} ~\sbeta < \infty.
\label{liminf1}
	\end{align}
\end{lemm}
\begin{proof}
	Let $\alpha>f_0$.  
	Choose $\delta >0$ such that \eqref{e:pack} holds.
	Define $r_n$ by $nr_n^d =(\alpha\vvol_d)^{-1} k(n)$ if $\sbeta= \infty$,
	and $nr_n^d = (\alpha \vvol_d)^{-1} \hat H_\sbeta(1) \log n$ if
	$\sbeta<\infty$. In either case  $B \cap 
	A^{(\delta)}\subset A^{(r_n)}\cap B$
	for all $n$ large,
	so by monotonicity it suffices to show the lower bound for
	the $k(n)$-coverage threshold of $B \cap A^{(\delta)}$. 
	The condition of Lemma \ref{gammalem} holds with $b=d$ and $a= \alpha \vvol_d$.
	Since $\alpha>f_0$ is arbitrary, the result follows.  
\end{proof}
\begin{lemm}[Covering number of $A$]
	\label{lemcovering}
	We have that
	(i) $\limsup_{r \downarrow 0} (r^d \kappa(A, r)) < \infty$,
	and (ii) 
	$\limsup_{r \downarrow 0} (r^{d-1} \kappa( A\setminus A^{(r)}, r)) < \infty$.
\end{lemm}
\begin{proof}
	Using Lemma \ref{lemcompare} and the
	assumed compactness of $A$,
	let $(x_i,\phi_i,V_i)$, $1 \leq i \leq k$ be a finite collection
	of triples as in Lemma \ref{lemcompare}, taking $\eps =1/2$ there,
	with $A \subset \cup_{i=1}^k V_i$, and moreover
	with either $x_i \in \partial A$
	or $\overline{V_i} \subset A^o$ for each $i \in [k]$.
	Then for each $i$, $\phi(A \cap V_i)$ is a bounded set in 
$\R^d$ and can be covered by a collection of $O(r^{-d})$ Euclidean balls 
	$B_{i,j}$ of radius $r/2$. Then  $\phi^{-1}(B_{i,j})$ is contained in
	a geodetic  ball  of radius $(3/2)r/2$ in $V_i$, so 
	$A \cap V_i$ can be covered by $O(r^{-d})$ geodetic
	balls of radius $r$, and (i) follows.

	For each $i \in [m]$,
	since $\phi(\partial A \cap  V_i)$
	is a bounded set in
	$\partial \bH$,
	it can be covered
	 by $O(r^{1-d})$ balls of radius $r/2$. Then we can deduce
	 (ii) 
	 similarly to (i).
\end{proof}

\begin{lemm}[Estimating the volume within $A$ of small geodetic balls]
	\label{lemMyBk}
	Given $\eps >0$,
	there exists $\delta >0$ such that 
	 for all $s\in (0,\delta)$ 
	 we have: 
	\begin{align}
		v(B(y,s)\cap A) & > (1-\eps) \vvol_d s^d ~~~~
		\forall ~ y \in A^{(s)}; \label{e:volLB1} \\
	v(B(y,s) \cap A)  & > (1- \eps) \vvol_d s^d/2 ~~~~ \forall y \in A;
		\label{e:volLB2} \\
		v(B(y,s) \cap A)  & < (1+\eps) \vvol_d s^d/2  
		~~~~ \forall y \in \partial A.
		\label{e:volUB}
	\end{align}
\end{lemm}
\begin{proof} 
	Let $\eps' \in (0,\frac12)$ with $(1+ 2 \eps')^{d+1} < 1+ \eps$
	and $(1- \eps')^{d+1} > 1- \eps$.  By compactness of $A$, 
	we can and do take a finite collection of triples
	$(x_i, V_i,\phi_i), 1 \leq i \leq \ell$ as
	in Lemma \ref{lemcompare} and $\delta_i >0, 1 \leq i \leq \ell$,
	such that $A \subset \cup_{i=1}^{\ell} B(x_i,\delta_i)$
	and for each $i$, 
	$B(x_i, 2 \delta_i) \subset V_i$,
	$B(\phi_i(x_i), 2 \delta_i) \subset \phi(V_i)$,
	and moreover if $x_i \in A^o$ then $B(x_i, 2 \delta_i) \subset A^o$.
	Take $\delta = \min\{\delta_1,\ldots,\delta_\ell\}$.

	Let $y \in A$, $s \in (0,\delta)$,
	and choose $i \in [\ell]$ such that
	$y \in B(x_i,\delta_i)$. Then $B(y,s) \subset B(x_i,2\delta_i )
	\subset V_i$. Also, using Lemma \ref{lemcompare},
	for $u \in B_{\R^d}(\phi_i(y),(1-\eps')s)$ we have
	  $\dist(\phi_i^{-1}(u),y) \leq (1-\eps')^{-1} \|u - \phi_i(y)\| 
	  \leq s$. 
	 Hence 
	 $ B_{\R^d}(\phi_i(y),(1- \eps')s) \subset \phi_i(B(y,s)) $.
	 Thus using Lemma \ref{lemcompare} again, 
	 we have
%
	\begin{align}
		v(B(y,s)\cap A)
		& \geq (1 + \eps')^{-1} \lambda_d(\phi_i(B(y,s) \cap A))
		\nonumber \\
		& \geq (1-\eps') \lambda_d( B_{\R^d}(\phi_i(y),(1-\eps')s)
		\cap \phi_i(A \cap V_i)).
		\label{e:forvolLB}
	\end{align}
	If $y \in A^{(s)}$ then $B(y,s) \subset A$ so that
	$v(B(y,s) \cap A) \geq (1-\eps')^{d+1} \vvol_d s^d$, and hence 
	\eqref{e:volLB1}.
	If $y \in A \setminus A^{(s)}$ then $B(x_i,2 \delta_i) \cap \partial
	A \neq \emptyset $ so $x_i \in \partial A$ and
	$\phi_i(y) \in \bH$, and by \eqref{e:forvolLB},
	$$
	v(B(y,s) \cap A) \geq (1-\eps') \lambda_d(B_{\R^d}(\phi_i(y),(1-\eps')s)
	\cap \bH) \geq (1-\eps')^{d+1} (\vvol_d/2) s^d.
	$$
%
Combined with \eqref{e:volLB1} this yields \eqref{e:volLB2}.
%
	If $y \in \partial A$ then
	$\phi_i(y) \in \partial \bH$ and
	by Lemma \ref{lemcompare},
	\bean
	v(B(y,s) \cap A)
	& \leq & (1-\eps')^{-1} \lambda_d( \phi_i(B(y,s) \cap A))
	\nonumber \\
	& \leq & (1+2 \eps')
	\lambda_d(  B_{\R^d}(\phi_i(y), (1+\eps') s ) \cap \phi_i(V_i \cap A))
	\nonumber \\
	& = & (1+2 \eps')
	\lambda_d(  B_{\R^d}(\phi_i(y), (1+\eps') s ) \cap \bH)
	\\
	& \leq &  (1+2\eps')^{d+1}  (\vvol_d/2) s^d.
	\eean
	This gives us \eqref{e:volUB}.
	\end{proof}

\begin{lemm}[Asymptotic a.s. upper bound for $\tR$]
\label{lemlimsup}
	It is almost surely the case that
	\begin{align}
	\limsup_{n\to\infty} ( n \vvol_d \tR_{n,k(n)}^d /k(n)) & \leq 1/f_0, 
		& {\rm if~} \sbeta = \infty;
	\label{eqsubseq3}
\\
	\limsup_{n\to\infty} ( n \vvol_d \tR_{n,k(n)}^d /\log n) & \leq \hH_\sbeta(1)/f_0 ,
		& {\rm if~}  \sbeta < \infty.
\label{eqsubseq2}
		\end{align} 
	Also if $\sbeta < \infty$, given $\alpha > \hH_\sbeta(1)$,
	there exists $\eps>0$ 
	such that 
	\bea
	\Pr[n \vvol_d f_0 \tR_{n,\lfloor (\sbeta + \eps) \log n \rfloor}^d 
	> \alpha \log n]
	= O( n^{-\eps}) 
	~~~~{\rm as} ~ n \to \infty.
	\label{0910a}
	\eea
\end{lemm}
\begin{proof}
	Let $\eps' \in (0,\frac12)$.  For all small enough $r$ and for
	$x \in B \cap A^{(r)}$,
	using the definition of $f_0$ and (if $B \neq A$) also the
	assumed continuity of $f$ on $A$,
	and then
	\eqref{e:volLB1} in Lemma \ref{lemMyBk},
	we have 
	\begin{align}
		\mu(B(x,r)) \geq
	f_0(1-\eps'/4) v(B(x,r)) \geq (1- \eps'/2) f_0 \vvol_d r^d.
\label{ctyclaim}
	\end{align}

	
	We shall apply Lemma \ref{lemmeta} with $A_r= B \cap A^{(r)}$.
	By Lemma \ref{lemcovering}-(i) and \eqref{ctyclaim}, 
	the conditions of Lemma \ref{lemmeta} hold 
	with $b=d$ and $a= (1-\eps'/2) f_0\vvol_d $.	
	Define $r_n$ by $nr_n^d = uk(n)$ with $u= (1+\eps')(f_0\vvol_d)^{-1} $
	if $\sbeta= \infty$,  and $nr_n^d = u' \log n$ with 
	$u' = (1+\eps')(f_0 \vvol_d)^{-1} \hat H_\sbeta(1) $
	if $\sbeta<\infty$. Then 
	$u >1/a$  and $u' >a^{-1}\hat H_\sbeta(1)$.

	If $\sbeta = \infty$, then by Lemma \ref{lemmeta}-(i), 
	we have almost surely that 
	for all large enough $n$,
	$B \cap A^{((u k(n)/n)^{1/d})} \subset F_{k(n),(uk(n)/n)^{1/d}}(\X_n)$
	and by \eqref{Fnequiv}, $\tR_{n,k(n)} \leq (u k(n)/n)^{1/d}$.
	Since $\eps' \in (0,1)$ is arbitrary this yields 
	\eqref{eqsubseq3}. 

	If $\sbeta < \infty$ then by Lemma \ref{lemmeta}-(ii)
	and \eqref{Fnequiv},
	there exists $\eps >0$ such that
	\begin{align*}
		& \Pr[ \tR_{n,\lfloor (\sbeta + \eps) \log n \rfloor }
		 \leq (u' (\log n)/n)^{1/d}] 
		 \\
		 =
		& 	\Pr[B \cap A^{((u' (\log n)/n)^{1/d})}
	\subset F_{\lfloor (\sbeta + \eps) 
	\log n
	\rfloor
		,( u' (\log n)/n)^{1/d}}(\X_n) ] 
		 = O(n^{-\eps})
		\end{align*}
	and hence   \eqref{0910a}. Then
	\eqref{eqsubseq2} follows from \eqref{0910a} by Lemma \ref{lemtrick}.
\end{proof}

\begin{proof}[Proof of Proposition \ref{thm1}]
	By Lemmas  \ref{lemliminf} and
 \ref{lemlimsup},  (\ref{0617a}) holds if $\sbeta = \infty$ and
(\ref{0315a}) holds if $\sbeta < \infty$.
It follows that almost surely
 $\tR_{n,k(n)} \to 0$ as $n \to \infty$, 
and therefore if also 
	$B \subset A^o$, recalling that $B$ is compact
	we have $\tR_{n,k(n)} = R_{n,k(n)}$ for all
large enough $n$.  Therefore in this case (\ref{0617a}) (if $\sbeta =\infty$)
or (\ref{0315a}) (if $\sbeta < \infty)$ still holds with $\tR_{n,k(n)}$
replaced by $R_{n,k(n)}$.
\end{proof}

\subsection{Proof of  Theorem \ref{thm2}}
\label{secproofstrong}




Recall that we are assuming that
$k(n)/\log n \to \sbeta \in [0,\infty]$ and $k(n) /n \to 0$,
 as $n \to \infty$,
and that
$f_1 := \inf_{B \cap \partial A} f$.

\begin{lemm}[Lower bound for packing number] \label{l:pack_b}
	Suppose that
	$B \cap \partial A = \overline{B_A^o \cap
	\partial A} \neq \emptyset$
	and $f|_{A} $ is continuous on $A$.
For any $\alpha>f_1$, we have 
\bea
	\liminf_{r \downarrow 0} 
	r^{d-1} \nu(B, r,\alpha\vvol_d r^d/2) >0.
	\label{0406a}
	\eea
\end{lemm}
\begin{proof}
Let $\eps \in (0,1/(9d))$. 
	Choose $x \in B_A^o \cap \partial A$
	with $f(x) < f_1 + \eps/4$.
	Then choose $\delta_1 >0$ such that
	$B(x,\delta_1) \cap A \subset B$
	and
	$f(y) \leq f(x) + \eps/2$ for all
	$y \in B(x,\delta_1) \cap A$.
	Then using \eqref{e:volUB} in
	Lemma \ref{lemMyBk}, choose 
	$\delta \in (0, \delta_1/2)$ such that
	\bea
	\mu(B(y, s)) < (f_1+ \eps) \vvol_d s^d/2,
	~~~
	\forall ~(y,s) \in (\partial A \cap B(x,\delta))
	\times (0,\delta).
	\label{0910b}
	\eea

	With $\eps$ and  $x$ as above,
	let $(V,\phi)$ be as in Lemma
	\ref{lemcompare}.
 Let $r \in (0,\delta)$.	Within  $\partial \bH$ consider the lattice
	$\LL_r := 
	( 4r \Z)^{d-1}
	\times 
	\{0\} 
	$.
	Then as $r \downarrow 0$,
	$
	\card (\LL_r \cap \phi(V)) = \Omega(r^{1-d}).
	$
	For distinct $z,z' \in \LL_r \cap \phi(V)$,
	set $y= \phi^{-1}(z),  y'= \phi^{-1}(z')$.
	Then  $\|\phi(y) -\phi(y') \| 
	\geq 4 r$, and therefore
	by Lemma \ref{lemcompare},
	$\dist(y,y') \geq 3r$. Hence the geodetic
	balls $B(y,r), y \in \phi^{-1}(\LL_r \cap \phi( U \cap B(x,\delta)))$
	are pairwise disjoint. Moreover, each such $y$
	lies in $\partial A$, 
	and by (\ref{0910b}), each of these balls
	satisfies $\mu(B(y,r)) \leq (f_1+\eps) \vvol_d r^d /2$.
	This gives us (\ref{0406a}).
\end{proof}

\begin{lemm}[Asymptotic lower bound]
\label{lemliminfb}
	Under the assumptions of Theorem \ref{thm2},
	\begin{align}
		& \Pr[\liminf_{n \to \infty}
\left( n \vvol_d R_{n,k(n)}^d /k( n) \right)
	\geq
 2/ f_1  ] =1, 
		&{\rm if}~
 \sbeta = \infty; 
\label{liminfeq0}
	\\
		& \Pr[\liminf_{n \to \infty}
\left( n \vvol_d R_{n,k(n)}^d /\log n \right) 
\geq
		2 \hH_\sbeta(1-1/d)/ f_1 ] =1, 
		&{\rm if} ~
\sbeta < \infty .
\label{liminfeq}
	\end{align}
\end{lemm}

\begin{proof}
	By Lemma \ref{l:pack_b}, the condition in 
Lemma \ref{gammalem} holds with $b= d-1$ and $a=\alpha \vvol_d/2$ for all $\alpha>f_1$. The result follows
by Lemma \ref{gammalem}.
\end{proof}


\begin{lemm}[Asymptotic upper bound]
\label{lemlimsupb}
	Under the assumptions of Theorem \ref{thm2},
\bea
\Pr[\limsup_{n \to \infty}
\left( n \vvol_d R_{n,k(n)}^d /k( n) \right) \leq 
	\max(1/f_0,2/ f_1)  ] =1, ~~{\rm if}~
 \sbeta = \infty;
	~~~~~~~~~~~~
\label{limsupeq0}
\\
\Pr \left[ \limsup_{n \to \infty}
\left( n \vvol_d R_{n,k(n)}^d /\log n \right) \leq 
	\max \left( \frac{\hH_\sbeta(1)}{f_0}, \frac{2 \hH_\sbeta(1- 1/d) }{ 
	f_1}  \right) 
	\right] =1, ~~ {\rm if} ~ 
\sbeta < \infty .
\nonumber \\
\label{limsupeq}
\eea
\end{lemm}
\begin{proof}
%
	By Lemma
	\ref{lemMyBk} (eq. \eqref{e:volLB2})
	and Lemma \ref{lemcovering}-(ii),
	we see that the condition of
Lemma \ref{lemmeta} holds when we take
	$A_r= B\setminus A^{(r)}$,
	$b= d-1$ and any $a<f_1\vvol_d/2$.
	
	Suppose $\sbeta = \infty$. Let 
	$u > \max(1/(f_0 \vvol_d), 2/(f_1 \vvol_d))$. 
	Then by Lemma \ref{lemmeta}(i),
	we have a.s. that
	for $n $ large enough,
	\begin{align}
	B \setminus A^{((uk(n)/n)^{1/d})} \subset F_{k(n),(u k(n)/n)^{1/d}}
	(\X_n).
		\label{e:0902a}
		\end{align}
		Also, by Proposition \ref{thm1} we have
		a.s. that for $n$ large enough $n \tR_{n,k(n)}^d/k(n)
		\leq u$, so that by \eqref{Fnequiv},
		$$
		B \cap A^{((uk(n)/n)^{1/d}) } \subset 
		F_{k(n),(u k(n)/n)^{1/d}} (\X_n).
		$$
		Combined with \eqref{e:0902a}
		this shows that a.s., for $n$ large enough
		$B\subset
		F_{k(n),(u k(n)/n)^{1/d}} (\X_n)$ and
		hence $R_{n,k(n) } \leq (u k(n)/n)^{1/d}$.
		This gives us \eqref{limsupeq0}.


Now suppose $\sbeta < \infty$, and  let
	$u > \max (\hat{H}_\sbeta(1)/(f_0 \vvol_d),  
	2 \hat{H}_{\sbeta}(1-1/d) / (f_1 \vvol_d))$.
	Then by Lemma \ref{lemmeta}(ii),
	there exists $\delta >0$ such that 
	\begin{align}
		\Pr[ \{B \setminus A^{((u (\log n)/n)^{1/d})} \subset
	F_{\lfloor (\sbeta + \delta) \log n \rfloor, (u (\log n)/n)^{1/d}}
		(\X_n) \}^c] 
	= O(n^{-\delta}).
		\label{0902b}
	\end{align}
	Also by \eqref{0910a} from Lemma \ref{lemlimsup} we also have
	for some $\delta' >0$
	that
	\begin{align*}
	\Pr[ n \tR_{n, \lfloor (\sbeta+ \delta')\log n \rfloor}^d
	> u \log n ] = O(n^{-\delta'}),
	\end{align*}
	so that by \eqref{Fnequiv}, 
	\begin{align}
	\Pr[ \{B \cap A^{((u (\log n)/n )^{1/d})} \subset
	F_{\lfloor (\sbeta + \delta') \log n, (u (\log n)/n)^{1/d}}
	(\X_n) \}^c] =O(n^{-\delta'}).
		\label{0902c}
	\end{align}
	Since
	\eqref{0902b} )(resp. \eqref{0902c})
	still holds if we reduce $\delta$ (resp. $\delta'$),
	taking $\eps = \min(\delta,\delta')$ we have
	$$
	\Pr \big[R_{n,\lfloor (\sbeta+ \eps) \log n \rfloor} \leq
	\big(\frac{u \log n}{n} \big)^{1/d} \big]
	=
	\Pr[ \{ B 
	\subset F_{ \lfloor (\sbeta + \eps) \log n\rfloor, (u (\log n)/n)^{1/d}
	\rfloor } (\X_n) \}^c ]
	= O(n^{-\eps}).
	$$
Hence by Lemma \ref{lemtrick}, 
a.s. for large enough  $n$
we have
$n R_{n,k(n)}^d / \log n \leq u$. This gives us \eqref{limsupeq}.
\end{proof}

\begin{proof}[Proof of Theorem \ref{thm2}]
	Since $R_{n,k} \geq \tR_{n,k}$ for all $n,k$,
	using Proposition 
	\ref{thm1} and Lemma \ref{lemliminfb} gives
	the required lower bound, and using Lemma \ref{lemlimsupb}
	gives the upper bound. 
\end{proof}

\section{Proof of Proposition \ref{Hallthm}}
\label{seclastpf}
\allco

Let $\cM, A$ be as in Section \ref{s:MathFrame}.
Define the class of sets
	\bea
	\label{cAdef}
	\cA := \{D: D \subset A, D ~{\rm closed}, v (\partial D) = 0 \}.
	\eea
Assume from now on that $f(x) = f_0 = 1/v(A)$ for all $x \in A$. Fix
$k \in \N$.  Let $\beta \in \R$ and let $(r_t)_{t >0}$ be strictly positive with
\bea
\lim_{t \to \infty} (  t \vvol_d f_0 r_t^d   - \log (t f_0)  -
(d+k -2) \lglg t ) = \beta.
\label{rt2}
\eea
Since $k$ is now fixed, from now on we write just $F_r(\X)$
instead of $F_{k,r}(\X)$, which was defined at
\eqref{F3def}.
Let $\eps' >0$
be a constant satisfying
\begin{align}
	1 - 2 \eps' >  \max \Big(
	(1- (8d)^{-1})^{1/(d+1)}, 9/10 \Big) .
	\label{e:defeps'}
\end{align}

\begin{prop}
	\label{propcompare2}
Let $x_0 \in A$. Then there exists an open set $V \subset \cM$ with $x_0 \in V$,
	and an injective open map $\phi: V \to \R^d$ having
	the properties described in Lemma \ref{lemcompare} with $\eps = \eps'$,
	such that moreover 
%
%
%
	for all $B \in \cA$ with $B  \subset V$ we have
	\bea
\lim_{t \to \infty}
	\Pr [ B \cap   A^{(r_t)}  \subset F_{r_t}(\Po_t) ]
	= \exp( -(c_{d}/(k-1)!) v( B  ) e^{-\beta}).
	\label{eqHausd2}
	\eea
%
%
%
%
\end{prop}

	Working towards a proof of Proposition \ref{propcompare2},
	fix $x_0 \in  A$.
Let
$V$ and $ \phi_\tx,$ $\tx \in V$, and $K_1$,
be as given by Lemma \ref{lemcompare3} (taking $\eps = \eps'$ there),
so that
$V \subset \cM$ is open with $x_0 \in V$, and
for each $\tx \in V $, $\phi_\tx$ is an
injection from $V$ to $\R^d$, with $\phi_{\tx}(x_0) = o $.
If $x \in A^o$, assume also that $V \subset A$.
Set $\phi = \phi_{x_0}$
With this choice of  $\phi$, taking a possibly smaller
set $V$ if needed, we  have all the properties
described in Lemma \ref{lemcompare}.


To complete the proof of  Proposition \ref{propcompare2}, it remains 
check that if $B \in \cA$ with $B \subset V$ then
\eqref{eqHausd2} holds.
This will take some effort. 
Until the end of the proof of Proposition \ref{propcompare2},
let $B \in \cA$ be fixed with  $B \subset V$.


Fix $\eta \in (0,1/(2d))$.
Given $t \geq 1$,  partition $\R^d$ into $d$-dimensional hypercubes
of side $t^{-\eta}$, and let those hypercubes in
the partition which are contained in
$\phi( B^o \cap A^{(r_t)})$ be
denoted $H_{t,1}\ldots,H_{t,\kappa_t}$.
 Also let those hypercubes in the partition which intersect
with $\phi(B )$ but are not contained in
$\phi(B^o \cap A^{(r_t)})$, be denoted 
 $H_{t,\kappa_t+1},\ldots,H_{t,\kappa^+_t}$.

Given $t \geq 1$,
	for each $j \in [\kappa_t^+]$,   
	let $y_{t,j}$ be the first point of  $ H_{t,j} \cap \phi(B)
	$ in the lexicographic ordering,
	 and set $x_{t,j} := \phi^{-1}(y_{t,j})$,
	so that $x_{t,j} \in B$.
	Then define $\psi_{t,j}: V \to \R^d$ by
	 $\psi_{t,j} := \phi_{x_{t,j}}$,
	as defined in Lemma \ref{lemcompare3}.

\begin{lemm}
	\label{lemcompare22}
	There exists $t_1 \in [1,\infty)$ such that
	for all $t \geq t_1$
	and each $j \in [\kappa_t]$:
	
	(a) for all measurable $ F \subset \phi^{-1}(H_{t,j})$ we have 
	$|(\lambda_d(\psi_{t,j}(F)) / v(F))-1| \leq  2dK_1 t^{-\eta}$;

	(b)
	$|(\|\psi_{t,j}(y)-\psi_{t,j}(z)\|/ \dist(y,z)) - 1| \leq  3dK_1
	t^{-\eta}$
	for all distinct $y,z \in \phi^{-1} ( H_{t,j})$;

	(c) the function  
	$\psi_{t,j} \circ \phi^{-1}: \phi(V) \to \R^d$
	extends to a linear map 
	(also denoted $\psi_{t,j} \circ \phi^{-1}$)
	from $\R^d $ to itself. This linear map satisfies
	$\psi_{t,j} \circ \phi^{-1}(\bH)= \bH$, and
	$\|\psi_{t,j} \circ 
	\phi^{-1} - I_d\|_{d \times d} < \eps'$.
\end{lemm}


\begin{proof}
	For measurable 
	$F \subset \phi^{-1}(H_{t,j})$, by
	Lemma \ref{lemcompare3}(c)
	\bean
	\left| \frac{\lambda_d(\psi_{t,j}(F))}{v(F)} -1 \right| =
	\left| \frac{\lambda_d(\phi_{x_{t,j}}(F))}{v(F)} -1 \right| 
	\leq   K_1 \diam( F \cup \{x_{t,j} \}), 
	\eean
	and since $\phi(F \cup \{x_{t,j}\} ) \subset H_{t,j}$,
	by Lemma \ref{lemcompare3}(f)
	the above is bounded by $ 2  K_1
	\diam (H_{t,j})$ and hence by $2d K_1 t^{-\eta}$.
	This yields part (a).

	Next let $y,z \in \phi^{-1}(H_{t,j})$. Then
	by Lemma \ref{lemcompare3}(e),
	\bean
	\left| \frac{\|\psi_{t,j}(y)- \psi_{t,j}(z)\|}{\dist(y,z)} -1 \right|
	= \left| 
	\frac{\|\phi_{x_{t,j}}(y)- \phi_{x_{t,j}}(z)\|}{\dist(y,z)}
	-1 \right|
	\leq K_1 ( \dist( x_{t,j},y) + \dist(y,z) ),
	\eean
	and since $\phi(x_{t,j}), \phi(y),\phi(z)$ all lie in $H_{t,j}$
	using also Lemma \ref{lemcompare3}(f) we see that
	the last line is at most
	$2 \sqrt{d} (1+ 2 \eps') K_1  t^{-\eta}$, which yields part (b).

	Part (c) follows from Lemma \ref{lemcompare3}(b).
\end{proof}

We shall refer to the sets $\psi_{t,j} \circ \phi^{-1}(H_{t,j}), 1 \leq j \leq 
\kappa_t^+$,
as {\em cubettes}. 
We would like to reassemble these cubettes 
into a block of macroscopic size. However,
the cubettes are not quite rectilinear in general.
Therefore to make them
 fit together requires further carpentry.
 For $D \subset \R^d, r>0$, 
 we shall use notation $D^{(r)}: = \{x \in \R^d: B_{\R^d}(x,r) \subset D\}$
 (so if $D$ is closed then so is $D^{(r)}$ whereas we took $A^{(r)}$
 to be open in $\cM$).

Let $\ggamma = 3/(4d)$, so that
$\eta < (2d)^{-1} < \ggamma < d^{-1}$.
We shall divide the 
set $\psi_{t,j} \circ \phi^{-1}(H_{t,j}) $
into smaller rectilinear hypercubes  of side 
$
t^{-\ggamma} 
$,
called {\em mini-cubes}.

Partition $\R^{d} $ into half-open $d$-dimensional hypercubes 
of side $t^{-\ggamma}$. Given $j \in [\kappa_t^+]$, 
denote the hypercubes in the partition which are contained
in the reduced cubette 
$(\psi_{t,j}\circ \phi^{-1}(H_{t,j}))^{(50r_t)}$
by $Q_{t,j,1},\ldots, Q_{t,j,\nu_{t,j}}$.
Also let those hypercubes in the partition
which intersect with 
$\psi_{t,j}\circ \phi^{-1}(H_{t,j})$
but are not contained in
$(\psi_{t,j}\circ \phi^{-1}(H_{t,j}))^{(50r_t)}$
be denoted $Q_{t,j,\nu_{t,j}+1}, \ldots, Q_{t,j,\nu^+_{t,j}}.$
Then $Q_{t,j,1}, \ldots, Q_{t,j,\nu^+_{t,j}}$
are our mini-cubes associated with the cubette 
$\psi_{t,j}\circ \phi^{-1}(H_{t,j})$, as illustrated in Figure
\ref{f:cubette}.

 \begin{figure}[h]
                \centering
           \includegraphics[width=1.0\linewidth, 
	   ]{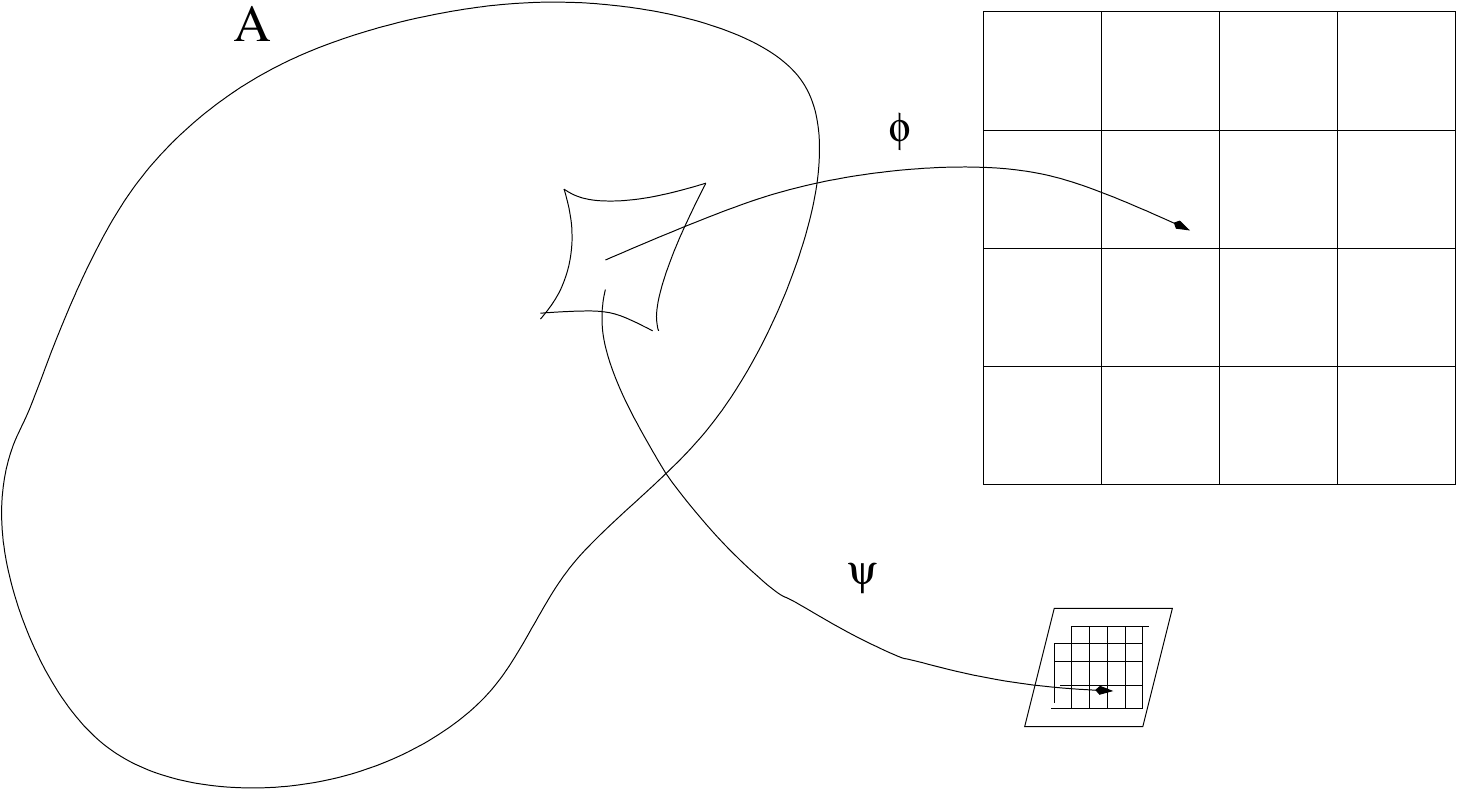}
                \caption{\label{f:cubette} The coarser grid
		shows some of the cubes $H_{t,j}$, and the
		parallelogram shows one of the
		cubettes $\psi_{t,j} \circ \phi^{-1}(H_{t,j})$.
		The finer grid shows the mini-cubes $Q_{t,j,\ell}$
	within that cubette.
	 }
        \end{figure}

We shall reassemble the mini-cubes to make a set  of macroscopic size.  
 Define  translations  $\sigma_{t,j,\ell},
 1 \leq j \leq \kappa_t^+, 1 \leq \ell \leq \nu_{t,j}^+$ of $\R^d$
such that
the translated  cubes $\sigma_{t,j,\ell}(Q_{t,j,\ell}),
  j \in [\kappa_t^+],   \ell \in [ \nu_{t,j}^+]$
are disjoint and the sets
   $ \Gamma_t := \cup_{j=1}^{\kappa_t} \cup_{\ell = 1}^{\nu_{t,j}} 
   \sigma_{t,j,\ell}(Q_{t,j,\ell} ) $
   and
   $ \Gamma^*_t := \cup_{j=1}^{\kappa_t^+} \cup_{\ell = 1}^{\nu_{t,j}^+} 
   \sigma_{t,j,\ell}(Q_{t,j,\ell} ) $,
   are both approximate  $d$-dimensional cubes, 
  where  an {\em approximate $d$-dimensional cube}
  is a union of packed rectilinear cubes of side
  $t^{-\alpha}$ that is contained in the cube   $(0,kt^{-\alpha}]^d$
  but contains  $(0,(k-1)t^{-\alpha}]^d$,
  for some $k \in \N$.
   
Let $\Gamma_t^- := \Gamma_t^{(2r_t)}$,
 the set of points in  $\Gamma_t$
 distant  at least  $2 r_t$ from $\partial \Gamma_t$.
\begin{lemm}[Volume or reassembled approximate cubes]
	\label{lemsurf2}
	It is the case that
\bea
	\lim_{t \to \infty} \lambda_d((\Gamma_t^*)^{(2r_t)}) =
	\lim_{t \to \infty} \lambda_d(\Gamma_t^*) =
	\lim_{t \to \infty} \lambda_{d}(\Gamma_t^-) 
	= 
	\lim_{t \to \infty} \lambda_{d}(\Gamma_t) 
	=
	v( B  ).
\label{0502a2}
\eea
\end{lemm}
\begin{proof}
	Let $t \geq t_1$, where $t_1$ is as in Lemma \ref{lemcompare2},
	and let  $j \in [\kappa_t^+]$.  Define  the cubette
	\bea
	G_{t,j} := \psi_{t,j} \circ \phi^{-1}(H_{t,j}) , 
	\label{eqGdef2}
	\eea
	and note that $\cup_{\ell=1}^{\nu_{t,j}} Q_{t,j,\ell} 
	\subset G_{t,j} \subset \cup_{\ell=1}^{\nu_{t,j}^+}
	Q_{t,j,\ell} $.
	Suppose $x \in
	\cup_{\ell =1}^{\nu_{t,j}^+} Q_{t,j,\ell}
	\setminus \cup_{\ell =1}^{\nu_{t,j}} Q_{t,j,\ell}.
	$
	Then  provided $t$ is large enough,
	$$
	\dist (x, \partial G_{t,j}) \leq 50 r_t + \sqrt{d} t^{-\ggamma}
	\leq  d t^{-\ggamma}.
	$$
	Then 
	by Lemmas \ref{lemcompare22} and \ref{lemcompare3}
	we have
	\bean
	\dist ( \phi \circ \psi_{t,j}^{-1} (x), \partial H_{t,j})
	\leq 2d t^{-\ggamma}.
	\eean
	That is, $ \phi \circ \psi_{t,j}^{-1} (
	\cup_{\ell =1}^{\nu_{t,j}^+} Q_{t,j,\ell} 
	\setminus \cup_{\ell =1}^{\nu_{t,j}} Q_{t,j,\ell} )$
	is contained in a $d$-dimensional hypercubic 
	annulus of thickness $4d t^{-\ggamma}$.
	Therefore
	\begin{align*}
	\lambda_{d}( \phi \circ \psi_{t,j}^{-1} ( 
	\cup_{\ell =1}^{\nu_{t,j}^+} Q_{t,j,\ell}
	\setminus 
	\cup_{\ell =1}^{\nu_{t,j}} Q_{t,j,\ell}))
	& \leq ((t^{-\eta} + 2d t^{-\ggamma})^{d} - 
	(t^{-\eta} - 2d t^{-\ggamma} )^{d}) 
	\\
	& = O(t^{\eta(1-d) - \ggamma}).
		\end{align*}
	Then using  Lemma \ref{lemcompare3} we obtain that
	\bea
	\lambda_{d} \big( (   
	 \cup_{\ell =1}^{\nu_{t,j}^+} Q_{t,j,\ell})
	\setminus \cup_{\ell =1}^{\nu_{t,j}} Q_{t,j,\ell} \big)
	\leq 
	c t^{\eta(1- d) - \ggamma} ,
	\label{0630a2}
	\eea
	where the constant $c$ does not depend on $t$, $j$ or $\ell$.
	Also
	$
	\lambda_{d}(\phi \circ \psi_{t,j}^{-1}(G_{t,j}) )
	= \lambda_{d} ( H_{t,j}  )   
	= t^{-\eta d},
	$
	so using Lemma \ref{lemcompare3} we obtain that
	$
	\lambda_{d}(G_{t,j} ) \geq (1/2) t^{-\eta d}$,
	and combining this with (\ref{0630a2}) we obtain the
	second inequality of the following (the first inequality is from
	set inclusion):
	\bea
	\lambda_{d}(G_{t,j}) \geq \lambda_{d} (\cup_{\ell=1}^{\nu_{t,j}}
	Q_{t,j,\ell}  ) \geq \lambda_{d}(G_{t,j})
	\left( 1- \frac{2c  t^{\eta  (1-d)
	- \ggamma}}{t^{- \eta d} } \right).
	\label{0913a}
	\eea
	Similarly
	\bea
	\lambda_{d}(G_{t,j}) \leq \lambda_{d} (\cup_{\ell=1}^{\nu_{t,j}^+}
	Q_{t,j,\ell}  ) \leq \lambda_{d}(G_{t,j})
	\left( 1 + \frac{2c  t^{\eta  (1-d)
	- \ggamma}}{t^{- \eta d} } \right).
	\label{0913b}
	\eea

	Since  $\eta < 1/(2d) < \ggamma$,
	we thus obtain from \eqref{0913a} that
	\bea
	\lambda_{d}(\Gamma_t) 
	= \sum_{j=1}^{\kappa_t}
	\lambda_{d}(\cup_{\ell=1}^{\nu_{t,j}} Q_{t,j,\ell} )
	\geq (1+o(1)) \sum_{j=1}^{\kappa_t} \lambda_{d}(G_{t,j})
	~~~ {\rm as}~ t \to \infty,
	\label{0630b2}
	\eea
	and by a similar argument  using (\ref{0913b}),
	\bea
	\lambda_d(\Gamma_t^*) \leq
	(1+ o(1)) \sum_{j=1}^{\kappa_t^+} \lambda_d(G_{t,j}).
	\label{0806b}
	\eea

	Set $U_t := \cup_{j=1}^{\kappa_t}\phi^{-1} (H_{t,j} )$
	and $U_t^+ := \cup_{j=1}^{\kappa_t^+}\phi^{-1} (H_{t,j} )$.
	 By  the definition, $U_t \subset  B^o $ and
	$\liminf_{t \to \infty} (U_t) =  B^o $,
	since  $ \phi( B^o ) \setminus \phi( U_t)$  is contained 
	in a region within distance $O( t^{-\eta})$
	of the boundary of $\phi( B^o \cap A^{(r_t)} )$.  
	Therefore $v(U_t) \to v( B^o )
	= v( B ) $ as $t \to \infty$ (the equality is because
	$B \in \cA$).

	 Also $ B \subset U_t^+ $ and
	$\limsup_{t \to \infty} (U_t^+) =  B $,
	since  $ \phi( U_t^+ \setminus   B ) $  is contained 
	in a region within 
	distance $O( t^{-\ggamma})$
	of the boundary of $\phi( B )$.
	Therefore $v(U_t^+) \to v( B )$ as $t \to \infty$.
	Then by \eqref{eqGdef2}
	and property (a) in Lemma \ref{lemcompare22},
	\begin{align*}
	\sum_{j=1}^{\kappa_t} \lambda_{d}
		(G_{t,j} ) & 
		\geq (1 - 2dK_1 t^{-\eta} )
	\sum_{j=1}^{\kappa_t} v(\phi^{-1}(H_{t,j}) )
		=
		(1 - 2dK_1 t^{-\eta})  v(U_t);  \\
	\sum_{j=1}^{\kappa_t^+} \lambda_{d}
		(G_{t,j} ) & 
		\leq (1 + 2dK_1 t^{-\eta} )
	\sum_{j=1}^{\kappa_t^+} v(\phi^{-1}(H_{t,j}) )
		= 
		(1 + 2dK_1 t^{-\eta})  
		v(U_t^+), 
		\end{align*}
	so that
	$ \lim_{t \to \infty} \sum_{j=1}^{\kappa_t} \lambda_{d}(
	G_{t,j}) = \lim_{t \to \infty} \sum_{j=1}^{\kappa_t^+} \lambda_{d}(
	G_{t,j}) = v ( B  ).
	$
	Combined with \eqref{0630b2} and \eqref{0806b},
	this yields the limit for $\lambda_d(\Gamma_t)$
	and for $\lambda_d(\Gamma_t^*)$
	in \eqref{0502a2},
	and then the limit for $\lambda_d(\Gamma_t^-)$ 
	and for $\lambda_d((\Gamma_t^*)^{(2r_t)})$
	is clear.
%
%
%
\end{proof}

Given $a >0$, 
Let $\cH_a$ denote a homogeneous Poisson point process of intensity
$a$ in $\R^d$.
%
For simple point processes $\X,\Y$ in $\R^d$,
we write $\X \eqd \Y$ if they have the
same distribution (see e.g. \cite{LP}). We
 write $\X \ll \Y$, and say $\Y$ stochastically dominates $\X$,
if
there exist coupled point processes $\X', \Y'$
such that $\X' \eqd \X$, and $\Y' \eqd \Y'$, and
$\X' \subset \Y'$ almost surely.

\begin{lemm}[Stochastic Domination lemma]
	\label{lemPPdom2}
There exists $t_2 \geq t_1$, where $t_1$ is as in Lemma 
	\ref{lemcompare22}, such that for all $t \geq t_2$,
	$H_{t,j} \subset \phi(V)$ for all $j \in [\kappa_t^+]$ and
\bea
\cup_{j=1}^{\kappa_t} \cup_{\ell =1}^{\nu_{t,j}}
	\sigma_{t,j,\ell} ( \psi_{t,j}(\Po_t \cap  \phi^{-1}(H_{t,j})
	) \cap Q_{t,j,\ell} )
\ll
\cH_{tf_0(1+ 4d K_1 t^{-\eta})} \cap \Gamma_t ,
\label{eqPPdom3}
\eea
and
\bea
	\cH_{tf_0(1- 4d K_1 t^{-\eta})} \cap \cup_{j=1}^{\kappa_t^+}
	\cup_{\ell = 1}^{\nu_t^+}
	\sigma_{t,j,\ell} (
	Q_{t,j,\ell}
	\cap 
	\psi_{t,j}(A \cap \phi^{-1}(H_{t,j}))
	)
	~~~~~~~~~~~~
	\nonumber \\
\ll
\cup_{j=1}^{\kappa_t^+} \cup_{\ell =1}^{\nu_{t,j}^+}
	\sigma_{t,j,\ell} ( \psi_{t,j}(\Po_t \cap \phi^{-1}(H_{t,j}))
	\cap Q_{t,j,\ell} ).
\label{eqPPdom32}
\eea
\end{lemm}
\begin{proof}
	Let $t \geq t_1$ and $j \in [\kappa_t^+]$, $\ell \in [\nu_{t,j}^+]$.
	Since $B \in \cA$, $B \subset V$
	and $\phi$ is an open map, we can and do
	assume $t$ is large enough that $H_{t,j} \in 
	\phi(V) $ for all $j \in [\kappa_t^+]$.
The point process $\Po_t \cap V$ is Poisson on $V \cap A$
with intensity measure $t f_0 v(dx)$,
where $v$ is the Riemannian volume measure on $\M$.
Hence by the Mapping theorem (see e.g. \cite{LP}),
	$\psi_{t,j}(\Po_t \cap \phi^{-1}(H_{t,j}))$ is
a Poisson point process on $\psi_{t,j} \circ \phi^{-1} (H_{t,j})$
with intensity measure $t f_0 \mu_{t,j}(\cdot)$, where
	the measure $\mu_{t,j}$ is given
	by $\mu_{t,j}(G)= v(\psi_{t,j}^{-1}(G) \cap A)$ for
	measurable $G \subset
	\psi_{t,j} \circ \phi^{-1}(H_{t,j})$.
	For such $G$, by Lemma \ref{lemcompare22}(a) 
$$
	\left| \frac{\lambda_d(G \cap \psi_{t,j}(A))}{\mu_{t,j}(G)} - 1\right| =
	\left| \frac{\lambda_d(\psi_{t,j}(\psi_{t,j}^{-1}(G) \cap A)
	)}{v(\psi_{t,j}^{-1}(G) \cap A)} - 1\right|  \leq 
	 2dK_1 t^{-\eta},
$$
	so that for large $t$, we
	have $\mu_{t,j}(G) 
	\in (1 \pm 4d K_1 t^{-\eta})\lambda_d(G \cap \psi_{t,j}(A))$.

	By the Radon-Nikodym theorem $\mu_{t,j}$ has
	a density with respect to $\lambda_d$, and this
	density lies in the range
	$(1\pm 4 d K_t t^{-\eta}){\bf 1}_{\psi_{t,j} (A)}$,
	$\lambda_d$-almost everywhere in $\psi_{t,j} \circ
	\phi^{-1}(H_{t,j})$. 
Therefore by the Superposition theorem (see e.g. \cite{LP}),
	one can obtain a Poisson point process with the same
	distribution as $\psi_{t,j}(\Po_t \cap \phi^{-1} (H_{t,j}))$
	by the superposition of $\cH_{tf_0(1 - 4 d K_1 t^{-\eta})}
	\cap \psi_{t,j}(A  \cap \phi^{-1}(H_{t,j}))$
	and a further independent 
	Poisson process. Hence
	\begin{align}
	\cH_{tf_0(1-4d K_1t^{-\eta})} 
	\cap \psi_{t,j} (A \cap \phi^{-1}(H_{t,j})) 
	\ll
\psi_{t,j}(\Po_t \cap 
	\phi^{-1}(H_{t,j})), 
		\label{ppdom1}
	\end{align}
	and by a similar argument
	\begin{align}
\psi_{t,j}(\Po_t \cap 
	\phi^{-1}(H_{t,j})) 
\ll  
	\cH_{tf_0(1+4d K_1t^{-\eta})} \cap \psi_{t,j} (A
	\cap \phi^{-1}(H_{t,j})).
\label{ppdom2}
	\end{align}

By the Mapping theorem,
for each $t,j,\ell$ and each $s >0$ we have
$\sigma_{t,j,\ell} (\cH_s) \eqd \cH_s$.
 Therefore using \eqref{ppdom1} and \eqref{ppdom2}
 we have that 
 \bea
\cH_{tf_0(1-4d K_1t^{-\eta})} \cap
\sigma_{t,j,\ell} (Q_{t,j,\ell} \cap
\psi_{t,j} (A \cap \phi^{-1}(H_{t,j}))) 
~~~~~~~~~~~~~~~~~~~~~~~~~
\nonumber \\ 
\ll
	\sigma_{t,j,\ell}( \psi_{t,j}(\Po_t \cap \phi^{-1}(H_{t,j}))
	\cap  Q_{t,j,\ell})
~~~~~~~~~~~~~~~~~~~~~~~~~
\nonumber	\\
~~~~~~~~~~~~~~~ \ll
\cH_{tf_0(1+4d K_1t^{-\eta})} \cap
\sigma_{t,j,\ell} (Q_{t,j,\ell} \cap
\psi_{t,j} (A \cap \phi^{-1}(H_{t,j}))). 
\label{0817a}
\eea
The sets $\phi^{-1}(H_{t,j}) \cap \psi_{t,j}^{-1}(Q_{t,j,\ell}),$
$j \in [\kappa_t^+], \ell \in [\nu_{t,j}^+]$ are disjoint,
since the sets $H_{t,1}, \ldots, H_{t,\kappa_t^+}$ are 
disjoint and for each $j$ the sets $Q_{t,j,1}\ldots,Q_{t,j,\nu_{t,j}^+}$
are  disjoint. Therefore the Poisson processes 
$\Po_t \cap \phi^{-1}(H_{t,j}) \cap \psi_{t,j}^{-1}(Q_{t,j,\ell})$,
$j \in [\kappa_t^+], \ell \in [\nu_{t,j}^+]$ are mutually independent.
Since the sets $\sigma_{t,j,\ell}(Q_{t,j,\ell}),  j \in
	[\kappa_t^+],  \ell \in [\nu_{t,j}^+]$, are  disjoint, 
by using the first relation in (\ref{0817a}),
taking unions over such $(j,\ell)$ and
using the Superposition theorem,
we obtain (\ref{eqPPdom32}).  Similarly using the 
second relation in (\ref{0817a}) and taking the union now
only over $j \in [\kappa_t], \ell \in [\nu_{t,j}]$ 
yields
	(\ref{eqPPdom3}). 
\end{proof}

Given $(r_t)_{t >0}$ satisfying \eqref{rt2},
we now define
\begin{align}
r_t^+:= r_t(1+ 8 d K_1 t^{-\eta});~~~~~~
r_t^-:= r_t(1- 8 d K_1 t^{-\eta}).
	\label{rtpmdef}
\end{align}


\begin{lemm}[Covering an approximate Euclidean cube]
	\label{lemfromCov2}
	It is the case that
\bea
\lim_{t \to \infty} 
\Pr[ 
	(\Gamma_t^*)^{(2r_t)}  \subset
	F_{r_t^-}( \cH_{tf_0(1-6d K_1t^{-\eta})}
	\cap  \Gamma_t^* ) ] 
=
\lim_{t \to \infty} 
	 \Pr[ 
	\Gamma_t^-  \subset
	F_{r_t^+}( \cH_{tf_0(1+6d K_1t^{-\eta})}
	\cap  \Gamma_t )] 
\nonumber \\
	 = \exp( -(c_{d}/(k-1)!) v( B ) e^{-\beta}).
	~~~~~~~~~~~
\label{0616c2}
\eea
\end{lemm}
\begin{proof}
The limiting statement \eqref{rt2}
still holds if we replace $r_t$ with $r_t(1\pm 8 d K_1t^{-\eta})$
and $t$ with $t(1 \pm 6d K_1t^{-\eta})$. Therefore 
	(\ref{0616c2})
follows by
	Lemma \ref{lemsurf2}
	and 
	\cite[Lemma 7.2]{ECover}.
\end{proof}

For $j \in [\kappa_t^+]$ and $\ell \in [\nu_{t,j}^+]$, set 
$Q_{t,j,\ell}^- := Q_{t,j,\ell}^{(2r_t)}$.
Given $t$, we shall sometimes write just $\cup_{j,\ell}$
for $\cup_{j=1}^{\kappa_t} \cup_{\ell=1}^{\nu_{t,j}}$ and
$\cup_{j,\ell}^*$ for $\cup_{j=1}^{\kappa_t^+} \cup_{\ell=1}^{\nu_{t,j}^+}$.
As in \cite{ECover}, we shall refer to the union in $\R^d$ of
sets $\sigma_{t,j,\ell}(Q_{t,j,\ell}
\setminus Q_{t,j,\ell}^-)$, and also to the union
in $A$ of their pre-images, as {\em tartan} (plaid).

\begin{lemm}[Covering a tartan set in $\R^d$]
	\label{lemplaid2}
	It is the case that
	\bea
	\lim_{t \to \infty}
	\Pr[ 
	\cup_{j,\ell} \sigma_{t,j,\ell} (
	Q_{t,j,\ell} \setminus
	Q_{t,j,\ell}^- ) \subset 
	F_{r_t^+} ( \cH_{tf_0(1+ 6d K_1 t^{-\eta})} ) ]
	=1.
	\label{0619c2}
	\eea
\end{lemm}
	\begin{proof}
		Let $\alpha' \in (\alpha, 1/d)$ and
		let $f_0^- \in (f_0(1 + \alpha'-1/d) , f_0)$.
		Let $\tmu$ be the uniform probability distribution on a fixed
		open cube of volume $1/f_0^-$ containing
		the cube $[0, f_0^{-1/d}]^d$. Given $t >1$,
		let $\tPo_t$ be a Poisson point process
		in $\R^d$ 
		with intensity $t \tmu$.
		Then $\tPo_t \ll \cH_{tf_0(1+ 6 d K_1 t^{-\eta})}$,
		and therefore it suffices to prove that
		\begin{align}
			\lim_{t \to \infty}
			\Pr[ 
			\cup_{j,\ell} \sigma_{t,j,\ell} (Q_{t,j,\ell}
			\setminus Q_{t,j,\ell}^- ) \subset 
		F_{r_t} ( \tPo_{t} ) ] =1.
			\label{e:0904a}
			\end{align}

		We shall apply Lemma  \ref{lemmeta}(iii)
		in the space $\R^d$, taking
		$\mu= \tmu$.  We claim that
 \bea
 		\kappa(
		\cup_{j,\ell} \sigma_{t,j,\ell}
		(Q_{t,j,\ell} \setminus Q_{t,j,\ell}^-) , r_t) 
		= O( t^{\ggamma d} \times (t^{-\ggamma}/r_t)^{d-1} )
 = O( r_t^{- d(\ggamma' +1 - 1/d)}).
		\label{0619a3}
 \eea
 Indeed, the number of sets in the union is
is $O(t^{\ggamma d})$, while for each $(j,\ell)$ the set 
 $\sigma_{j,\ell} ( Q_{t,j,\ell} \setminus Q_{t,j,\ell}^-) $
 is a $d$-dimensional hypercubic annular region of
 thickness $O(r_t)$ and diameter $O(t^{-\ggamma})$, and
 therefore can be covered by $O((t^{-\ggamma}/r_t)^{d-1})$ balls
		of radius $r_t$.
		This gives the first part of (\ref{0619a3}),
		and the second part comes from (\ref{rt2}).
		Also for $t$ sufficiently large $\tmu(B(x,s))
		\geq f_0^- \vvol_d s^d$ for all $x \in \Gamma_t^-$
		and $s \in (0,r_t]$.

		Thus the condition  $u>b/(ad)$ of
		Lemma \ref{lemmeta}(iii)  holds with $u= 1/(f_0\vvol_d)$, $b/d = \alpha'+1-1/d$ and $a= f_0^- \vvol_d$. Then \eqref{e:0904a}
		follows by Lemma \ref{lemmeta}(iii).
	\end{proof}

\begin{lemm}[Upper bound for covering interior region]
	\label{lem06192}
	It is the case that
	\bea
	\limsup_{t \to \infty}
	\Pr[   B \cap A^{(r_t)}  \subset F_{r_t}(\Po_t) ]
	\leq \exp(-(c_{d}/(k-1)!) v( B) e^{-\beta} ).
	\label{0619d2}
	\eea
\end{lemm}
	\begin{proof}
Let $t >0$.  Suppose that event
		$\{  B \cap A^{(r_t)}   \subset F_{r_t}(\Po_t) \}$ occurs.  Let
$
		y \in  \cup_{j,\ell}
		\sigma_{t,j,\ell}(Q_{t,j,\ell}^-).
$ 
Take $j_0 \in [\kappa_t]$, $\ell_0 \in [\nu_{t,j_0}]$
		and
		$z \in Q_{t,j_0,\ell_0}^-$
such that $y = \sigma_{t,j_0,\ell_0}(z)$. 
Since $Q_{t,j_0,\ell_0}
\subset \psi_{t,j_0} \circ \phi^{-1} (H_{t,j_0})$, we can and do take
 $u \in \phi^{-1}(H_{t,j})$ such that 
$z = \psi_{t,j_0}(u)$, and hence 
		$y = \sigma_{t,j_0,\ell_0} \circ \psi_{t,j_0}(u)$. 
		Then by the definitions just before Lemma \ref{lemcompare22},
		$H_{t,j_0} \subset \phi(B \cap A^{(r_t)})$,
		so that
\bean
		z \in Q_{t,j_0,\ell_0} 
		\subset \psi_{t,j_0}\circ \phi^{-1}(H_{t,j_0})
		\subset \psi_{t,j_0}(B \cap A^{(r_t)} ).
\eean
		Hence $ u  \in B \cap A^{(r_t)} $, so
		by our assumption
		$u \in F_{r_t}(\Po_t)$, and there are at least $k$ points
$w$ of $\Po_t \cap B(u,r_t)$. 
For each such point $w$, since
		$u \in B \cap  A^{(r_t)}$
		and $\dist(u,w) \leq r_t$, and $B \in \cA, B \subset V$,
we have		$w \in V \cap A$.
		Using Lemma \ref{lemcompare3}(e) and recalling the definition of $x_{t,j}$
		and $\psi_{t,j}$ from just before Lemma \ref{lemcompare22},
we have 
		\begin{align}
			\|\psi_{t,j_0}(w) - \psi_{t,j_0}(u)\| 
			& \leq r_t (1+  K_1
		(\dist(w,u) + \dist(u,x_{t,j_0})) )
\nonumber \\
			& \leq r_t ( 1+ K_1 ( r_t +  d t^{-\eta})),
\label{0616a2}
		\end{align}
where the second inequality uses the fact that both $u$ and $x_{t,j_0}$
lie in $\phi^{-1}(H_{t,j_0})$, so $\|\phi(u) - \phi(x_{t,j_0}) \|
		\leq \sqrt{d}
		t^{-\eta}$ and so by Lemma  \ref{lemcompare3} again,
$\dist(u,x_{t,j_0}) \leq d t^{-\eta}$.

In particular, $\|\psi_{t,j_0}(w) - \psi_{t,j_0}(u)\| \leq 2 r_t$, and
therefore since $z= \psi_{t,j_0}(u) \in Q_{t,j_0,\ell_0}^-$ 
		we have 
$\psi_{t,j_0}(w) \in Q_{t,j_0,\ell_0} \subset 
		\psi_{t,j_0} \circ \phi^{-1}(H_{t,j_0})$.
		Since $\sigma_{t,j_0,\ell_0}$ is
an isometry and $r_t \leq t^{-\eta}$ we also have from \eqref{0616a2}
		and \eqref{rtpmdef}
		for $t$ large that
$$
\|
\sigma_{t,j_0,\ell_0} \circ \psi_{t,j_0}(w) - y
\| \leq r_t(1+ 2dK_1 t^{-\eta}) \leq r_t^+,
$$
and hence 
$$
y \in
F_{r_t^+} ( \sigma_{t,j_0,\ell_0} 
( \psi_{t,j_0}(\Po_t) \cap Q_{t,j_0,\ell_0}))
\subset
F_{r_t^+} ( \cup_{j=1}^{\kappa_t} \cup_{\ell=1}^{\nu_{t,j}} \sigma_{t,j,\ell} 
( \psi_{t,j}(\Po_t \cap \phi^{-1}(H_{t,j})) \cap Q_{t,j,\ell})).
$$
Therefore we have the event inclusion 
\begin{align*}
	\{  B \cap A^{(r_t)} \subset F_{r_t}(\Po_t) \}
\subset \{
	&
	\cup_{j,\ell}
	\sigma_{t,j,\ell}(Q_{t,j,\ell}^-) 
	\nonumber \\ &
	\subset 
	F_{r_t^+} (
	\cup_{j,\ell}
	\sigma_{t,j,\ell} ( \psi_{t,j}(\Po_t \cap \phi^{-1}(H_{t,j})) \cap Q_{t,j,\ell} ) )
\}
.
\end{align*}
Using (\ref{eqPPdom3}) and the fact that $Q_{t,j,\ell} \subset
\psi_{t,j} \circ \phi^{-1}(H_{t,j}) $ for $\ell \in [\nu_{t,j}]$,
we obtain that
\begin{align}
\limsup_{t \to \infty} \Pr [  B \cap A^{(r_t)}  \subset
	F_{r_t}(\Po_t) ]  
\leq 
\limsup_{t \to \infty} \Pr[ 
\cup_{j,\ell} \sigma_{t,j,\ell}(Q^-_{t,j,\ell}) 
	\subset F_{r_t^+} (\cH_{tf_0(1+6d K_1 t^{-\eta})} 
) ].
\label{0616d2}
\end{align}

By the union bound
\begin{align}
\Pr[ 
	\cup_{j,\ell} & \sigma_{t,j,\ell} (Q_{t,j,\ell}^-)  
\subset F_{r_t^+}(\cH_{tf_0(1+6d K_1t^{-\eta}) }  ) ]
\leq \Pr[ \Gamma_{t}^-  
\subset F_{r_t^+} (\cH_{tf_0(1+6d K_1t^{-\eta}) })   ]
\nonumber \\
	& + \Pr [ \{ 
	 \Gamma_t^-  \setminus
	\cup_{j,\ell} \sigma_{t,j,\ell} (Q_{t,j,\ell}^- ) \subset 
	F_{r_t^+} ( \cH_{tf_0(1+ 6d K_1 t^{-\eta})} ) \} ^c ],
	\label{e:0917a}
\end{align}
and since $\Gamma_t \setminus \Gamma_t^- \subset
\cup_{j,\ell} \sigma_{t,j,\ell} (Q_{t,j,\ell} \setminus
Q_{t,j,\ell}^-)$,
it follows from Lemma \ref{lemplaid2} that the second term in
the right hand side of \eqref{e:0917a} tends to zero.
By Lemma \ref{lemfromCov2}, the first term tends to $\exp(- ( c_{d}/(k-1)!)
v(B) e^{-\beta})$, so by \eqref{0616d2} we obtain
%
\eqref{0619d2}.
\end{proof}

	\begin{lemm}[Covering a region in $A^{(r_t)}$ with tartan set 
		removed]
		\label{lemcovA2}
		It is the case that
		\bea
		\liminf_{t \to \infty}
		\Pr[ \cup_{j,\ell}^* (\psi_{t,j}^{-1 }(Q_{t,j,\ell}^-)
		\cap \phi^{-1}(H_{t,j}^{(2r_t)})) \cap A^{(r_t)}
		\subset F_{r_t}(\Po_t) ] 
		\nonumber \\
		\geq \exp(-(c_{d}/(k-1)!) 
		v(  B  ) e^{-\beta}).
		\label{0706a2}
		\eea
	\end{lemm}
	\begin{proof}
		Let $t >0$.  For this proof, we define the event $E_t$ by
\bean
		E_t : = \{ (\Gamma_t^*)^{(2r_t)}  \subset
		F_{r_t^-} (\cup^*_{j,\ell} ( \sigma_{t,j,\ell}  
		[
		\psi_{t,j}(\Po_t \cap 
		\phi^{-1}(H_{t,j}))
		\cap
		Q_{t,j,\ell} 
		]
~~~~~~~~~~~~~~
		\\
~~~~~~~~~~~~~~
		\cup 
		[ \cH_{tf_0(1-6 d K_1 t^{-\eta})} 
		\cap
		\sigma_{t,j,\ell}(
		Q_{t,j,\ell}
		\setminus \psi_{t,j} (A \cap \phi^{-1}(H_{t,j})))]) )
\},
\eean
		where we here assume the Poisson  processes
	 $ \cH_{tf_0(1-6dK_1 t^{-\eta})} $ and $\Po_t$ are independent. 
		By (\ref{eqPPdom32}), and the Superposition theorem 
		\cite{LP},
	the point process
\bean
		\cup^*_{j,\ell} (  \sigma_{t,j,\ell}  
		[\psi_{t,j}(\Po_t \cap 
		\phi^{-1}(H_{t,j}))
		\cap Q_{t,j,\ell}
		]
		~~~~~~~~
		~~~~~~~~
		~~~~~~~~
		~~~~~~~~
		\\
		~~~~~~~~
		~~~~~~~~
		~~~~~~~~
		\cup [ 
		\cH_{tf_0(1-6dK_1 t^{-\eta})} 
		\cap \sigma_{t,j,\ell}(
		 Q_{t,j,\ell}
		\setminus \psi_{t,j} (A \cap \phi^{-1}(H_{t,j})))] 
		)
\eean
		stochastically dominates the point process
$		\cH_{tf_0(1-6dK_1 t^{-\eta})} \cap \Gamma_t^*$. 
		Hence by Lemma \ref{lemfromCov2},
\bea
\liminf_{t \to \infty} 
\Pr[ E_t ] 
		\geq \exp( - (c_{d}/(k-1)!) 
		v(  B  )  e^{-\beta}).
\label{0505a2}
\eea

Suppose $E_t$ occurs.
		Let $x \in \cup_{j,\ell}^* ( 
		\psi_{t,j}^{-1 }(Q_{t,j,\ell}^-) \cap \phi^{-1}(
		H_{t,j}^{(2r_t)}))
		\cap A^{(r_t)}$.
		Take $j_0 \in [\kappa_t^+]$ and
		$\ell_0 \in [\nu_{t,j_0}^+]$
		such that $x \in \psi_{t,j_0}^{-1}( Q_{t,j_0,\ell_0}^-) 
		\cap \phi^{-1}(H_{t,j_0}^{(2r_t)})$.
Then
$$
y:= \sigma_{t,j_0,\ell_0} \circ \psi_{t,j_0}(x)
\in \sigma_{t,j_0,\ell_0}(Q_{t,j_0,\ell_0}^-)
		\subset (\Gamma_t^*)^{(2r_t)} ,
		$$
		and since we assume $E_t$ occurs,
\bean
		\card ( B_{\R^d}(y, (1-3dK_1 t^{-\eta})r_t) 
\cap
		\cup^*_{j,\ell}
	(	\sigma_{t,j,\ell}
[\psi_{t,j}(\Po_t \cap \phi^{-1}(H_{t,j}))  \cap Q_{t,j,\ell}] 
~~~~~~
		\\
		~~~~~~
		\cup [ \cH_{tf_0(1-6dK_1 t^{-\eta})} 
		\cap
		\sigma_{t,j,\ell}(
		Q_{t,j,\ell}
		\setminus \psi_{t,j} (A \cap \phi^{-1}(H_{t,j})))] ))
		\geq k.
\eean
Since $\psi_{t,j_0}(x) \in Q_{t,j_0,\ell_0}^-$, 
		we have
$
B_{\R^d}(y,(1-3dK_1 t^{-\eta})r_t) 
\subset 
\sigma_{t,j_0,\ell_0}( Q_{t,j_0,\ell_0} ). 
$
	Since the sets $\sigma_{t,j,\ell}(Q_{t,j,\ell})$,
	$j \in [\kappa_t^+], \ell \in [\nu_{t,j}^+]$,
	are disjoint,  there are at least $k$ points $\tiy$ satisfying
\bea
\tiy \in B_{\R^d}(y,(1-3dK_1 t^{-\eta})r_t) \cap 
(
\sigma_{t,j_0,\ell_0}[
	\psi_{t,j_0}(\Po_t \cap 
\phi^{-1}(H_{t,j_0}))
	\cap
	Q_{t,j_0,\ell_0}
]
\nonumber \\
\cup 
		[\cH_{tf_0(1-6dK_1  t^{-\eta}) } \cap
		\sigma_{t,j,\ell_0}( Q_{t,j_0,\ell_0}
		\setminus \psi_{t,j_0}(A \cap \phi^{-1}(H_{t,j_0})) ) ]
		)
.
\label{0817c}
\eea
Given such a $\tiy$, set
$\tx := \psi_{t,j_0}^{-1} (\sigma_{t,j_0,\ell_0}^{-1}(\tiy)) $.
Then $\|\psi_{t,j_0}(x) -\psi_{t,j_0}(\tx) \|
= \| \sigma_{t,j_0,\ell_0}^{-1} (y) - \sigma_{t,j_0,\ell_0}^{-1}(\tiy) \| \leq
r_t$,
and by Lemma \ref{lemcompare3} (f), 
$\dist(x,\tx) \leq 2 r_t$. 
By Lemma \ref{lemcompare3}(e) and the definition of
$\psi_{t,j}$ just before Lemma \ref{lemcompare22},
\begin{align}
\left| \frac{\| \psi_{t,j_0}(x)- \psi_{t,j_0}(\tx)\|}{\dist(x,\tx)}
-1 \right|
	& \leq  K_1 (\dist(x,\tx) + \dist(x, x_{t,j_0}) ),  
\label{0806c}
\end{align}
where $x_{t,j_0}$ is an element of $\phi^{-1}(H_{t,j_0})$.

Since 
$x, x_{t,j_0} \in \phi^{-1}(H_{t,j_0})$,
we have
$
\dist(x, x_{t,j_0}) \leq \sqrt{d} \dist(\phi(x),
\phi(x_{t,j_0}) ) \leq d t^{-\eta}.
$
Also $\dist(x,\tx) \leq 2r_t = o(t^{-\eta})$.
Thus by (\ref{0806c}) we obtain that for $t$ large,
$$
\| \psi_{t,j_0}(x) - \psi_{t,j_0}(\tx) \| 
\geq  (1 -  2d K_1 t^{-\eta})
 \dist (x,\tx),
$$
and hence
\bean
\dist(\tx ,x ) \leq
(1+ 3d K_1 t^{-\eta})
\| \psi_{t,j_0}(\tx) - \psi_{t,j_0} (x)\|
= (1+ 3d K_1 t^{-\eta})
\|\tiy - y\|
\leq r_t,
\eean
and therefore since we assume $x \in \phi^{-1}(H_{t,j_0}^{(2r_t)}) 
\cap A^{(r_t)}$
we have $\tx \in \phi^{-1} (H_{t,j_0}) \cap A$,
so that $\tiy$ is not in the second set on the right hand side of 
 (\ref{0817c}),
and hence
$\tx \in \Po_t $. Therefore
$
x \in 
F_{r_t} (\Po_t ),
$
so the event on the left side of (\ref{0706a2}) occurs.
Then the result follows from (\ref{0505a2}).
\end{proof}

	\begin{lemm}[Covering tartan sets in $A^{(r_t)}$]
		\label{lemJuly2}
		As $t \to \infty$, we have:
		\begin{align}
			\Pr[
	 \cup_{j,\ell}^* \psi_{t,j}^{-1}(Q_{t,j,\ell} \setminus 
		Q_{t,j,\ell}^-) \cap A^{(r_t)} \subset 
		F_{r_t} ( \Po_{t} ) ]
		 \to 1;
			\label{e:0904b}
			\\
		\Pr[ \cup_{j=1}^{\kappa_t^+} \phi^{-1}(H_{t,j} \setminus 
		H_{t,j}^{(2r_t)}) \cap A^{(r_t)} \subset 
			F_{r_t} ( \Po_{t} ) ] \to 1. 
			\label{e:0904c}
		\end{align}
	\end{lemm}
	\begin{proof}
		Let $\alpha' = \frac{7}{8d}$, so that
		$\alpha<\alpha'<1/d$ and $(1-\eps')^{d+1}>1+\alpha'-1/d$
		by \eqref{e:defeps'}. We claim that
 \bea
		\kappa(\cup^*_{j,\ell} \psi_{t,j}^{-1} (Q_{t,j,\ell}  
		\setminus Q_{t,j,\ell}^-) \cap A^{(r_t)} ,  r_t)
		= O( t^{\ggamma d} \times (t^{-\ggamma}/r_t)^{d-1} )
		= O(r_t^{-d(\alpha'+1-1/d)}).
		~
		\label{0619a4}
 \eea
	Indeed,  for each  $(j,\ell)$, as discussed in the proof 
	of Lemma \ref{lemplaid2}
	the set
	 $ (Q_{t,j,\ell} \setminus Q_{t,j,\ell}^-)$ 
	can be covered by $O((t^{-\ggamma}/r_t)^{d-1})$ Euclidean balls
	of radius $r_t/2$.
	The mapping $\psi_{t,j}$ changes  all distances by a factor of
	at most 2, and therefore  also
	$\psi_{t,j}^{-1}(Q_{t,j,\ell} \setminus Q_{t,j,\ell}^-)$
	can be covered by $O((t^{-\ggamma}/r_t)^{d-1})$ geodetic
	balls of radius $ r_t$.
	The number of sets 
	$\psi_{t,j}^{-1}(Q_{t,j,\ell} \setminus Q_{t,j,\ell}^-)$
	in the union is $O(t^{\ggamma d})$,  
	and the first part of the assertion (\ref{0619a4}) follows.
		Then the second part comes from (\ref{rt2}).

	Let $y\in\cup^*_{j,\ell} \psi_{t,j}^{-1} (Q_{t,j,\ell}  
	\setminus Q_{t,j,\ell}^-) \cap A^{(r_t)}$.
	Since $B$ is compact with $B \subset V$, and
		$\psi_{t,j}^{-1}(Q_{t,j,\ell}) \cap \phi^{-1}(H_{t,j}) \neq
		\emptyset$
		and  $H_{t,j} \cap \phi(B) \neq \emptyset$ for all $j,\ell$,
		for all large enough $t$ we have $B(y,r_t)\subset V$.
	Using  Lemma
		\ref{lemcompare},
 we have for $0 < s \leq r_t$ that
 \begin{align}
	 v(B(y, s)) & \geq (1-  \eps') \lambda_d(\phi(B(y, s)))
	 \nonumber
 \\
		& \geq (1- \eps') \lambda_d( B_{\R^d}(\phi(y),(1-\eps')s ))= (1-\eps')^{d+1} \vvol_d s^d.
		\label{e:vBLB}
 \end{align}
	Therefore, the condition $u >b/(ad)$ of Lemma \ref{lemmeta}(iii)
	holds with $u=1/(\vvol_d f_0),$ $ a=(1- \eps')^{d+1}f_0\vvol_d$ and 
		$b/d= 1+\alpha'-1/d$. Then  \eqref{e:0904b}  follows
		by Lemma \ref{lemmeta}(iii).
		
	The proof of \eqref{e:0904c} is similar, and we omit the details.
	\end{proof}


\begin{proof}[Proof of Proposition \ref{propcompare2}]
	Let $x \in   B \cap A^{(r_t)}  $.
	Then $\phi(x) \in \phi( B) 
	 \subset \cup_{j=1}^{\kappa_t^+}
	H_{t,j}$. Take $j_0 \in [\kappa_t^+]$ such that 
	$\phi(x) \in H_{t,j_0}$. Then $\psi_{t,j_0}(x)  \in
	\psi_{t,j_0} \circ \phi^{-1}(H_{t,j_0}) \subset
	\cup_{\ell=1}^{\nu_{t,j_0}^+}  (Q_{t,j_0,\ell})$, so
	$x \in \psi_{t,j_0}^{-1} (Q_{t,j_0,\ell_0})$ for some
	$\ell_0 \in [\nu_{t,j}^+] $. Thus
 $$
	 B \cap A^{(r_t)}    \subset
	\big( \cup_{j,\ell}^* \psi_{t,j}^{-1}(Q_{t,j,\ell}) \big)
	\cap \cup_{j=1}^{\kappa_t^+} \phi^{-1}(H_{t,j}). 
$$
Therefore by Lemmas
	 \ref{lemcovA2} 
	and \ref{lemJuly2},
	we obtain that
\bean
\liminf_{t \to \infty}
	\Pr [  B \cap A^{(r_t)}  \subset F_{r_t}(\Po_t) ]
	\geq \exp( -(c_{d}/(k-1)!) v(  B  ) e^{-\beta} ).
\eean
Combined with Lemma \ref{lem06192}, this yields 
	\eqref{eqHausd2}.
	In view of the discussion just
	after the statement  of
	Proposition \ref{propcompare2}, this
	completes  the proof
	of that result.
\end{proof}

\begin{lemm}[Neighbourhood with nice boundary]
	\label{lemWexist2}
	Let $x_0 \in V \subset \cM$, with $V$ open.
	There exists a compact subset 
	$W$ of $\cM$ with $v(\partial W ) =0$, $x_0 \in W^o$,
$W \subset V$, and
$\limsup_{r \downarrow 0} r^{d-1} \kappa(
	\partial W,r  ) < \infty$. 
\end{lemm}
\begin{proof}
	Assume without loss of generality that there exists
	$\phi:V \to \R^d$ with the properties described
	in Lemma \ref{lemcompare}, with $\eps = 1/99$.
	(if not, one can take a smaller
	set $V$.)
	Then we can and do choose $r_0>0$ such that  $B_{\R^d}(\phi(x_0),2r_0)
	\subset
	\phi(V)$. Take $W = \phi^{-1}(B_{\R^d}(\phi(x_0),r_0))$.
	
	Then we claim $W $  has the required properties. Indeed,
	$\phi(\partial W  ) = \partial B(\phi(x_0),r_0) $; hence
	$\lambda_d(\phi(\partial W)) =0$
	so $v(\partial W) =0$ 
	by Lemma \ref{lemcompare},
	and $W $ is compact. 
 Moreover	the $(d-1)$-dimensional sphere
	$\phi(\partial W  ) $
	can be covered by $O(r^{1-d})$
balls of radius $r/2$, and the pre-images of these balls 
under $\phi$ are contained in geodetic balls of radius $r$ in $\cM$,
	by Lemma \ref{lemcompare},
	and cover $\partial W$. 
\end{proof}

By Proposition \ref{propcompare2}, Lemma \ref{lemWexist2},
and
a compactness argument we can and do take integers $m' \geq m \geq 1$,
and a finite collection of
quadruples  $(x_i, V_i,  W_i, \phi_i)$, $i \in [m'] $,
with $ \partial A \subset \cup_{i=1}^m W_i $ and
$A \subset \cup_{i=1}^{m'} W_i$, such that
for each $i \in [m']$ we have:

\begin{itemize}
	\item
$V_i$ is  open  in $\M$ and 
$W_i$ is  compact  in $\M$ 
with $v(\partial W_i ) =0$,  $x_i \in W_i^o$, $W_i \subset V_i$ and
\bea
		\limsup_{r \downarrow 0} r^{d-1} \kappa ( \partial W_i ,r) 
< \infty;
		\label{0913c}
\eea

\item
if  $i \leq m $ then $x_i \in  \partial A$, while
if $i >m $ then $V_i \subset A^o$;

\item
$\phi_i: V_i \to \R^d$ is an injective open map, with
	\bea
	|(\|\phi_i(y)-\phi_i(z)\|/ \dist(y,z)) - 1| \leq \eps',
	~~ \forall ~ y,z \in V_i, y \neq z;
	\label{0817d}
	\eea 
\item
		if $i \leq m$ then $\phi_i(V_i \cap A) = \phi_i(V_i) \cap \bH$; 
	\item
	 for all measurable $ F \subset V_i$ we have 
	$|(\lambda_d(\phi_{i}(F)) / v(F))-1| \leq \eps'$;
\item
for all   $B \in \cA$ with $ B \subset V_i$, we have (\ref{eqHausd2}).
%
\end{itemize}


We refer to $W_1,\ldots,W_{m'}$ as {\em patches} and 
	define the {\em interior patch boundary region} $\Delta_t^{\rm int}$
	(in the interior of $A$ but near the boundaries of the patches) by
	$$
	\Delta_t^{\rm int} := \{ x \in 	A^{(r_t)}:
		B(x,r_t) \cap (\cup_{i=1}^{m'} \partial W_i) \neq \emptyset
		 \}.
	$$

	\begin{lemm}[Covering the interior patch boundary region]
		\label{lemDelta2}
		It is the case that
		$\lim_{t \to \infty}
		\Pr[ \Delta_t^{\rm int} \subset F_{r_t}(\Po_t) ] =1.
		$
	\end{lemm}
	\begin{proof}
		For all $t$ sufficiently large
		and  any $y\in \Delta_t^{\rm int}$, $s \in (0,r_t]$ we have
\bean
		v(B(y, s) ) \geq (1-  \eps') \lambda_d (\phi(B(y, s) ))
		\geq (1- \eps') \lambda_d( B_{\R^d}(\phi(y),(1-\eps') s) ).
 \eean
Also, using \eqref{0913c} we have $\kappa(\Delta^{\rm int}_t,r_t) = 
		O(r_t^{1-d})$.  Therefore by \eqref{e:defeps'}
		the condition $u>b/(ad)$ of Lemma \ref{lemmeta}(iii) holds 
		with $u=1/(f_0\vvol_d)$, 
		$a=(1- \eps')^{d+1} f_0 \vvol_d$, and $b/d=1-1/d$. The result follows.
	\end{proof}
\begin{proof}[Proof of Proposition \ref{Hallthm}]
		Define the sets
	$W_1^* := W_1$,
	and for
		$i = 2,3,\ldots,m'$ set $W_i^* :=
	W_i \setminus \cup_{j=1}^{i-1} W_j$.
	Then  $ A \subset \cup_{i=1}^{m'} W_i^*$,
	and $W_1^*,\ldots,W_{m'}^*$ are disjoint.
		For $i \in [m']$ let $W_{i,t} := \overline{W_i^*}
		\setminus \Delta_t^{\rm int}$.
	 Then for large enough $t$ we have 
	 \bea
		A^{(r_t)} \subset \cup_{i=1}^{m'} ( W_i \cap A^{(r_t)} )
		= \cup_{i=1}^{m'} ( W_i^* \cap A^{(r_t)} )
		\subset 
		\cup_{i=1}^{m'} (W_{i,t} \cap A^{(r_t)}) \cup 
		\Delta_t^{\rm int}.
	 \label{0709b2}
	 \eea
	 Let $B \in \cA$.
		We claim that for all $i \in [m']$ we have
	 \bea
	 \lim_{t \to \infty}
		\Pr[W_{i,t} \cap B \cap   A^{(r_t)}
		\subset F_{r_t}(\Po_t)]
		= \exp(- (c_{d}/(k-1)!) v( W_i^* \cap  B)e^{-\beta}).
	 \label{0709a2}
	 \eea
		Indeed, since $\partial(B \cap \overline{W_i^*})
		\subset \partial B \cup \cup_{\ell=1}^{m'} \partial W_\ell$, 
		we have 
		$v(\partial(B \cap \overline{W_i^*})) =0$ and
		$B \cap \overline{W_i^*} \in \cA$.
		Also $B \cap \overline{W_i^*} \subset V_i$.
		Hence  \eqref{eqHausd2} 
	from Proposition \ref{propcompare2}
		applies to $B \cap 
		\overline{W_i^*}$
		and yields
	 \begin{align*}
	 \lim_{t \to \infty}
		\Pr[\overline{W_{i}^*}  \cap B \cap   A^{(r_t)}
		\subset F_{r_t}(\Po_t)]
		 = \exp(- (c_{d}/(k-1)!) v( \overline{W_i^*}
		 \cap  B)e^{-\beta}),
	 \end{align*}
	 and combined with Lemma \ref{lemDelta2} and the fact
	 that $v(\partial W_i^*) =0$, this yields the
	 claim (\ref{0709a2}).

	 Next, suppose $i \in [m']$
	 and $t >0$.
	 Let $x \in W_{i,t} \cap   A^{(r_t)}$.
	 Since $\partial ( W_i^* \cap A)
	 \subset \cup_{\ell =1}^{m'} \partial W_\ell
	 \cup \partial A$,
	  and $x \notin \Delta_t^{\rm int} $, we have
	  $\dist(x, \partial (W_i^* \cap A))
	  \geq  r_t$. Therefore $B(x, r_t) \subset 
	 \overline{W_i^*} \cap  A$.
	   Hence the event $\{W_{i,t}\cap B \cap A^{(r_t)} 
	   \subset F_{r_t}(\Po_t)\}$ is in the completion of the $\sigma$-algebra
	   generated by the restriction of $\Po_t$ to 
	   $W_i^*$, using the fact that $v(\partial W_i^*) =0$.

	   This shows that the events $\{(W_{i,t} \cap B \cap A^{(r_t)})
	   \subset F_{r_t}(\Po_t)\},
	   1 \leq i \leq m'$ are mutually independent. Combined with
	   (\ref{0709b2}) and 
	   Lemma \ref{lemDelta2}, this shows that
	   \begin{align}
	   \lim_{t \to \infty} \Pr[ B \cap  A^{(r_t)} 
		   \subset F_{r_t}(\Po_t) ] & = \lim_{t \to \infty}
		   \Pr[ \cap_{i=1}^{m'} \{ W_{i,t} \cap B \cap  A^{(r_t)} 
		   \subset F_{r_t}(\Po_t) \} ]
	   \nonumber \\
		   & = \lim_{t \to \infty} \prod_{i=1}^{m'} \Pr [
		   W_{i,t} \cap B \cap  A^{(r_t)} 
		   \subset F_{r_t}(\Po_t)  ]
    \nonumber \\
		   & = \exp ( -( c_{d}/(k-1)!) v(B)e^{-\beta}),
		   \label{e:0829}
	   \end{align}
    where the last line comes from (\ref{0709a2}).
	By \eqref{Fnequiv} we have that
	\bea
	\tR_{Z_t,k} \leq r_t \Longleftrightarrow 
	B \cap A^{(r_t)} \subset F_{r_t}(\Po_t). 
	\label{0913d}
	\eea
	Choosing  $r_t$ so the pre-limit in \eqref{rt2},
	is actually equal to $\beta $ for $t$ large, 
	using \eqref{e:0829} and \eqref{0913d}  
	we obtain the last line of \eqref{1228a}.
		We then obtain 
		 the first line of \eqref{1228a} by 
		 \cite[Lemma 7.1]{ECover}, the proof of which
		 carries over to the present setting.

		 If $B$ is compact and $B \subset A^o$, then
		 there exists $r>0$ such that $B \subset A^{(r)}$;
		 hence
		 by \eqref{rt2} and \eqref{eqmaxspacPo},
		 for large enough $t$ we have $R_{Z_t} \leq r_t
		 \Longleftrightarrow \tR_{Z_t} \leq r_t$, so
		 (\ref{0913d}) holds for $R_{Z_t,k}$ as well, so
		 we can derive  (\ref{0114a}) in the same manner
		 as (\ref{1228a}).
\end{proof}


\section{Proof of Theorem \ref{th:weak}}
\label{s:ProofWeak}
\allco
Let $k \in \N$, $\zeta \in \R $, and
assume throughout this section that $(r_t)_{t >0}$ satisfy $r_t >0$ and
\bea
\frac{f_0 t \vvol_d r_t^d}{2} - \left(\frac{d - 1}{d}\right) \log (t f_0) -
\left( d+k-3+1/d \right)
\lglg t \to \zeta
~~~ {\rm as} ~ t \to \infty.
\label{rt1}
\eea
	As in Section \ref{seclastpf}, let $\eps' > 0$
	 satisfying \eqref{e:defeps'}, and write $F_{r}(\X)$
	 for $F_{k,r}(\X)$.
Define 
	\bea
	\label{cBdef}
	\cB := \{ D \subset A: D ~{\rm closed},
	D \cap \partial A \neq \emptyset,
	v (\partial D) = 0, 
	~
	\tv( \partial_{\partial A} (D \cap \partial A))=0
	\}.
	\eea



\begin{prop}
	\label{propcompare}
	Let $x_0 \in \partial A$. Then there exists
	an open set $V \subset \cM$ with $x_0 \in V$,
	and an injective
	function $\phi: V \to \R^d$ having the properties described
	in Lemma \ref{lemcompare}, taking $x=x_0$ and $\eps = \eps'$,
	such that moreover
%
%
%
%
	for all $B \in \cB$ with $B \subset V$,
	\begin{align}
	\lim_{t \to \infty} \Pr[  B \setminus A^{(r_t)} \subset
		F_{r_t}(\Po_t)  ] 
		& =
\lim_{t \to \infty}
\Pr [ 
	B \cap \partial A \subset F_{r_t}(\Po_t) ]
	\nonumber \\
		& = \exp( -c_{d,k} \tv(  B \cap \partial A ) e^{-\zeta}).
	\label{eqHausd}
		\end{align}
\end{prop}

	Working towards a proof of Proposition \ref{propcompare},
	fix $x_0 \in \partial A$.
Let
$V$,  and $ \phi_\tx, $ $\tx \in V$,
be as given by Lemma \ref{lemcompare3}, taking $\eps = \eps'$
so that
$V \subset \cM$ is open with $x_0 \in V$, and
for each $\tx \in V $, 
$\phi_\tx$ is an
injection from $V$ to $\R^d$, with
$\phi_{\tilde x}(x_0) =o$. Set $\phi := \phi_{x_0}$.
With this choice of $V$ and $\phi$ (possibly after taking a smaller $V$)
we have the properties described in Lemma \ref{lemcompare}.
In particular $\phi(V \cap A) = \phi(V) \cap \bH$.

To complete the proof of  Proposition \ref{propcompare}, it remains 
to show that \eqref{eqHausd} holds for all
 $B \in \cB$ with $B \subset V$.
This will take some effort. 
Fix such a set $B$.
%
 For $D \subset \partial A$, let $D^o_{\partial A}:=
 D \setminus (\overline{\partial A \setminus D})$
 the interior of $D$ relative to $\partial A$.
Define $\gamma_t $ for $ t > 0$  by
 \begin{align}
	 \gamma_t := \sup
	 \{
	 \dist( 
	 x , B \cap \partial A):
	 x \in B \setminus A^{(r_t)} \}
	 ,
	 \label{e:gammadef}
 \end{align}
 and observe that $\gamma_t \to 0$ as $t \to \infty$ by the compactness of $B$.

Fix $\eta \in (0,1/(2d))$.
Given $t \geq 1$,  partition $\partial \bH$ into half-open
$(d-1)$-dimensional hypercubes
of side $t^{-\eta}$, and let those hypercubes in
the partition which are contained in
the set $\phi( ( B \cap \partial A)^o_{\partial A} )$ 
 be denoted $H_{t,1},\ldots,H_{t,\kappa_t}$.
Let those hypercubes in
the partition which are not contained in
$\phi(( B \cap \partial A)_{\partial A}^o )$ 
but do lie within distance $2r_t + 2 \gamma_t$ of
$\phi( B \cap \partial A )$ 
 be denoted $H_{t,\kappa_t+1},\ldots,H_{t,\kappa^+_t}$.

For $1 \leq j \leq \kappa^+_t$,
set $H_{t,j}^{(2r_t)} :=
\{x \in H_{t,j}: B_{\R^d}(x,2r_t) \cap \partial \bH \subset H_{t,j}\}$,
and let 
\begin{align}
	S_{t,j} & :=
H_{t,j} \oplus [o,5r_t e_d] =
\{x + re_d: x \in H_{t,j}, r \in [0,5r_t] \};
\nonumber
\\
	S_{t,j}^* & :=  H_{t,j} \oplus  [o,2r_te_d] , 
~~~~~
	S_{t,j}^-  :=  H_{t,j}^{(2r_t)} \oplus  [o,2r_te_d] .
\label{eqStjdef}
\end{align}
Since $B \cap  \partial A$ is compact and contained in the open set $V$,
we can and do choose $t_0 \in [1,\infty)$
such that for all $t \geq t_0$
we have that 
\bea
\{x + u + h e_d:
x \in \phi( B  \cap \partial A ), u \in \bH,
\|u \| \leq \sqrt{d}t^{-\eta}, h \in [0,5r_t] \} \subset \phi(V), 
\label{0913f}
\eea
and hence for each $j \in [\kappa_t^+]$ we have
$S_{t,j} \subset  \phi(V)$;
since also $S_{t,j} \subset \bH$, in fact
$S_{t,j}^* \subset S_{t,j} \subset  
\phi(V \cap A)$.
We refer to the  sets $S_{t,j} $ as {\em slabs}.

Given $t \geq t_0$,
	for each $j \in [\kappa_t^+]$,   
	let $y_{t,j}$ be the first point of
	$H_{t,j}$
	in the lexicographic ordering,
	 and set $x_{t,j} := \phi^{-1}(y_{t,j})$,
	so that $x_{t,j} \in V$.
	Then define $\psi_{t,j}: V \to \R^d$ by
	 $\psi_{t,j} := \phi_{x_{t,j}}$,
	as defined in Lemma \ref{lemcompare3}.
	Let $K_1$ be as given in Lemma \ref{lemcompare3}.
%
\begin{lemm}
	\label{lemcompare2}
	There exists $t_3 \in [t_2,\infty)$, where $t_2$ is as in
	Lemma
	\ref{lemPPdom2},
	such that
	for all $t \geq t_3$
	and each $j \in [\kappa_t^+]$
	and all distinct 
	$y,z \in \phi^{-1} ( S_{t,j})$ we
	have:
	
	(a) for all measurable $ F \subset \phi^{-1}(S_{t,j})$ we have 
	$|(\lambda_d(\psi_{t,j}(F)) / v(F))-1| \leq  d K_1 t^{-\eta}$;

	(b) for all measurable
	$ F \subset \phi^{-1}(S_{t,j}) \cap \partial A$ we have 
	$|(\lambda_{d-1}(\psi_{t,j}(F)) / \tv(F))-1| \leq  d K_1 t^{-\eta}$;

	(c)
	$|(\|\psi_{t,j}(y)-\psi_{t,j}(z)\|/ \dist(y,z)) - 1| \leq  
	2 d K_1 t^{-\eta}$;

	(d) the function  
	$\psi_{t,j} \circ \phi^{-1}: \phi(V) \to \R^d$
	extends to a linear map 
	(also denoted $\psi_{t,j} \circ \phi^{-1}$)
	from $\R^d $ to itself. This linear map satisfies
	$\psi_{t,j} \circ \phi^{-1}(\bH)= \bH$, and
	$\|\psi_{t,j} \circ 
	\phi^{-1} - I_d\|_{d \times d} < \eps'$;

	(e)
	It is the case that
	$\psi_{t,j}(V \cap A)= \psi_{t,j}(V) \cap \bH$.
\end{lemm}

\begin{proof}
	For measurable 
	$F \subset \phi^{-1}(S_{t,j})$, by
	Lemma \ref{lemcompare3}(c)
	\bean
	\left| \frac{\lambda_d(\psi_{t,j}(F))}{v(F)} -1 \right| =
	\left| \frac{\lambda_d(\phi_{x_{t,j}}(F))}{v(F)} -1 \right| 
	\leq  K_1 \diam( F \cup \{x_{t,j} \}), 
	\eean
	and since $\phi(F \cup \{x_{t,j}\} ) \subset S_{t,j}$,
	by Lemma \ref{lemcompare3}(f)
	the above is bounded by $K_1 \diam (S_{t,j})$ and hence
	by $d K_1 t^{-\eta}$.
	This yields part (a).

	The proof of (b) is almost identical, now using part (d)
	of Lemma \ref{lemcompare3}.

	Next let $y,z \in \phi^{-1}(S_{t,j})$. Then
	by Lemma \ref{lemcompare3}(e),
	\begin{align*}
	\left| \frac{\|\psi_{t,j}(y)- \psi_{t,j}(z)\|}{\dist(y,z)} -1 \right|
		& = \left| 
	\frac{\|\phi_{x_{t,j}}(y)- \phi_{x_{t,j}}(z)\|}{\dist(y,z)}
	-1 \right|
	\\
		& \leq K_1 ( \dist( x_{t,j},y) + \dist(y,z) ),
\end{align*}
	and since $\phi(x_{t,j}), \phi(y),\phi(z)$ all lie in $S_{t,j}$
	using also Lemma \ref{lemcompare3}(f) we see that
	the last line is at most
	$2 K_1 d t^{-\eta}$, which yields part (c).

	Part (d) follows from Lemma \ref{lemcompare3}(b),
	and part (e) comes from Lemma \ref{lemcompare3}(a).
\end{proof}

We shall refer to the sets $\psi_{t,j} \circ \phi^{-1}(S_{t,j}), 1 \leq j \leq 
\kappa_t$,
as {\em slabettes}. 
In general,
the slabettes are parallelepipeds rather than rectilinear slabs.
The thickness of these slabettes may be less than that of
the original slabs (which was $5r_t$) and be different for
different slabettes, but is at least $4r_t$, as we shall now prove.
We also show that if we reduce the base of the slabette slightly,
the product of the reduced base with the interval $[0,4r_t]$
(a `straightened out reduced slabette')
is contained in the original slabette. 

	Given $t \geq t_3$, $j \in [\kappa_t^+]$,
	define $L_{t,j}$  and $L_{t,j}^{-} $,
	 the `lower face' and `reduced lower face' (respectively)
	 of the slabette $\psi_{t,j} \circ \phi^{-1} (S_{t,j})$, by
	\bea
	L_{t,j}:= \psi_{t,j}\circ \phi^{-1}
	(H_{t,j});
	~~~~~
	L_{t,j}^{-} := \{x \in L_{t,j}: B_{\R^d}(x,50 r_t) \cap
	\partial \bH\subset L_{t,j} \}. 
	\label{eqDtjdef}
	\eea
\begin{lemm}[Straightened out reduced slabette lies inside the original slabette]
	\label{lemslabette}
	For all $t \geq t_3,$ $j \in [\kappa_t^+]$,
	we have $L^-_{t,j} \oplus [o,4r_t e_d] \subset 
	\psi_{t,j} \circ \phi^{-1}(S_{t,j})$.
\end{lemm}

\begin{proof}
	Note that that the linear map $\psi_{t,j} \circ \phi^{-1}$
	has full rank. We assert that 
	\bea
	\| \psi_{t,j}  \circ \phi^{-1} \|_{d \times d} \leq 1 +  \eps';
	~~~~~~~
	\| (\psi_{t,j}  \circ \phi^{-1})^{-1} \|_{d \times d} \leq 1 + 2 \eps'.
	\label{0623a}
	\eea
	Indeed, the first inequality of \eqref{0623a}
	comes from Lemma
	  \ref{lemcompare2}(d), and
	given $x \in \R^d$, setting
	$y = \psi_{t,j}  \circ \phi^{-1}(x)$,
	by Lemma  \ref{lemcompare2}(d)
	we have $\|y-x\| \leq \eps' \|x\|$ so that
	$\|y\| \geq (1- \eps')\|x\|$ and the second inequality of
	\eqref{0623a} follows
	(using that $\eps' < 1/2$).

	
	Let $x \in L_{t,j}^-$ and $r \in [0,4 r_t].$
	Let $w = (\psi_{t,j} \circ \phi^{-1})^{-1}(x + r e_d)$,
	and let $w'$ be the orthogonal  projection of $w $ onto
	$\partial \bH$. 
	Since the linear map $\psi_{t,j} \circ \phi^{-1}$
	maps $\partial \bH $ onto itself,
	using linearity 
	(and that $\eps' <1/8$)
	we have
	\begin{align*}
		\| w- w'\|  = \langle w, e_d \rangle & =
	\langle (\psi_{t,j} \circ \phi^{-1})^{-1}(re_d), e_d \rangle 
	\\
		& \leq (1+  \eps') r \leq 5 r_t.
	\end{align*}
	Therefore  since
	$\eps' < 1/10$,
	$
	\| \psi_{t,j}\circ \phi^{-1}(w'-w) \| \leq (1+  \eps') 5r_t \leq 
	(5.5)r_t,
	$
	so that
	\begin{align*}
	\|\psi_{t,j} \circ \phi^{-1}(w') - x\|
		& \leq \| \psi_{t,j} \circ \phi^{-1} (w'-w) \|
	+ \| \psi_{t,j} \circ \phi^{-1} (w) -x\|
	\\
		& \leq 5.5 r_t + \| \psi_{t,j} \circ \phi^{-1}(re_d)   \|
	\leq 10 r_t.
	\end{align*}
	Thus $\psi_{t,j} \circ \phi^{-1}(w') \in L_{t,j}$,
	so that $w' \in H_{t,j} $. Therefore
	since $w = w' + r' e_d$ for some $r' \in [0, 5r_t]$,
	also $w \in S_{t,j}$ and thus $x + re_d \in \psi_{t,j} \circ
	\phi^{-1}(S_{t,j})$.
\end{proof}

We would like to reassemble the slabettes 
$\psi_{t,j} \circ \phi^{-1}(S_{t,j})$ 
into a slab of macroscopic size. However, if $d \geq 3$ the faces
of the slabettes are not quite rectilinear in general;
for example, if $d=3$ then they are parallelograms.
Therefore to make them
 fit together requires further carpentry, as in Section \ref{seclastpf}.

To do this, let $\ggamma = 3/(4d)$, so
that $\eta < (2d)^{-1} < \alpha < d^{-1}$.
We shall divide each region 
$(\psi_{t,j} \circ \phi^{-1}(S_{t,j})) \cap (\R^{d-1} \times [0,4r_t])$
into smaller rectilinear blocks of dimensions
$
t^{-\ggamma} \times \cdots \times t^{-\ggamma} \times 4 r_t
$,
called {\em mini-slabs} and denoted $T_{t,j,\ell}, 1 \leq
\ell \leq \nu_{t,j}$.

We define the {\em side boundary} 
 of the slabette
$\psi_{t,j} \circ \phi^{-1}(S_{t,j}) $
to consist of all faces of the slabette that are adjacent to
the lower face $L_{t,j}$.
Within the slabette there is a region  near the side boundary, 
and also a region  a distance greater than $4r_t$ from the lower face of the
slabette,   
which are not included in any mini-slabs
$T_{t,j,\ell}, 1 \leq \ell \leq \nu_{t,j}$.
We shall add some further mini-slabs to completely
cover the lower face.

Here is the definition of mini-slabs in detail.
Partition $\partial \bH $ into 
$(d-1)$-dimensional
hypercubes (or if $d=2$, intervals)
of side $t^{-\ggamma}$ and
denote the hypercubes in the partition which lie within 
$L_{t,j}^-$
by $Q_{t,j,1},\ldots,$ $ Q_{t,j,\nu_{t,j}}$.
Denote the hypercubes in the partition which intersect
$L_{t,j}$
but are not contained in $L_{t,j}^-$
by $Q_{t,j,\nu_{t,j}+1},\ldots, Q_{t,j,\nu^+_{t,j}}$.

Our mini-slabs associated
with the slabette $\psi_{t,j}\circ \phi^{-1}(S_{t,j})$
are the hyperrectangles
 \bea
 T_{t,j,\ell} :=
 Q_{t,j,\ell} \oplus [o,4r_te_d], 
 ~~~ 1 \leq \ell \leq \nu_{t,j}^+.
 \label{eqTdef}
 \eea
Given $(t,j,\ell)$, with $t >0$,
$j \in [ \kappa_t^+]$ and $\ell \in [ \nu_{t,j}^+]$,
define also the reduced mini-slab 
\bea
T_{t,j,\ell}^- :=
 Q_{t,j,\ell}^- \oplus [o,2r_te_d], 
\label{eqTminus}
\eea
where $Q_{t,j,\ell}^-:= \{x \in Q_{t,j,\ell}: B_{\R^d}(x,2r_t)
\cap \partial \bH \subset Q_{t,j,\ell}\}$.
For each $t >0$, $j \in [\kappa_t],$ $\ell \in [\nu_{t,j}]$ we
have 
$ Q_{t,j,\ell}
\subset L_{t,j}^-$
so that
by Lemma \ref{lemslabette},
\begin{align}
	T_{t,j,\ell} \subset \psi_{t,j} \circ \phi^{-1}(S_{t,j}).
	\label{e:0813a}
	\end{align}

We shall reassemble the mini-slabs 
to make a slab with a  base of macroscopic size.  
To do this, define  translations  $\sigma_{t,j,\ell},
 1 \leq j \leq \kappa_t^+, 1 \leq \ell \leq \nu_{t,j}^+$ of $\R^d$
such that $\sigma_{t,j,\ell}(Q_{t,j,\ell}),
  j \in [\kappa_t^+],   \ell \in [ \nu_{t,j}^+]$
are disjoint,  and such that  the sets
$\Gamma_t $ and $\Gamma_t^+$, given by
\bean
\Gamma_t 
:=
\cup_{j=1}^{\kappa_t} \cup_{\ell =1}^{\nu_{t,j}}
\sigma_{t,j,\ell}(Q_{t,j,\ell}); ~~~~~~~~~~ 
\Gamma_t^* := 
\cup_{j=1}^{\kappa_t^+} \cup_{\ell =1}^{\nu_{t,j}^+}
\sigma_{t,j,\ell}(Q_{t,j,\ell}),
\eean
are approximate  $(d-1)$-dimensional cubes 
	as described in Section \ref{seclastpf}.

For $F \subset \partial \bH$, let $\partial_{\partial \bH} F$ denote
the boundary of $F$ relative to $\partial \bH$.
Let $\Gamma_t^-$ be the set of points in  $\Gamma_t$
 distant  more than $2 r_t$ from $\partial_{\partial \bH} \Gamma_t$, and let
 $(\Gamma_t^*)^-$ be the set of points in  $\Gamma_t^*$
 distant  more than $2 r_t$ from $\partial_{\partial \bH} \Gamma_t^*$.
\begin{lemm}[Total surface measure of bases of mini-slabs]
	\label{lemsurf}
	It is the case that
	\begin{align}
	\lim_{t \to \infty} \lambda_{d-1}(\Gamma_t^-) 
	= 
	\lim_{t \to \infty} \lambda_{d-1}(\Gamma_t) 
	=
	\tv(B \cap \partial A );
\label{0502a}
\\
	\lim_{t \to \infty} \lambda_{d-1}((\Gamma_t^*)^-) 
	= 
	\lim_{t \to \infty} \lambda_{d-1}(\Gamma_t^*) 
	=
	\tv(B \cap \partial A ).
	\label{0819a}
	\end{align}
\end{lemm}
\begin{proof}
	Let $t \geq t_3$ and $j \in [\kappa_t^+]$.
	Recalling the definition of $L_{t,j}$ at
	\eqref{eqDtjdef}, 
	note that 
	$\cup_{\ell=1}^{\nu_{t,j}} Q_{t,j,\ell}
	\subset L_{t,j} \subset
	\cup_{\ell=1}^{\nu_{t,j}^+} Q_{t,j,\ell}$.
	Suppose $x \in
	\cup_{\ell =1}^{\nu_{t,j}^+} Q_{t,j,\ell}
	\setminus 
	\cup_{\ell =1}^{\nu_{t,j}} Q_{t,j,\ell}$.
	Then  provided $t$ is large enough,
	$$
	\dist (x, \partial_{\partial  \bH}  
	L_{t,j}) \leq 50 r_t + (d-1) t^{-\ggamma}
	\leq  d t^{-\ggamma}.
	$$
	Then with $H_{t,j} $ defined at (\ref{eqStjdef}),
	by Lemmas \ref{lemcompare2} and \ref{lemcompare3},
	and \eqref{eqDtjdef},
	we have
	\begin{align*}
	\dist ( \phi \circ \psi_{t,j}^{-1} (x), 
	\partial_{\partial \bH} H_{t,j}
	)
	\leq 2d t^{-\ggamma}.
	\end{align*}
	That is, $ \phi \circ \psi_{t,j}^{-1} (
	\cup_{\ell=1}^{\nu_{t,j}^+} Q_{t,j,\ell}
	\setminus \cup_{\ell =1}^{\nu_{t,j}} Q_{t,j,\ell}
	)
	$
	is contained in a $(d-1)$-dimensional hypercubic 
	annulus of thickness $4d t^{-\ggamma}$.
	Therefore
	\begin{align*}
	\lambda_{d-1}( \phi \circ \psi_{t,j}^{-1} ( 
	\cup_{\ell=1}^{\nu_{t,j}^+} Q_{t,j,\ell}
	\setminus \cup_{\ell =1}^{\nu_{t,j}} Q_{t,j,\ell}))
		& \leq ((t^{-\eta} + 4d t^{-\ggamma} )^{d-1} - (t^{-\eta}
	- 2d t^{-\ggamma} )^{d-1}) 
	\\
		& = O ( t^{\eta (2-d) - \ggamma} ) .
	\end{align*}
	Then using  Lemma \ref{lemcompare3} we obtain that
	\bea
	\lambda_{d-1}(   
	\cup_{\ell =1}^{\nu_{t,j}^+} Q_{t,j,\ell}
	\setminus 
	\cup_{\ell =1}^{\nu_{t,j}} Q_{t,j,\ell})
	\leq c t^{\eta(2- d) - \ggamma} ,
	\label{0630a}
	\eea
	for some constant $c$ not depending on $t$ or $j$.  Also
	$
	= \lambda_{d-1} 
	(H_{t,j})
	= t^{\eta(1-d)},
	$
	so using \eqref{eqDtjdef} and
	Lemma \ref{lemcompare3},
	we obtain that
	$
	\lambda_{d-1}(L_{t,j} ) \geq (1/2) t^{\eta(1-d)}$,
	and combining this with (\ref{0630a}) we obtain the
	second inequality of the following (the first inequality is from
	set inclusion):
	\bea
	\lambda_{d-1}(L_{t,j}) \geq \lambda_{d-1} 
	(\cup_{\ell=1}^{\nu_{t,j}}
	Q_{t,j,\ell}  ) \geq \lambda_{d-1}(L_{t,j})
	\left( 1- \frac{2ct^{\eta  (2-d)
	- \ggamma}}{t^{\eta(1-d)} } \right).
	\label{0913g}
	\eea
	Similarly,
	\bea
	\lambda_{d-1}(L_{t,j}) \leq \lambda_{d-1} 
	(\cup_{\ell=1}^{\nu_{t,j}^+}
	Q_{t,j,\ell}  ) \leq \lambda_{d-1}(L_{t,j})
	\left( 1+ \frac{2ct^{\eta  (2-d)
	- \ggamma}}{t^{\eta(1-d)} } \right).
	\label{0913h}
	\eea
	Since $\eta < 1/(2d) < \ggamma$,
	we  obtain from (\ref{0913g}) that
	\bea
	\lambda_{d-1}(\Gamma_t) 
	= \sum_{j=1}^{\kappa_t}
	\lambda_{d-1}(\cup_{\ell=1}^{\nu_t}
	Q_{t,j,\ell})
	= (1+o(1)) \sum_{j=1}^{\kappa_t} \lambda_{d-1}(L_{t,j})
	~~~ {\rm as}~ t \to \infty,
	\label{0630b}
	\eea
	and similarly by \eqref{0913h},
	\bea
	\lambda_{d-1}(\Gamma_t^*) 
	= (1+o(1)) \sum_{j=1}^{\kappa_t^+} \lambda_{d-1}(L_{t,j})
	~~~ {\rm as}~ t \to \infty.
	\label{0819b}
	\eea

	Set $U_t := \cup_{j=1}^{\kappa_t}\phi^{-1} (H_{t,j })
	$ and $U_t^+ := \cup_{j=1}^{\kappa_t^+}\phi^{-1} (H_{t,j} )$.
	Then by \eqref{eqDtjdef}
	and property  (b) in Lemma \ref{lemcompare2},
	\begin{align*}
		\sum_{j=1}^{\kappa_t} \lambda_{d-1}(L_{t,j}
		)
		& \geq (1 - dK_1 t^{-\eta} )
	\sum_{j=1}^{\kappa_t} \tv(\phi^{-1}
	(H_{t,j}))
	\\
		& = (1 - dK_1 t^{-\eta})  \tv(U_t); \\ 
	\sum_{j=1}^{\kappa_t^+} \lambda_{d-1}(
	L_{t,j}
		) & \leq (1 + d K_1 t^{-\eta} )
	\sum_{j=1}^{\kappa_t^+} \tv(\phi^{-1}(H_{t,j}) )
	\\
		& = (1 + d K_1 t^{-\eta})  \tv(U^+_t), 
	\end{align*}
By the definition at \eqref{eqStjdef}, $U_t \subset  B \cap \partial A$
and also $\liminf_{t \to \infty} (U_t) \supset ( B \cap
\partial A)_{\partial A}^o$, since  $ \phi(  B \cap \partial A) \setminus
	\phi( U_t)$  is contained in a region within 
	distance $O( t^{-\ggamma})$ of  $\partial_{\partial \bH}
	\phi( B \cap \partial A)$.  Therefore $\tv(U_t) \to \tv( B
\cap \partial A)$ as $t \to \infty$.
	(Here we are using the assumption that
	 $\tv(\partial_{\partial A}(B \cap \partial A)) =0$.)
Also
	  $ B \cap \partial A \subset U_t^+ $ and
	$\limsup_{t \to \infty} (U_t^+) =  B \cap \partial A $,
	since  $ \phi( U_t^+  ) $  
	is contained 
	in a region within distance $O( t^{-\ggamma} +r_t + \gamma_t)$ of 
$
\phi(B \cap \partial A )$,
	where $\gamma_t$ was defined at \eqref{e:gammadef}.
	Therefore $\tv(U_t^+) \to \tv(B \cap \partial A )$ as $t \to \infty$.
	It follows that
%
	$$
	\lim_{t \to \infty} \sum_{j=1}^{\kappa_t} \lambda_{d-1}(
	L_{t,j}) 
	= \lim_{t \to \infty}
	\sum_{j=1}^{\kappa_t^+} \lambda_{d-1}(
	L_{t,j})
	= \tv ( B \cap \partial A  ).
	$$
	Combined with (\ref{0630b}) and
	(\ref{0819b}),
	this yields 
	(\ref{0502a})
	and (\ref{0819a}). 
\end{proof}

Recall from just before Lemma \ref{lemPPdom2}
the definition of $\X \ll \Y$ for  point processes.

\begin{lemm}[Second stochastic domination lemma]
	\label{lemPPdom}
There exists $t_4 \geq t_3$ such that for all $t \geq t_4$,
\bea
\cup_{j=1}^{\kappa_t} \cup_{\ell =1}^{\nu_{t,j}}
\sigma_{t,j,\ell} ( \psi_{t,j}(\Po_t \cap V) \cap T_{t,j,\ell} )
\ll
\cH_{tf_0(1+ 2d K_1 t^{-\eta})} \cap (\Gamma_t \oplus [o,4r_t e_d] ),
\label{eqPPdom}
\eea
and
\bea
\cH_{tf_0(1- 2d K_1 t^{-\eta})} \cap
	(	\cup_{j=1}^{\kappa_t^+} \cup_{\ell =1}^{\nu_{t,j}^+}
	\sigma_{t,j,\ell}
	(T_{t,j,\ell} \cap
	\psi_{t,j} \circ
	\phi^{-1}(S_{t,j}) ) )
	~~~~~~
	~~~~~~
	\nonumber \\
	~~~~~~
	~~~~~~
\ll
\cup_{j=1}^{\kappa_t^+} \cup_{\ell =1}^{\nu_{t,j}^+}
\sigma_{t,j,\ell} ( \psi_{t,j}(\Po_t \cap V) \cap T_{t,j,\ell} ).
\label{eqPPdom2}
\eea
\end{lemm}
\begin{proof}
	Let $t \geq t_3$ and $j \in [\kappa_t^+]$.
	As remarked just after \eqref{0913f}, we
	have $S_{t,j} \subset \phi(V \cap A)$.
	The point process $\Po_t \cap \phi^{-1}(S_{t,j})$
is Poisson on $\cM$
	with intensity measure $t f_0 {\bf 1}_{\phi^{-1}(S_{t,j})} v(dx)$,
where $v$ is the Riemannian volume measure on $\M$.
Hence by the Mapping theorem,
	$\psi_{t,j}(\Po_t \cap \phi^{-1}(S_{t,j}))$ is
a Poisson point process on
	$\psi_{t,j} \circ \phi^{-1} (S_{t,j})$
with
	intensity $t f_0 \mu_{t,j}(\cdot)$, where
	the measure $\mu_{t,j}$ is given
	by $\mu_{t,j}(G)= v(\psi_{t,j}^{-1}(G))$
	for measurable 
%
	$G \subset \psi_{t,j} \circ \phi^{-1}(S_{t,j})$.
	For such $G$,
	by Lemma \ref{lemcompare2}(a),
$$
	\left| \frac{\lambda_d(G)}{\mu_{t,j}(G)} - 1\right| =
	\left| \frac{\lambda_d(\psi_{t,j}(\psi_{t,j}^{-1}(G)) )
	}{v(\psi_{t,j}^{-1}(G))} - 1\right|  \leq 
	 d K_1  t^{-\eta},
$$
	so that for large $t$, we
	have $\mu_{t,j}(G) \in (1 \pm 2 d K_1  t^{-\eta})\lambda_d(G)$.
Therefore by the Superposition theorem 
we obtain that 
	\begin{align}
\cH_{tf_0(1- 2d K_1 t^{-\eta})} \cap \psi_{t,j} \circ \phi^{-1}(S_{t,j})
		& \ll
\psi_{t,j}(\Po_t \cap \phi^{-1}(S_{t,j})) 
\nonumber \\
		& \ll  
\cH_{tf_0(1+ 2d K_1 t^{-\eta})} \cap \psi_{t,j} \circ \phi^{-1}(S_{t,j}).
\label{ppdom}
	\end{align}
	Hence
	we have for all $j \in [\kappa_t^+]$ and
	$\ell \in [ \nu_{t,j}^+]$ that 
 \bea
\cH_{tf_0(1- 2d K_1 t^{-\eta})} \cap
\sigma_{t,j,\ell} (T_{t,j,\ell} 
	\cap \psi_{t,j} \circ \phi^{-1}(S_{t,j})
	)
\nonumber \\
\ll
	\sigma_{t,j,\ell}( \psi_{t,j}(\Po_t \cap \phi^{-1}(S_{t,j}))
	\cap  T_{t,j,\ell})
\nonumber	\\
\ll
\cH_{tf_0(1+ 2d K_1 t^{-\eta})} \cap
\sigma_{t,j,\ell} (T_{t,j,\ell}).
\label{0820a}
\eea
Since the sets $\phi^{-1}(S_{t,j}) \cap \psi_{t,j}^{-1}(T_{t,j,\ell})$,
$j \in [\kappa_t^+], \ell \in [\nu_{t,j}^+]$, are 
disjoint, the Poisson processes $\Po_t \cap \phi^{-1}(S_{t,j})
\cap \psi_{t,j}^{-1}(T_{t,j,\ell}),$ $j \in [\kappa_t^+],
\ell \in [\nu_{t,j}^+]$, are mutually independent. 
Since the sets $\sigma_{t,j,\ell}(T_{t,j,\ell}),  j \in
	[\kappa_t^+],  \ell \in [\nu_{t,j}^+]$, are disjoint,
	taking unions over such $(j,\ell)$ and using the
	first relation of (\ref{0820a}) and the
	Superposition theorem yields 
	(\ref{eqPPdom2}). Similarly using the second relation
in (\ref{0820a}) and taking the union only over
$j \in [\kappa_t], \ell \in [\nu_t]$ yields
	(\ref{eqPPdom}). 
\end{proof}


As in Section \ref{seclastpf}, we define
	$r_t^+$ and $r_t^-$ by 
	\eqref{rtpmdef} (but now with $r_t$ satisfying \eqref{rt1}.)

\begin{lemm}[Covering reassembled mini-slabs in $\bH$]
	\label{lemfromCov}
	Given $a_1 >0$, we have
	\begin{align}
		& \lim_{t \to \infty} 
\Pr[ ((\Gamma_t^*)^{(2r_t)} \oplus [o,a_1 r_t e_d]) \subset
		F_{r_t^-}( \cH_{tf_0(1- 6dK_1 t^{-\eta})}
\cap \bH
	) ] 
\nonumber \\
		& =
\lim_{t \to \infty} 
		\Pr[ \Gamma_t  \subset F_{r^+_t}( \cH_{tf_0(1+ 6dK_1 t^{-\eta})}
	\cap \bH
	) ] 
\nonumber \\
		& = \exp( -c_{d,k} \tv(B  \cap \partial A ) e^{-\zeta}).
\label{0616c}
	\end{align}
\end{lemm}
\begin{proof}
The limiting statement (\ref{rt1})
still holds if we replace $r_t$ with $r_t(1\pm   8dK_1 t^{-\eta})$
and $t$ with $t(1 \pm 6d K_1 
	t^{-\eta})$. Therefore the last line of (\ref{0616c})
follows by
(\ref{0502a}) and
	\cite[Lemma 7.4]{ECover}.
	Also, using (\ref{0819a}),
	we have
	$$
	\lim_{t \to \infty} \Pr [ (\Gamma_t^*)^{(2r_t)} 
	\subset F_{r_t^-} (\cH_{t f_0(1- 6dK_1 t^{-\eta})} \cap
	\bH ) ]= \exp(-c_{d,k} \tv (B \cap \partial A) e^{-\zeta} ).
	$$
	Then for the second line of (\ref{0616c}) we can also
	use \cite[Lemma 7.4]{ECover}, provided we have
	\bea
	\lim_{t \to \infty}
	\Pr[ (\Gamma_t^* \setminus  (\Gamma_t^*)^{(2r_t)}) \oplus 
	[o,a_1 r_t e_d]
	\subset F_{r_t} (\cH_{tf_0} \cap \bH) ] =1.
	\label{0624a}
	\eea
	
		We shall apply Lemma \ref{lemmeta} in the Euclidean setting to see \eqref{0624a}.
		Let $D:= [-1,(\tv(B \cap \partial A))^{1/(d-1)} +1]^{d-1}$.
		Then $D$ is a $(d-1)$-dimensional cube and
		by Lemma \ref{lemsurf}, $D$ contains
		$\Gamma_t^*$ and $\Gamma_t$ in
		its interior for all large enough $t$.
	Let $\mu^*$ denote the restriction of the measure
	$f_0 \lambda_d $ to $D \times [0,(f_0\lambda_{d-1}(D))^{-1}]$,
	i.e. the uniform probability measure on this set,
	and for $t>0$ let $\Po^*_t$ denote
	a Poisson process on $\R^d$ with intensity measure
	$t \mu^*$.
	Then $\Po^*_t \ll \cH_{tf_0} \cap \bH$, so to prove
	\eqref{0624a} it suffices to prove that
	\begin{align}
		\lim_{t \to \infty}
		\Pr[ (\Gamma_t^* \setminus  (\Gamma_t^*)^{(2r_t)}) \oplus 
	[o,a_1 r_t e_d]
		\subset F_{r_t} (\Po_t^*)  ] = 1. 
		\label{e:0904d}
	\end{align}

We can cover $(\Gamma_t^* \setminus (\Gamma_t^*)^{(2r_t)})
		\oplus [o,a_1 r_t e_d]$ by $m_t$ balls of radius $r_t$ with
		$m_t =  O(r_t^{-(d-2)})$.
	For any $x\in (\Gamma_t^* \setminus  (\Gamma_t^*)^{(2r_t)}) \oplus 
	[o,a r_t e_d]$ we have $\mu^*(B_{\R^d}(x,s)) \ge f_0\vvol_d s^d/2$ for 
	$s \in (0,r_t]$. 
	Recalling \eqref{rt1}, we observe that 
	the condition $ua>b/d$ of Lemma \ref{lemmeta}(iii)
	holds with $u=2(d-1)/(d f_0\vvol_d)$, 
	$a=\vvol_d f_0/2$ and $b=d-2$. Then
	\eqref{e:0904d} follows.
\end{proof}

Recall the definition of $Q_{t,j,\ell}^-$ near \eqref{eqTminus}.
Given $t$, we shall sometimes write just
$\cup_{j,\ell}$ for $\cup_{j=1}^{\kappa_t} \cup_{\ell=1}^{\nu_{t,j}}$
and
$\cup_{j,\ell}^*$ for $\cup_{j=1}^{\kappa_t^+} \cup_{\ell=1}^{\nu_{t,j}^+}$.

\begin{lemm}[Covering a tartan region in $\bH$]
	\label{lemplaid}
	It is the case that
	\bea
	\lim_{t \to \infty}
	\Pr[ \cup_{j=1}^{\kappa_t^+}
	\cup_{\ell =1}^{\nu_t^+} \sigma_{t,j,\ell}(Q_{t,j,\ell}
	\setminus Q_{t,j,\ell}^-)
	\subset 
	F_{r_t^+} ( \cH_{tf_0(1+ 6dK_1 t^{-\eta})} \cap \bH) ]
	=1.
	\label{0619c}
	\eea
\end{lemm}
	\begin{proof}
Let $\mu^*$ and $\Po_t^*$ be as in the proof of Lemma \ref{lemfromCov}.
		Then the pre-limit in the left hand side of \eqref{0619c}
		is bounded below by
		$\Pr[ \cup_{j,\ell}^* 
		\sigma_{t,j,\ell}(Q_{t,j,\ell} \setminus Q_{t,j,\ell}^-)
		\subset F_{r_t}(\Po_t^*)]$.

Let $\alpha' \in (\alpha,1/d)$.  We claim that
 \bea
		\kappa(\cup^*_{j,\ell}  \sigma_{t,j,\ell}
		(Q_{t,j,\ell} \setminus Q_{t,j,\ell}^-), r_t)  
		= O( t^{\ggamma(d-1)} \times (t^{-\ggamma}/r_t)^{d-2} )
		= O(r_t^{-d(\alpha'+1-(2/d))})
		\label{0619a}
 \eea
 Indeed, the number of  sets in the union is $O(t^{\ggamma(d-1)})$, while
		for each $(j,\ell)$ the set 
 $\sigma_{j,\ell} ( Q_{t,j,\ell} \setminus Q_{t,j,\ell}^-) $
 is a $(d-1)$-dimensional hypercubic annular region of
 thickness $O(r_t)$ and diameter $O(t^{-\ggamma})$, and
 therefore can be covered by $O((t^{-\ggamma}/r_t)^{d-2})$ balls of
		radius $r_t$.  This gives the first part of (\ref{0619a}),
		and the second part comes from (\ref{rt1}).

		For all large enough $t$, all $y\in
		\cup_{j,\ell}^*  \sigma_{t,j,\ell} (Q_{t,j,\ell} \setminus
		Q_{t,j,\ell}^- )$ and $s \in (0,r_t]$,
		we have $\mu^*(B_{\R^d}(y,s))\ge f_0\vvol_d s^d/2$.
		The condition $ua>b/d$ of Lemma \ref{lemmeta}(iii)
		holds with $u=2(d-1)/(d f_0 \vvol_d), a = f_0\vvol_d /2$ and
		$b/d=\alpha'+1-2/d$. The result follows.
	\end{proof}

\begin{lemm}[Covering boundary region in $A$: upper bound]
	\label{lem0619}
	It is the case that
	\bea
	\limsup_{t \to \infty}
	\Pr[ B  \cap \partial A \subset F_{r_t}(\Po_t) ]
	\leq \exp(-c_{d,k} \tv( B \cap \partial A ) e^{-\zeta} ).
	\label{0619d}
	\eea
\end{lemm}
	\begin{proof}
		Let $t >0$.  Suppose event
		$\{  B \cap \partial A \subset F_{r_t}(\Po_t) \}$
		occurs.
		Let 
		$
		y \in  
		\cup_{j, \ell} 
		\sigma_{t,j,\ell}(Q_{t,j,\ell}^-). 
		$ 

		Take $j_0 \in [\kappa_t]$,
		$\ell_0 \in [\nu_{t,j_0}]$ and $z \in Q_{t,j_0,\ell_0}^-$
		such that $y = \sigma_{t,j_0,\ell_0}(z)$. By \eqref{eqTdef}
	and \eqref{e:0813a} we have
		 $ Q_{t,j_0,\ell_0} \subset T_{t,j_0,\ell_0}
		\subset \psi_{t,j_0} \circ \phi^{-1} (S_{t,j_0})$,
		we can and do take
 $u \in \phi^{-1}(S_{t,j_0})$ such that 
$z = \psi_{t,j_0}(u)$, and hence 
		$y = \sigma_{t,j_0,\ell_0} \circ \psi_{t,j_0}(u)$. 
		Then
%
\bean
		z \in T_{t,j_0,\ell_0} \cap \partial \bH
		\subset \psi_{t,j_0}\circ \phi^{-1}(S_{t,j_0})
		\cap \partial \bH 
		\subset \psi_{t,j_0}(B \cap \partial A),
\eean
so that $u \in  B \cap \partial A$. Hence by our assumption
		$u \in F_{r_t}(\Po_t)$, there are at least $k$ points
$w$ of $\Po_t \cap B(u,r_t)$. 
For each such  $w$, since $u \in B, w \in A$
and $\dist(u,w) \leq r_t$, provided $t $ is large enough $w \in V \cap A$.
		Using Lemma \ref{lemcompare3}(e) and recalling the definition of $x_{t,j}$
		and $\psi_{t,j}$ from just before Lemma \ref{lemcompare2},
we have 
		\begin{align}
			\|\psi_{t,j_0}(w) - \psi_{t,j_0}(u)\| &
			\leq r_t (1+  K_1
		(\dist(w,u) + \dist(u,x_{t,j_0})) )
\nonumber \\
			& \leq r_t ( 1+ K_1 ( r_t +  2d t^{-\eta})),
\label{0616a}
		\end{align}
where the second inequality uses the fact that both $u$ and $x_{t,j_0}$
lie in $\phi^{-1}(S_{t,j_0})$, so $\|\phi(u) - \phi(x_{t,j_0}) \|
\leq d t^{-\eta}$ and so by Lemma  \ref{lemcompare2} again,
$\dist(u,x_{t,j_0}) \leq 2d t^{-\eta}$.

In particular, $\|\psi_{t,j_0}(w) - \psi_{t,j_0}(u)\| \leq 2 r_t$, and
therefore since $z= \psi_{t,j_0}(u) \in T_{t,j_0,\ell_0}^-$ 
		we have 
$\psi_{t,j_0}(w) \in T_{t,j_0,\ell_0}$. Since $\sigma_{t,j_0,\ell_0}$ is
an isometry and $r_t \leq d K_1 t^{-\eta}$ we also have  from (\ref{0616a}) 
		for $t$ large that
$$
\|
\sigma_{t,j_0,\ell_0} \circ \psi_{t_0,j_0}(w) - y
\| \leq r_t(1+ 3dK_1 t^{-\eta}) ,
$$
and hence 
$$
y \in
F_{r_t^+} ( \sigma_{t,j_0,\ell_0} 
( \psi_{t,j_0}(\Po_t \cap V) \cap T_{t,j_0,\ell_0}))
\subset
F_{r_t^+} ( \cup_{j=1}^{\kappa_t} \cup_{\ell=1}^{\nu_{t,j}} \sigma_{t,j,\ell} 
( \psi_{t,j}(\Po_t \cap V) \cap T_{t,j,\ell})).
$$
Therefore we have the event inclusion 
\begin{align*}
	\{ B  \cap  \partial A \subset F_{r_t}(\Po_t) \}
	 \subset & \{
		 \cup_{j,\ell}
	\sigma_{t,j,\ell}(Q_{t,j,\ell}^-)
	 \subset 
	F_{r_t^+} (
	\cup_{j,\ell}
\sigma_{t,j,\ell} ( \psi_{t,j}(\Po_t \cap V) \cap T_{t,j,\ell} ) )
\}
.
\end{align*}
Then using \eqref{eqPPdom} from Lemma \ref{lemPPdom}, we obtain that
\begin{align*}
\limsup_{t \to \infty} \Pr [ B  \cap \partial A \subset
	F_{r_t}(\Po_t) ]  
\leq 
	\limsup_{t \to \infty} \Pr[ & 
(\cup_{j,\ell} \sigma_{t,j,\ell}(Q^-_{t,j,\ell}) )
\nonumber \\
	&
	\subset F_{r_t^+} (\cH_{tf_0(1+6dK_1 t^{-\eta})} \cap (\Gamma_t \oplus 
[o,4r_t e_d] ))
].
\end{align*}

Since
\bean
\Pr[ 
(\cup_{j,\ell} \sigma_{t,j,\ell} (Q_{t,j,\ell}^-) ) 
\subset F_{r_t^+}(\cH_{tf_0(1+6dK_1t^{-\eta}) }  \cap \bH) ]
\leq \Pr[ \Gamma_{t} 
\subset F_{r_t^+}(\cH_{tf_0(1+6dK_1t^{-\eta}) }  \cap \bH) ]
\\
+ \Pr [ \{ 
	\cup_{j,\ell} \sigma_{t,j,\ell} (Q_{t,j,\ell} \setminus 
	Q_{t,j,\ell}^- ) \subset 
	F_{r_t^+} ( \cH_{tf_0(1+ 6dK_1 t^{-\eta})} \cap \bH) \} ^c ],
\eean
it then follows from Lemma \ref{lemplaid}
that
\bean
\limsup_{t \to \infty} \Pr [ B \cap \partial A  \subset
F_{r_t}(\Po_t) ] 
\leq 
\limsup_{t \to \infty} \Pr[ \Gamma_{t}  
\subset F_{r_t^+}(\cH_{tf_0(1+6dK_1 t^{-\eta}) }  \cap \bH) ],
\eean
and therefore by Lemma \ref{lemfromCov}
we obtain (\ref{0619d}).
\end{proof}

	\begin{lemm}[Covering boundary region in $A$ with tartan set removed:
		lower bound]
		\label{lemcovA}
		It is the case that
		\bea
		\liminf_{t \to \infty} \Pr[ \cup_{j,\ell}^* (
		\phi^{-1}(S_{t,j}^-) \cap
		\psi_{t,j}^{-1 }(T_{t,j,\ell}^-)
		)
		\subset F_{r_t}(\Po_t) ] 
		\geq \exp(-c_{d,k} 
		\tv( B \cap \partial A ) e^{-\zeta}).
		\label{0706a}
		\eea
	\end{lemm}
	\begin{proof}
		Let $t >0$.
For this proof, we redefine the event $E_t$ by
\bea
		E_t : = \{ ((\Gamma_t^*)^{(2r_t)} \oplus [o,2r_te_d]) \subset
		F_{r_t^-} (
		\cup_{j,\ell}^* (\sigma_{t,j,\ell}  
		[\psi_{t,j}(\Po_t \cap \phi^{-1}(S_{t,j})) \cap T_{t,j,\ell}]
\nonumber		\\
		\cup
		[ \cH_{tf_0(1-2 d K_1  t^{-\eta})} \cap \sigma_{t,j,\ell}
		(T_{t,j,\ell} \setminus \psi_{t,j} \circ \phi^{-1}
		(S_{t,j})) ])
		) 
\},
		\label{0820b}
\eea
		where we assume here that the Poisson process 
		$\cH_{tf_0(1- 2d K_1  t^{-\eta})}$ is independent of $\Po_t$.
		By (\ref{eqPPdom2}), and the Superposition theorem,
		the point process
		$$
		\cup_{j,\ell}^* (\sigma_{t,j,\ell}  
		[\psi_{t,j}(\Po_t \cap \phi^{-1}(S_{t,j})) \cap T_{t,j,\ell}]
		\\
		\cup
		[ \cH_{tf_0(1-2 d K_1  t^{-\eta})} \cap \sigma_{t,j,\ell}
		(T_{t,j,\ell} \setminus \psi_{t,j} \circ \phi^{-1}
		(S_{t,j})) ])
		$$
		stochastically dominates 
		$\cH_{tf_0(1- 2d K_1  t^{-\eta})} \cap (\Gamma_t^* \oplus
		[o,4 r_te_d])$. Therefore
		by
		Lemma \ref{lemfromCov},
\bea
\liminf_{t \to \infty} 
\Pr[ E_t ] 
		\geq \exp( -c_{d,k} \tv( B  \cap \partial A )  e^{-\zeta}).
\label{0505a}
\eea

Suppose $E_t$ occurs.
		Let $x \in \cup_{j,\ell}^* (
		\phi^{-1}(S_{t,j}^-)
		\cap
		\psi_{t,j}^{-1 }(T_{t,j,\ell}^-)
		)$.
		Take $j_0 \in [\kappa_t]$ and
		$\ell_0 \in [\nu_{t,j_0}]$
		such that 
		$\phi(x) \in S_{t,j_0}^-$
		and
		$\psi_{t,j_0}(x) \in T_{t,j_0,\ell_0}^-$.
Then
$$
y:= \sigma_{t,j_0,\ell_0} \circ \psi_{t,j_0}(x)
\in \sigma_{t,j_0,\ell_0}(T_{t,j_0,\ell_0}^-)
		\subset (\Gamma_t^*)^{(2r_t)} \oplus [o,2 r_t e_d],
$$
		and since we assume $E_t$ occurs,
		\bean
		\card (
B(y, (1- 6dK_1 t^{-\eta})r_t)
\cap
		\cup_{j,\ell}^* (\sigma_{t,j,\ell}
[\psi_{t,j}(\Po_t \cap \phi^{-1}(S_{t,j}))  \cap T_{t,j,\ell}]
\\
\cup 
		[ \cH_{tf_0(1-2 d K_1  t^{-\eta})} \cap \sigma_{t,j,\ell}
		(T_{t,j,\ell} \setminus \psi_{t,j} \circ \phi^{-1}
		(S_{t,j})) ])
) \geq k.
\eean
Since
		$\psi_{t,j_0}(x) \in T_{t,j_0,\ell_0}^-$,
		and the sets $\sigma_{t,j,\ell}(T_{t,j,\ell})$,
		$j \in [\kappa_t^+]$, $\ell \in [\nu^+_{t,j}]$ are
		disjoint,
$$
B(y,(1- 6dK_1 t^{-\eta})r_t) \cap (
\cup_{j,\ell}^* \sigma_{t,j,\ell}( T_{t,j,\ell} ) ) 
\subset 
\sigma_{t,j_0,\ell_0}( T_{t,j_0,\ell_0} ), 
$$
so that there exist at least $k$ points $\tiy$ satisfying
\bean
\tiy \in B(y,(1- 6dK_1 t^{-\eta})r_t) \cap ( 
\sigma_{t,j_0,\ell_0}[
	\psi_{t,j_0}(\Po_t \cap \phi^{-1}(S_{t,j_0}))
	\cap
	T_{t,j_0,\ell_0}]
~~~~~~~~
\\
~~~~~~~~
\cup 
		[ \cH_{tf_0(1-2 d K_1  t^{-\eta})} \cap \sigma_{t,j_0,\ell_0}
		(T_{t,j_0,\ell_0} \setminus \psi_{t,j_0} \circ \phi^{-1}
		(S_{t,j_0})) ]) .
\eean
Given such a $\tiy$, set
$\tx := \psi_{t,j_0}^{-1} (\sigma_{t,j_0,\ell_0}^{-1}(y))$.
Then $\|\psi_{t,j_0}(x) - \psi_{t,j_0}(\tx)\| =
\|\sigma_{j_0,\ell_0}^{-1}(y) -
\sigma_{j_0,\ell_0}^{-1}(\tiy) \| \leq r_t$, and
by Lemma \ref{lemcompare3}(f), $\dist(x,\tx) \leq 2 r_t$.
By Lemma \ref{lemcompare3}(e) and the definition of
$\psi_{t,j}$ just before Lemma \ref{lemcompare2},
\begin{align}
\left| \frac{\| \psi_{t,j_0}(x)- \psi_{t,j_0}(\tx)\|}{\dist(x,\tx)}
-1 \right|
	& \leq  K_1 (\dist(x,\tx) + \dist(x, x_{t,j_0}) )  
\nonumber \\
	& \leq  K_1 (2 r_t + \dist(x, x_{t,j_0}) ),  
\label{0806c2}
\end{align}
where $x_{t,j_0}$ is an element of $\phi^{-1}(S_{t,j_0})$.

Since $x$ and $x_{t,j_0}$ both lie in $\phi^{-1}(S_{t,j_0})$,
$
\dist(x, x_{t,j_0}) \leq \sqrt{d} \dist(\phi(x),
\phi(x_{t,j_0}) ) \leq d t^{-\eta}.
$
Thus by (\ref{0806c2}) we obtain that
$$
\| \psi_{t,j_0}(x) - \psi_{t,j_0}(\tx) \| \geq (1- 3d K_1 t^{-\eta})
 \dist (x,\tx),
$$
and hence  $\dist(x,\tx) \leq (1+ 6d K_1 t^{-\eta})
\| \psi_{t,j_0}(x) - \psi_{t,j_0}(\tx)\|$. Thus, 
%
\bean
\dist(\tx ,x ) \leq
(1+ 6d K_1 t^{-\eta})
\|\tiy - y\|
\leq r_t,
\eean
so $\|\phi(x) - \phi(\tx)\| \leq 2 r_t$,
and therefore since we assume $x \in \phi^{-1}(S_{t,j_0}^-)$ 
we have $\tx \in \phi^{-1} (S_{t,j_0}) $, 
so that $\tiy$ is not in the second set on the right hand side of 
 (\ref{0820b}),
and hence
$\tx \in \Po_t $. Therefore $ x \in F_{r_t} (\Po_t ), $
so the event on the left side of (\ref{0706a}) occurs.
Then the result follows from (\ref{0505a}).
\end{proof}

Recall from \eqref{eqStjdef} the definitions of $S_{t,j}$, $S_{t,j}^*$ and
$S_{t,j}^-$ for $t >0, j \in [\kappa_t^+]$.  As explained there,
$S_{t,j} \subset \phi(V \cap A)$ for large enough $t$ and all
$j \in [\kappa_t^+]$.  The next two lemmas show that the `tartan` region 
	near $\partial A$ that was left out in Lemma \ref{lemcovA}
	is likely to be covered.

	\begin{lemm}
		\label{lemAug21}
		As $t \to \infty$, $\Pr[ \cup_{j=1}^{\kappa_t^+} 
		\phi^{-1}(S^*_{t,j} \setminus S_{t,j}^-) \subset 
		F_{r_t} ( \Po_{t} ) ] \to 1$. 
	\end{lemm}
	\begin{proof}
		Let $\eta' = 1/(2d)$ (so $0 < \eta < \eta'$)
		and recall $\eps' $ is as at \eqref{e:defeps'}.
		We claim that 
 \bea
		\kappa( \cup_{j=1}^{\kappa_t^+} \phi^{-1} (S^*_{t,j}  
		\setminus S_{t,j}^-)  ,r_t)
		= O( t^{\eta(d-1)} \times (t^{-\eta}/r_t)^{d-2} )
			= O(r_t^{-d(\eta'+1-(2/d))})
		.
		\label{0619a22}
 \eea
	Indeed,  for each  $j$, as discussed in the proof 
	of Lemma \ref{lemplaid}, the set
	 $ (S_{t,j}^* \setminus S_{t,j}^-)$ 
	can be covered by $O((t^{-\eta}/r_t)^{d-2})$ Euclidean balls
	of radius $\eps' r_t/2$.
	The mapping $\phi$ changes  all distances by a factor of
	at most 2, and therefore  also
	$\phi^{-1}(S_{t,j}^* \setminus S_{t,j}^-)$
	can be covered by $O((t^{-\eta}/r_t)^{d-2})$ geodetic
	balls of radius $r_t$.  Since
		$\kappa_t^+ = O(t^{\eta(d-1)})$,  
	the first part of \eqref{0619a22} follows.
		Then the second part comes from (\ref{rt1}).

		Let $y\in\cup_{j=1}^{\kappa_t^+} 
		\phi^{-1}(S^*_{t,j} \setminus S_{t,j}^-)$.
		Then $y \in V \cap A$. 
		By Lemma \ref{lemcompare3}, similarly to \eqref{e:vBLB},
	we have that $\mu(B(y,s))\ge (1- \eps')^{d+1}f_0\vvol_d s^d/2$
		for $0 < s \leq r_t$. Hence, the condition $ua>b/d$ of
		Lemma \ref{lemmeta}(iii) holds with
		$u=2(d-1)/(df_0\vvol_d)$, $a=(1-\eps')^{d+1}
		f_0\vvol_d/2$ and 
		$b/d=1+\eta'-2/d = 1- 3/(2d)$ (note that \eqref{e:defeps'}
		implies that $(1-  \eps')^{d+1} > 1 - (1/(2d-2))$).
		The result follows.
	\end{proof}

	\begin{lemm}
		\label{lemJuly}
		As $t \to \infty$, $\Pr[ \cup_{j,\ell}^* 
		\phi^{-1}(S_{t,j}) \cap  
		\psi_{t,j}^{-1} (T_{t,j,\ell} \setminus T_{t,j,\ell}^-) 
		\subset F_{r_t} ( \Po_{t} ) ] \to 1$. 
	\end{lemm}
	\begin{proof}
		Set $\alpha'= 7/(8d)$ so that $\alpha<\alpha'<1/d$.
		Similarly to in the proof of
	of Lemma \ref{lemplaid}, we have
 \bea
		\kappa( \cup_{j,\ell}^* \phi^{-1}(S_{t,j}) \cap
		\psi_{t,j}^{-1} (T_{t,j,\ell}  \setminus T_{t,j,\ell}^-),
		 r_t)
		= O( t^{\ggamma}  r_t^{2-d} )
			=O(r_t^{-d(\alpha'+1-2/d)}) .
		\label{0619a2}
 \eea
%

		As in the previous lemma, we have 
		$\mu(B(x,s))\ge (1- \eps')^{d+1}f_0\vvol_d 
		s^d$ for $t$ large, $s < r_t$.
		The result follows by applying 
		Lemma \ref{lemmeta}(iii) with
		$u= 2(d-1)/(df_0 \vvol_d)$, $a = (1- \eps')^{d+1} f_0
		\vvol_d/2$ and $b/d = \alpha' + 1 -2/d= 1-(9/(8d))$.
		The condition $u > b/(ad)$
		amounts to $(1-   \eps')^{d+1}>1 -(1/(8d-8))$,
		which follows from \eqref{e:defeps'}.
	\end{proof}

    \begin{lemm}[Covering the boundary region: lower bound] 
	    \label{lemrind}
	    It is the case that
	   $
	    \liminf_{t \to \infty} \Pr[
		    B \setminus A^{(r_t)} \subset F_{r_t}
		    (\Po_t) 
		    ] \geq \exp(-c_{d,k}
		    \tv( B \cap \partial A) e^{-\zeta}).
	$
    \end{lemm}
    \begin{proof}
%


	    Let $x \in B \setminus A^{(r_t)}$.
	    Then there exists $\tx \in \partial A$ with
	    $\dist(x,\tx) \leq r_t$.  Provided
	    $t$ is large enough, both $x$ and $\tx$ lie in $V$,
	    and by Lemma \ref{lemcompare3}(f), $\|\phi(x)
	    -\phi(\tx)\| \leq 2r_t$.  Moreover
	    $\phi(\tx) \in  \partial \bH$, so
	    $\dist (\phi(x),\partial \bH) \leq 2r_t$.

	    Let $y$ be the point in $\partial \bH$  closest to
	    $\phi(x)$. By the above $\|\phi(x) -y \| \leq 2r_t$.
	    Also $\dist(x,B \cap \partial A) \leq \gamma_t$
	    by \eqref{e:gammadef}, so
	    $\dist(\phi(x), \phi(B \cap \partial A)) \leq 2 \gamma_t$
	    and $\dist(y, \phi(B \cap \partial A)) \leq 2 (\gamma_t +r_t)$.
	    Provided $t$ is large enough,
	    $\phi(x) \in \cup_{j=1}^{\kappa_t^+} S^*_{t,j}$.
	    Hence by \eqref{eqStjdef},
	    \begin{align*}
		    B \setminus A^{(r_t)} \subset \big( \cup_{j=1}^{\kappa_t^+}
			    \phi^{-1}(S^*_{t,j} \setminus S_{t,j}^-)
			    \big)
			    \cup 
			    \big(
			    \cup_{j,\ell}^*
			    \phi^{-1}(S_{t,j}^-) \cap \psi_{t,j}^{-1}
			    (T_{t,j,\ell} \setminus T_{t,j,\ell}^-)
			    \big)
			    \\
			    \cup \big(
			    \cup_{j,\ell}^*
			    \phi^{-1}(S_{t,j}^{-1})
			    \cap \psi_{t,j}^{-1}(T_{j,\ell}^-)
			    \big).
	    \end{align*}
	    Hence by Lemmas \ref{lemAug21} and \ref{lemJuly},
	    \bean
	    \liminf_{t \to \infty}
	    \Pr[ B \setminus A^{(r_t)} \subset F_{r_t}(\Po_t)]
	    \geq
	    \liminf_{t \to \infty}
	    \Pr[ \cup^*_{j,\ell}( \phi^{-1}(S^-_{t,j}) \cap
	    \psi_{t,j}^{-1} (T_{t,j,\ell}^-) ) \subset F_{r_t}(\Po_t) ],  
	    \eean
	    and then using Lemma \ref{lemcovA}, 
	    we obtain
	    the result.
    \end{proof}

\begin{proof}[Proof of Proposition \ref{propcompare}]
	By Lemmas \ref{lem0619} and \ref{lemrind},
	given $B \in \cB$ with $B \subset V$
	we have  (\ref{eqHausd}).
	Combined with the discussion just after the statement of
	Proposition \ref{propcompare},
	this completes  the proof.
\end{proof}

\begin{lemm}[Any point in $\partial A$ has a neighbourhood with nice boundary] 
	\label{lemWexist}
	Let $x_0 \in \partial A$ and let $V \subset \cM$ be
	open with $x_0 \in V$. Then
	there exists an open set
$W \subset \cM$  with $x_0 \in W$, and
	also
$\overline{W} \subset V$, and also
	$v(\partial W) =0$ and $\tv(\partial_{\partial A}
	(W \cap \partial A)) =0$, and
$\limsup_{r \downarrow 0} r^{d-2} \kappa(
	\partial_{\partial A}( W  \cap \partial A),r) < \infty$.
\end{lemm}
\begin{proof}
	Assume without loss of generality that there exists $\phi:V \to 
	\R^d$ such that $(V,\phi)$ has the properties described in
	Lemma \ref{lemcompare}, taking $\eps = 1/99$ there
	(if not, one can take a smaller
	set $V$).
	Then we can and do choose $r_0>0$ such that  $B_{\R^d}(\phi(x_0),2r_0) \subset
	\phi(V)$. Take $W = \phi^{-1}(B_{\R^d}(\phi(x_0),r_0))$. 
	
	Then  $W $ has the
required properties. Indeed, 
	$\lambda_d(\phi(\partial W)) = 0 $ so $v(\partial W) =0$. Also
	$\phi(\partial_{\partial A}( W  \cap \partial A))
	= \partial B_{\R^d}(\phi(x_0),r_0)
	\cap \partial \bH$, which is a $(d-2)$-dimensional
sphere of radius $r_0$. Therefore 
	$\lambda_{d-1}(\phi(\partial_{\partial A}( W \cap \partial A))) =0 $,
	so that $\tv(\partial_{\partial A}( W \cap \partial A)) =0 $.
	Also
	$\phi(\partial_{\partial A}( W \cap \partial A)) $
	can be covered by $O(r^{2-d})$
balls of radius $r/2$, and the pre-images of these balls 
under $\phi$ are contained in geodetic balls of radius $r$ in $\cM$,
	by Lemma \ref{lemcompare3}(f),
	and cover $ \partial_{\partial A} (W \cap \partial A)$, 
so that
$\limsup_{r \downarrow 0} r^{d-2} \kappa(
	\partial_{\partial A}( W  \cap \partial A),r) < \infty$.
\end{proof}

By Proposition \ref{propcompare}, Lemma \ref{lemWexist} and 
a compactness argument we can and do take a finite collection of
quadruples  $(x_i, V_i,  W_i, \phi_i)$, $1 \leq i \leq m $,
where for each $i \in [m]$ we have $x_i \in  W_i \cap \partial A$
and $V_i$, $W_i$ 
are open sets in $\M$
with $ \partial A \subset \cup_{i=1}^m W_i $,
and with
$\overline{W_i} \subset V_i$, and
$\overline{W_i} \in \cB$, and
$\limsup_{r \downarrow 0} r^{d-2} \kappa (\partial_{\partial A}
(W_i \cap \partial A),r) 
< \infty$,
and $\phi_i: V_i \to \R^d$ is an injection satisfying:

	(a) $\phi_i(V_i \cap A) = \phi_i(V_i) \cap \bH$;

	(b) For all distinct 
	$y,z \in V_i$ we have
	$|(\|\phi_i(y)-\phi_i(z)\|/ \dist(y,z)) - 1| \leq \eps'$;

(c)
	 For all measurable $ F \subset V_i$ we have 
	$|(\lambda_d(\phi_i(F)) / v(F))-1| \leq \eps'$; and 
	
	(d) for all $B \in \cB$ with $B \subset V_i$
	we have (\ref{eqHausd}). \\

	Let $\Delta_t^{\rm bdy}$
	denote the set $\cup_{i=1}^m \cup_{x \in \partial_{\partial A} (W_i
	\cap \partial A)
	} B(x,99r_t)  
	$. We call this the {\em patch boundary region}; it is near the
	boundaries of the `patches' $W_i$, and also near
	the boundary of $A$.

	\begin{lemm}[Covering patch boundary region]
		\label{lemDelta}
		$\lim_{t \to \infty}
		\Pr[ \Delta_t^{\rm bdy} \subset F_{r_t}(\Po_t) ] =1.  $
	\end{lemm}
	\begin{proof}
	We claim that $\kappa(\Delta_t^{\rm bdy}, r_t)=O(r_t^{2-d})$.
		Indeed we can cover $ \cup_{i=1}^m \partial_{\partial A}
		(W_i \cap \partial A)$ with $O(r_t^{2-d})$ balls of radius 
		$r_t$.
		The corresponding balls of radius $100r_t$ cover
		$\Delta_t^{\rm bdy}$, and each ball of radius
		$100r_t$ can be covered by a finite number
		of balls of radius $r_t$.
		
	We claim moreover that $\mu(B(x,s))\ge (1-  \eps')^{d+1}
		f_0 \vvol_d s^d /2$ for all large enough $t$ and
		any $x\in \Delta_t^{\rm bdy}$, $s \in (0,r_t]$.
		Indeed, using properties (c) and (d) above,
		we have
		\begin{align*}
v(B(x, r_t) \cap A) & \geq (1-  \eps') \lambda_d (\phi(B(x, r_t)) \cap \bH)
 \\
	& \geq (1- \eps') \lambda_d( B_{\R^d}(\phi(x),r_t(1-\eps')) \cap \bH),
		\end{align*}
	 and the claim follows.
	We conclude by applying Lemma \ref{lemmeta}(iii)
	with $a= (1 -  \eps')^{d+1} \vvol_d f_0/2$ and
	$b/d = (d-2)/d$ and $u = (2/(\vvol_d f_0)) (d-1)/d$.
		For $u> b/(ad)$ 
		we need $(1-\eps')^{d-2} > 1- 1/(d-1)$, which holds
	by \eqref{e:defeps'}.
	\end{proof}
	\begin{prop}[Covering a macroscopic region near $\partial A$]
		\label{prop:rind}
		Let $B \in \cB$. Then
\bea
		\lim_{t \to \infty} \Pr[ B \setminus A^{(r_t)} \subset 
		F_{r_t}(\Po_t)
		] = \exp (- c_{d,k} \tv(B \cap \partial A)e^{-\zeta}).
\label{0719a}
\eea
	\end{prop}
	\begin{proof}
		Set $W_1^* := W_1$, and for $i = 2,3,\ldots,m$ set $W_i^* :=
	W_i \setminus \cup_{j=1}^{i-1} W_j$.
	Then  $\partial A \subset \cup_{i=1}^m W_i^*$,
	and $W_1^*,\ldots,W_m^*$ are disjoint.
	For each $i \in [m]$, and $r>0$, let 
	$$
	(W_i^* \cap \partial A)^{(r)}
	:= ( W_i^* \cap \partial A) \setminus 
		\cup_{x \in \partial_{\partial A} (W^*_i \cap \partial A)}
		B(x,r).
	$$
	Now let $i,j \in [m]$ be distinct. We claim that
	for $t$ sufficiently large,
	\bea
	\dist (
	(W_i^* \cap \partial A)^{(90r_t)} , 
	(W_j^* \cap \partial A)^{(90r_t)}  
	) \geq 50r_t.
	\label{0708a}
	\eea
	Indeed, let $x \in (W_i^* \cap \partial A)^{(90r_t)}$, 
	 $y \in (W_j^* \cap \partial A)^{(90r_t)}$. 
	 Since $W_i^* \subset W_i$, we have that
	 $\dist(W_i^*, \cM \setminus V_i) >0$.
	 Therefore provided $t$ is sufficiently large,
	 $\dist(W_i^*, \cM \setminus V_i) > 50r_t$, so
	 if $y \notin V_i$ then  $\dist (x,y) \geq 50r_t$.
	 On the other hand,
	 if $y \in V_i$, then by Proposition  \ref{propcompare},
	 \bean
		\dist(x,y) & \geq & (1+  \eps')^{-1} \| \phi_i(x) - \phi_i(y)\|
		\\ &
		\geq & (1+ \eps')^{-1} \dist (\phi_i(x), \partial \bH \setminus
	 \phi_i( W^*_i) )
		\\ &
		\geq & (1+  \eps')^{-2} \dist(x, \partial_{\partial A}
		(W_i^*\cap \partial A))
	 \geq 50r_t,
	 \eean
		and hence
		\eqref{0708a}.

		Next, for $i \in [m]$ let $W_{i,t} := W_i^* \setminus 
		\Delta_t^{\rm bdy}$.
	 Then for large enough $t$ we have 
	 \bea
	 \partial A
	 \subset \cup_{i=1}^m W^*_i \subset \cup_{i=1}^m W_{i,t} \cup 
		\Delta_t^{\rm bdy}.
	 \label{0709b}
	 \eea
	 
	  Let $i \in [m]$. 
		Since $B \in \cB$ and $\overline{W_i^*} \in \cB$,
		if $B \cap \overline{W^*_i} \cap \partial A \neq \emptyset $
	then
		$B \cap \overline{W_i^*} \in \cB$.
		By (\ref{eqHausd}) from Proposition \ref{propcompare},
		applied to $B \cap \overline{W_i^*}$ rather
		than to $B$ itself,
	  we have
	 \bea
	 \lim_{t \to \infty}
	 \Pr[B \cap \overline{W_{i}^*} \cap \partial  A  \subset 
	 F_{r_t}(\Po_t)]
		= \exp(- c_{d,k} \tv( B \cap W_i^* \cap \partial A)e^{-\zeta}),
		\label{0915a}
	 \eea
	 and
	 \bea
	 \lim_{t \to \infty}
	 \Pr[\{B \cap \overline{W_{i}^*} \cap \partial  A  \subset F_{r_t}
	 (\Po_t)
	 \} 
	 \setminus
	 \{B \cap \overline{W_{i}^*} \setminus
	 A^{(r_t)}  \subset F_{r_t}
	 (\Po_t)\} ] = 0.
		\label{0915b}
	 \eea
	 Note that \eqref{0915a} and \eqref{0915b} still hold even
		if $B \cap \overline{W^*_i} \cap \partial A = \emptyset.$
	 Using (\ref{0915a}) 
	 and
	 Lemma \ref{lemDelta} we may deduce that
	 \begin{align*}
	 \lim_{t \to \infty}
		 \Pr[B \cap W_{i,t} \cap \partial  A  \subset F_{r_t}
		 (\Po_t)]
		= \exp(- c_{d,k} \tv( B \cap W_i^* \cap \partial A)e^{-\zeta}).
		 \end{align*}

	 Now suppose $i,j \in [m]$ with $i \neq j$ and $t >0$.
	 Let $x \in W_{i,t} \cap \partial  A ,
	 y \in W_{j,t} \cap \partial A $.
	 Since $\partial_{\partial A}( W_i^* \cap \partial A)
	 \subset \cup_{\ell =1}^m \partial_{\partial A}( W_\ell \cap
	 \partial A)$,
	  and $x \notin \Delta_t^{\rm bdy} $, we have 
	  \begin{align*}
		  \dist(x, \partial_{\partial A} (W_i^* \cap \partial A))
		  \geq  \dist(x,
	  \cup_{\ell =1}^m \partial_{\partial A}( W_\ell \cap
		  \partial A))
	  \geq 99 r_t.
	  \end{align*}
	  Thus $x \in (W_i^* \cap \partial A)^{(90r_t)}$.
	  Similarly $y \in (W_j^* \cap \partial A)^{(90r_t)}$. 
	   Therefore by (\ref{0708a}), $\dist(x,y) \geq 50r_t$.

	   This shows that the events $\{(B \cap W_{i,t} \cap \partial A)
	   \subset F_{r_t}
	   (\Po_t)\},
	   1 \leq i \leq m$ are mutually independent. Combined with
	   (\ref{0709b}) and 
	   Lemma \ref{lemDelta}, this shows that
	   \bea
	   \lim_{t \to \infty} \Pr[B \cap \partial A  \subset
	   F_{r_t}
	   (\Po_t) ] = \lim_{t \to \infty}
	   \Pr[ \cap_{i=1}^m \{ B \cap W_{i,t} \cap \partial A 
	   \subset F_{r_t}
	   (\Po_t) \} ]
	   \nonumber \\
	   = \lim_{t \to \infty} \prod_{i=1}^m \Pr [B \cap
		   W_{i,t} \cap \partial A 
		   \subset F_{r_t}
		   (\Po_t)  ]
    \nonumber \\
    = \exp ( - c_{d,k} \tv(B \cap \partial A)e^{-\zeta}).
    \label{0718g}
    \eea

	 Since the sets $W_1,\ldots,W_m$ provide an open cover
	 (in $\cM$) of $\partial A$, by a compactness argument
	 there exists $r_0 >0$ such that $\cup_{x \in \partial A}B(x,r_0)
	 \subset \cup_{i=1}^m W_i
	 \subset \cup_{i=1}^m \overline{W_i^*}$.
	 Hence for large enough $n$
	 we have that $A \setminus A^{(r_t)} \subset
	 \cup_{i=1}^m \overline{W_i^*}$, and hence the event inclusion
	 \begin{align*}
	 \{B \cap \partial A & \subset F_{r_t}
	 (\Po_t) \} \setminus
	 \{ B \setminus A^{(r_t)} \subset F_{r_t}
	 (\Po_t)\}
	 \\
	 & \subset \cup_{i=1}^m( \{ B \cap \overline{W_i^*} \cap \partial A \subset F_{r_t}
	 (\Po_t) \}
	 \setminus \{ B \cap \overline{W_i^*}  \setminus A^{(r_t)} \subset 
	 F_{r_t}
	 (\Po_t) \} ),
	 \end{align*}
	 and hence by (\ref{0915b}), 
	 $\Pr[ \{ B \cap \partial A \subset F_{r_t}
	 (\Po_t) \} \setminus
	 \{  B \setminus A^{(r_t)} \subset F_{r_t}
	 (\Po_t)\}] \to 0$ as $t \to \infty$.
Combined with (\ref{0718g}), this shows that (\ref{0719a}) holds.
\end{proof}

To complete the proof of Theorem \ref{th:weak}, we also
need to deal with the possibility that the interior
region $B \cap A^{(r_t)}$ is not covered. Recall the definition
of the secondary coverage thresholds  $R'_{t,k}$ 
and $\tilde{R}_{Z_t,k}$ at (\ref{Rdashdef}) and
(\ref{eqmaxspacPo}) respectively.


\begin{proof}[Proof of Theorem \ref{th:weak}]
	Let $(r_t)_{t >0}$
	satisfy (\ref{rt1}), as before. Then

	\bea
	\lim_{t \to \infty} (
	t \vvol_d f_0 r_t^d - \log (t f_0) - (d + k -2) \lglg t )
	= \begin{cases}
		2 \zeta ~~ & {\rm if}~ d=2,~ k=1, 
		\\
		+ \infty ~~ & {\rm otherwise.}
	\end{cases}
	\label{0719b}
	\eea
	Therefore 
	if either $d \geq 3$, or $d=2 $ and $k > 1$,
	by
	(\ref{1228a}), as $t \to \infty$ we have 
	$
	\Pr[ \tR_{Z_t,k} \leq r_t ] \to 1.
	$
	Hence by (\ref{0913d}),
	$\Pr[B \cap A^{(r_t)} \subset F_{r_t}
	(\Po_t)]
	\to 1$, and therefore using Proposition \ref{prop:rind},
	we obtain that
	\begin{align*}
	\Pr[ B \subset F_{r_t}
		(\Po_t)] & = \Pr[
			\{ B \setminus A^{(r_t)} \subset F_{r_t}(\Po_t) \}
		\cap \{ B \cap A^{(r_t)}  \subset F_{r_t}(\Po_t) \} ]
	\\
		& \to \exp(- c_{d,k} \tv( B \cap \partial A)e^{-\zeta}),
	\end{align*}
	and hence the second equality of (\ref{eqmain3}), for $d \geq 3$
	or $d=2, k \geq 2$.

	Now suppose $d=2, k=1$, and recall that $c_2=1$. Then by
	(\ref{0913d}), (\ref{0719b}), 
	and Proposition \ref{Hallthm},
	\bea
	\lim_{t \to \infty} \Pr[B \cap A^{(r_t)} \subset F_{r_t}(\Po_t) ]
	= \lim_{t \to \infty} \Pr[ \tR_{Z_t,k} \leq r_t ]
	= \exp ( -  v(B) e^{-2 \zeta} ).
	\label{0719c}
	\eea
	Also by Proposition \ref{Hallthm}, for any fixed $r>0$
	we have
	\bean
	\limsup_{t \to \infty} \Pr[B \cap A^{(6r_t)} \subset F_{r_t}(\Po_t) ]
	\leq 
	\lim_{t \to \infty} \Pr[B \cap A^{(r)} \subset F_{r_t}(\Po_t) ]
	= \exp(- v(B \cap A^{(r)}) e^{-2 \zeta}),
	\eean
	and since $v(B \cap A^{(r)} ) \to v(B)$ as $r \downarrow 0$
	(because we assume $B \subset A$), 
	using (\ref{0719c}) we may deduce 
	that
	\bea
	\lim_{t \to \infty} \Pr[B \cap A^{(6r_t)} \subset F_{r_t}(\Po_t) ]
	= \exp(- v(B ) e^{-2 \zeta}),
	\label{0719d}
	\eea
	and  
	\bea
	\lim_{t \to \infty} \Pr[
		\{	B \cap A^{(6r_t)} \subset F_{r_t}(\Po_t) \}
		\setminus \{	B \cap A^{(r_t)} \subset F_{r_t}(\Po_t) \}
		]
	= 0.
	\label{0719e}
	\eea
	By (\ref{0719e}) we have
	\begin{align*}
	\lim_{t \to \infty} \Pr[ B  \subset F_{r_t}(\Po_t) ]
		& = \lim_{t \to \infty} \Pr[ (B  \setminus A^{(r_t)}) 
	\cup (B \cap A^{(6r_t)}) \subset
	F_{r_t}(\Po_t)   ]
	\\
		& = 
	 \lim_{t \to \infty} \Pr[ B \setminus A^{(r_t)} \subset F_{r_t}(\Po_t)   ]
	\Pr[ B \cap A^{(6r_t)} \subset F_{r_t}(\Po_t)   ],
	\end{align*}
	and by (\ref{0719a}) and  (\ref{0719d}), this limit
	is equal to $\exp( - v(B ) e^{-2 \zeta} - 
	c_{2,1}\tv(B \cap \partial A)e^{-\zeta}) $. Thus we have
	the second equality of (\ref{eqmain3}), for the case
	$d =2, k=1$ too.

	We may then deduce the first equality of (\ref{eqmain3})
	by using  \cite[Lemma 7.1]{ECover}.
\end{proof}


\end{document}